\documentclass[11pt]{amsart}
\usepackage{amssymb,amsmath,amsthm,amscd,mathrsfs,graphicx,color,multicol,booktabs,arydshln}
\usepackage[cmtip,all,matrix,arrow,tips,curve]{xy}

\usepackage{multirow}

\numberwithin{equation}{section}
\theoremstyle{plain}
\newtheorem{theorem}{Theorem}[section]

\newtheorem{lemma}[theorem]{Lemma}

\newtheorem{proposition}[theorem]{Proposition}
\newtheorem{corollary}[theorem]{Corollary}

\theoremstyle{remark}
\newtheorem{remark}[theorem]{Remark}

\theoremstyle{definition}
\newtheorem{definition}[theorem]{Definition}

\def\Pic{\operatorname{Pic}}

\def\Hom{\operatorname{Hom}}

\newcommand{\bP}{\mathbb{P}}
\newcommand{\bA}{\mathbb{A}}

\newcommand{\bQ}{\mathbb{Q}}
\newcommand{\bZ}{\mathbb{Z}}

\newcommand{\sF}{\mathscr{F}}

\newcommand{\bC}{\mathbb{C}}

\newcommand{\fh}{\mathfrak{h}}

\newcommand{\gp}{\mathfrak{p}}
\newcommand{\sH}{\mathscr{H}}

\newcommand{\calF}{\mathcal{F}}

\newcommand{\calO}{\mathcal{O}}

\newcommand{\Aut}{\mathrm{Aut}}
\newcommand{\Stab}{\mathrm{Stab}}

\newcommand{\Hilb}{\mathrm{Hilb}}
\newcommand{\PGL}{\mathrm{PGL}}
\newcommand{\GL}{\mathrm{GL}}

\newcommand{\SL}{\mathrm{SL}}

\newcommand{\sX}{\mathscr{X}}
\newcommand{\sD}{\mathscr{D}}

\newcommand{\Proj}{\mathrm{Proj}}

\newcommand{\im}{\operatorname{Im}}

\newcommand{\Gr}{\mathrm{Gr}}

\newcommand{\SO}{\mathrm{SO}}

\newcommand{\gquot}{/\!\!/}
 
\def\Reg{\operatorname{Reg}}

\DeclareMathOperator{\PO}{PO}

\newcommand\aff{{\mathbb A}}

\newcommand{\CC}{{\mathbb C}}

\newcommand{\cC}{{\mathscr C}}
\newcommand{\cD}{{\mathscr D}}

\newcommand{\cH}{{\mathscr H}}
\newcommand{\cI}{{\mathscr I}}

\newcommand{\cL}{{\mathscr L}}

\newcommand{\cO}{{\mathscr O}}
\newcommand{\cP}{{\mathscr P}}

\newcommand{\cS}{{\mathscr S}}
\newcommand{\cT}{{\mathscr T}}

\newcommand{\cZ}{{\mathscr Z}}

\newcommand{\dra}{\dashrightarrow}

\newcommand{\es}{\varnothing}
\newcommand{\FF}{{\mathbb F}}

\newcommand{\gm}{\mathfrak{m}}
\newcommand{\gM}{\mathfrak{M}}

\newcommand{\hra}{\hookrightarrow}
\newcommand{\la}{\langle}

\newcommand{\lra}{\longrightarrow}

\newcommand{\n}{\noindent}
\newcommand{\NN}{{\mathbb N}}
\newcommand{\ov}{\overline}
\newcommand{\PP}{{\mathbb P}}
\newcommand{\QQ}{{\mathbb Q}}
\newcommand{\ra}{\rangle}

\newcommand{\wh}{\widehat}
\newcommand{\wt}{\widetilde}
\newcommand{\ZZ}{{\mathbb Z}}

\DeclareMathOperator{\Chow}{Chow}

\DeclareMathOperator{\cod}{cod}

\DeclareMathOperator{\diag}{diag}
\DeclareMathOperator{\divisore}{div}

\DeclareMathOperator{\initial}{in}

\DeclareMathOperator{\mult}{mult}

\DeclareMathOperator{\rk}{rk}

\DeclareMathOperator{\weight}{wt}

\usepackage{ifthen}
\usepackage{rotating}
\usepackage{supertabular}
\newcommand{\cit}[1]{{\rm \textbf{#1}}}
\newcommand{\Ref}[2]{\cit{%
\ifthenelse{\equal{#1}{thm}}{Theorem}{}%
\ifthenelse{\equal{#1}{ass}}{Assumption}{}%
\ifthenelse{\equal{#1}{chp}}{Chapter}{}%
\ifthenelse{\equal{#1}{prp}}{Proposition}{}%
\ifthenelse{\equal{#1}{lmm}}{Lemma}{}%
\ifthenelse{\equal{#1}{cnj}}{Conjecture}{}%
\ifthenelse{\equal{#1}{crl}}{Corollary}{}%
\ifthenelse{\equal{#1}{dfn}}{Definition}{}%
\ifthenelse{\equal{#1}{expl}}{Example}{}%
\ifthenelse{\equal{#1}{hyp}}{Hypothesis}{}%
\ifthenelse{\equal{#1}{rmk}}{Remark}{}%
\ifthenelse{\equal{#1}{clm}}{Claim}{}%
\ifthenelse{\equal{#1}{notconv}}{Notation/Convention}{}%
\ifthenelse{\equal{#1}{pred}}{Prediction}{}%
\ifthenelse{\equal{#1}{prb}}{Problem}{}%
\ifthenelse{\equal{#1}{table}}{Table}{}%
\ifthenelse{\equal{#1}{exe}}{Exercise}{}%
\ifthenelse{\equal{#1}{qst}}{Question}{}%
\ifthenelse{\equal{#1}{rcp}}{Recipe}{}%
\ifthenelse{\equal{#1}{sec}}{Section}{}%
\ifthenelse{\equal{#1}{subsec}}{Subsection}{}%
\ifthenelse{\equal{#1}{subsubsec}}{Subsubsection}{}%
\ifthenelse{\equal{#1}{univ}}{Universal Property}{}%
\ifthenelse{\equal{#1}{trm}}{Terminology}{}%
\ifthenelse{\equal{#1}{tbl}}{Table}{}%
\  \ref{#1:#2}%
}}

\usepackage{fullpage}
  
\begin{document}
 \title[Moduli hyperelliptic quartics]{GIT versus Baily-Borel compactification for $K3$'s which are  double covers of $\PP^1\times\PP^1$}
 
 \author[R. Laza]{Radu Laza}
\address{Stony Brook University,  Stony Brook, NY 11794, USA}
\email{radu.laza@stonybrook.edu}

 \author[K. O'Grady]{Kieran O'Grady}
\address{``Sapienza'' Universit\'a di Roma, Rome, Italy}
\email{ogrady@mat.uniroma1.it}
\date{\today}

\begin{abstract}
In previous work, we have introduced a program aimed at studying the birational  geometry of locally symmetric varieties of Type IV associated to moduli of certain projective varieties of $K3$ type.
In particular, a concrete goal of our program is to understand the relationship between GIT and Baily-Borel compactifications for quartic $K3$ surfaces, $K3$'s which are double covers of a smooth quadric surface, and double EPW sextics.  In our first paper \cite{log1}, based on arithmetic considerations, we have given conjectural decompositions into simple birational transformations of the period maps from the GIT moduli spaces mentioned above  to the corresponding Baily-Borel compactifications. In our second paper \cite{log2} we studied the case of quartic $K3$'s; we have given geometric meaning to this decomposition and we have  partially verified our conjectures. Here, we give a full proof of the conjectures in~\cite{log1} for the moduli space of  $K3$'s which are double covers of a smooth quadric surface. The main new tool here is VGIT for $(2,4)$ complete intersection curves.
\end{abstract}

\thanks{Research of the first author is supported in part by NSF grants DMS-1254812 and DMS-1361143 and a Simons fellowship. Research of the second author is supported in part by PRIN 2015.}
  \maketitle
  
\bibliographystyle{amsalpha}

\section{Introduction}
\subsection{Background and motivation}
Almost 40 years ago, J. Shah~\cite{shah} analyzed the GIT moduli space $\gM_{ps}$ of plane sextic curves. He compared $\gM_{ps}$ to the Baily-Borel compactification $\sF_2^{*}$ of the period space of degree-$2$ polarized $K3$ surfaces, via the period map associating to a sextic $C\subset\PP^2$ (with simple singularities) the primitive Hodge structure on the double cover $X\to\PP^2$ with branch curve $C$. The period map is birational, not regular, with exactly one point of indeterminacy $\omega\in\gM_{ps}$ corresponding to the polystable point $3D$, where  $D\subset\PP^2$ is a smooth conic. The blow up $\wh{\gM}_{ps}\to\gM_{ps}$ of a scheme supported on $\omega$ resolves the indeterminacy of $\gp$, hence there is a regular lift 
$\wh{\gp}\colon \wh{\gM}_{ps}\to\sF^{*}$ of $\gp$. The exceptional divisor of $\wh{\gM}_{ps}\to\gM_{ps}$ is mapped by $\wh{\gp}$ to the closure of the Heegner divisor $H_u\subset\sF$ parametrizing periods of unigonal degree $2$ polarized $K3$'s.

A few years later, Looijenga~\cite{looijengavancouver} revisited Shah's work from a different point of view. Looijenga started from the \lq\lq other end\rq\rq, i.e.~$\sF_2^{*}$, and noticed that the closure of $H_u$ is not $\QQ$-Cartier (notice that $H_u$ itself is $\QQ$-Cartier, in fact $\sF_2$ is $\QQ$-factorial). 
The small birational map $ \wh{\gM}_{ps}\to\sF^{*}$ was constructed by Looijenga as the $\QQ$-Cartierization of $H_u^{*}$, and then 
 $\wh{\gM}_{ps}\to\gM_{ps}$ as the contraction of the strict transform of $H^{*}_u$. Later Looijenga~\cite{looijenga1,looijengacompact} developed analogous ideas for pairs $(\sF,\sH)$, where $\sF$ is a locally symmetric of ball type or of Type IV, and $\sH$ is an effective  Heegner divisor. He constructed the so-called semitoric compactification $(\sF\setminus\sH)\subset\sF^L$. Very often the complement of  $\sF\setminus\sH$ in 
 $\sF^L$ has codimension at least 2; this suggests that in cases where $\sF$ is the period space for a certain class of projective varieties, 
 $\sF^L$ might be equal to a GIT quotient (as in the case of plane sextics). 
 
 One case in which the above prediction works perfectly is that of cubic $4$-folds. Laza~\cite{gitcubic,cubic4fold} and Looijenga~\cite{lcubic} proved that the GIT moduli space, which may be viewed as a projective birational model of the moduli space of hyperk\"ahler $4$-folds of Type $K3^{[2]}$ with a polarization of square $6$ and divisibility $2$, is isomorphic to the semi-toric compactification determined by the \lq\lq unigonal\rq\rq\ Heegner divisor. 
 
 In general, the complexity of the semi-toric compactification $\sF^L$ (i.e.~the number of elementary birational modifications necessary to go from
  $\sF^{L}$ to the Baily-Borel compactification $\sF^{*}$) is essentially equal to the complexity of the  hyperplane arrangement $\wt{\sH}$ in the bounded symmetric domain $\sD$ (here $\sF=\Gamma\backslash\sD$, where $\Gamma$ is an arithmetic group)   given by preimages of $\sH$. The extreme simplicity of the period map of plane sextics is explained by the fact that no two irreducible components of $\wt{\sH}$ meet, and the period map for cubic $4$-folds is relatively simple because  at most two irreducible components of $\wt{\sH}$ meet. 
 
 By way of contrast, the hyperplane arrangement  naturally associated to the period map  of quartic surfaces (notice that the generic degree-$4$ polarized $K3$ surfaces is a quartic surface) is as complex as it possibly can be, i.e.~there are linearly independent hyperplanes with non-empty intersection of any cardinality up to $19$ (the dimension of the period space).
On the other hand, the period map from the GIT moduli space of quartic surfaces  to the relevant Baily-Borel compactification has an intricate structure, but not as intricate as predicted by a direct implementation of Looijenga's philosophy.
In~\cite{log1} we carried out a thorough study of the hyperplane arrangement relevant to the period map for quartic surfaces, and similar arrangements in arbitrary dimension. Beyond that of quartic surfaces, our conjectures cover period maps for GIT moduli spaces  of $(4,4)$-curves on $\PP^1\times\PP^1$ and of EPW-sextics modulo the duality involution (or equivalently EPW cubes). Motivated by the Hassett-Keel program for the moduli space of curves, we reformulated the problem as   that of studying the variation of birational model corresponding to the total ring of sections of $\lambda+\beta\Delta$ on the relevant locally symmetric space $\sF$, where $\lambda$ is the Hodge ($\QQ$) divisor, and $\Delta$ is a suitable effective $\QQ$-divisor with the property that for $\beta=1$ we get the GIT moduli space (in the three cases mentioned above). We formulated precise conjectures on the behaviour of the models $\sF(\beta)$, i.e.~of the birational map 
$\sF(\beta)\dra\sF^{*}=\sF(0)$. We found that Borcherds' celebrated relation explains the mismatch between the naive version of Looijenga's philosophy and the behaviour of the period map for quartic surfaces. 

In~\cite{log2} we gave strong evidence (of geometric nature) in favor of the correctness of our conjectures regarding the period map for quartic surfaces. 

In the present paper we will completely verify our conjectures on the period map for $(4,4)$ curves on $\PP^1\times\PP^1$. We emphasize that the complexity of the period map in this case is only one degree less than that of quartic surfaces, and that our Hassett-Keel-Looijenga (HKL) program for the $D_n$ series has a strong inductive structure. Thus we view our results as compelling evidence in favour of the validity of our conjectures.
\subsection{The main result}
We start by introducing the main actors. Let
\begin{equation}\label{eccoemme}
\gM:=|\cO_{\PP^1\times\PP^1}(4,4)| \gquot\Aut(\PP^1\times\PP^1)
\end{equation}
be the GIT moduli space of $(4,4)$ curves on $\bP^1\times \bP^1$.  Let $C$ be a $(4,4)$ curve with simple singularities (it is stable by 
Shah~\cite[Sect.~4]{shah4}), and let $\pi\colon X_C\to \PP^1\times\PP^1$ be the double cover with branch curve $C$. Then $X_C$ is a $K3$ surface (eventually with canonical singularities). Moreover $\pi^{*}(\cO_{\PP^1}(1)\boxtimes\cO_{\PP^1}(1))$ is a polarization of $X_C$, and with this polarization, $X_C$ is a hyperelliptic polarized $K3$ of degree $4$. The corresponding period space, which we denote by $\sF$ (see~\Ref{subsec}{predictqh} for the definition) is an $18$ dimensional locally symmetric variety. We let $\sF\subset\sF^{*}$ be the Baily-Borel compactification. By associating to a generic $[C]\in\gM$ the    Hodge structure on $H^2(X_C)$ (more precisely, on the orthogonal of $\pi^{*} H^2(\PP^1\times\PP^1)$), we get the rational period map
\begin{equation}\label{eccoper2}
\gp\colon\gM\dashrightarrow \sF^*.
\end{equation}
By Global Torelli, $\gp$ is birational. By Baily-Borel, $\sF^{*}$ is identified with $\Proj R(\sF, \lambda)$, where $\lambda$ is the Hodge ($\QQ$-) line bundle on $\sF$.  We proved that also $\gM$ is identified with $\Proj$ of a ring of sections of a $\QQ$-Cartier divisor on $\sF$. In order to recall the result, we need to introduce another   key actor, namely the Heegner divisor $H_h\subset\sF$ parametrizing periods of $K3$ surfaces which are double covers of a \emph{quadric cone}. Let $\Reg(\gp)\subset\gM$ be the regular locus of $\gp$. Then $\gp(\Reg(\gp))\cap\sF$ contains  $(\sF\setminus H_h)$ - in fact it follows from our main result that they are equal. Our \emph{boundary divisor} is $\Delta=H_h/2$. We proved that $\gM$ is identified with $\Proj R(\sF, \lambda+\Delta)$ - see Proposition~4.0.20 and  Equation~(4.0.21) in~\cite{log1}. Our conjectures are concerned with the behaviour of the graded $\CC$-algebra 
$R(\sF, \lambda+\beta\Delta)$ for $\beta\in[0,1]\cap\QQ$. First, we conjecture that it is finitely generated. Secondly we predict the critical values of $\beta$, i.e.~the Mori chamber decomposition of the sector $\{\lambda+\beta\Delta\}_{\beta\in[0,1]\cap\QQ}$. Lastly, we describe the centers of the corresponding flips or contractions. This last part of the conjecture is formulated in terms of towers
\begin{equation}\label{zetahyp1}
Z^8\subset Z^7\subset Z^6\subset Z^4\subset Z^3\subset Z^2\subset Z^1\subset \sF
\end{equation}
and 
\begin{equation}\label{wstrata1}
W_0\subset W_1\subset W_2\subset W_3 \subset W_5 \subset W_6 \subset W_7=\gM^{IV}\subset \gM.
\end{equation}
(The \lq\lq missing\rq\rq\ indexes are not misprints.)
The superscripts in~\eqref{zetahyp1} denote codimension (in $\sF$), while those in~\eqref{wstrata1} denote dimension.
The  definition of the stratification~\eqref{zetahyp1} is  in~\Ref{subsec}{predictqh}, that of  the stratification~\eqref{wstrata1} is  in~\Ref{dfn}{eccow}. Regarding $Z^{\bullet}$, it will suffice to recall that $Z^1=H_h$, and a general feature of the $Z^k$'s. First, recall that $\sF=\Gamma\backslash\sD^{+}$, where $\sD^{+}$ is a bounded symmetric domain of Type IV, and $\Gamma$ is an arithmetic group. Letting $\rho\colon\sD^{+}\to\sF$ be the quotient map, the inverse image $\rho^{-1}H_h$ breaks up into the union of an infinite collection of  hyperplane sections $\cH_j\subset\sD^{+}$ ($\sD^{+}$ is embedded as an open dense subset of a $18$-dimensional quadric in $\PP^{19}$).
The feature of  the $Z^k$'s that we wish to point out is that they are defined as images of loci defined by suitable \lq\lq Noether-Lefschetz\rq\rq\ conditions, in fact in most cases  we simply require that the  points in $\cD^{+}$ are contained in  (at least) $k$ linearly independent  hyperplanes $\cH_j$.
Regarding $W_{\bullet}$, it will suffice to recall that $\gM^{IV}$ is the locus parametrizing polystable $(4,4)$ curves $C$ such that the corresponding double cover $X_C\to\PP^1\times\PP^1$ has non-slc singularities (i.e.~significant limit singularities in the notation of Mumford and Shah). In particular 
$\gM\setminus\gM^{IV}$ is contained in the regular locus of $\gp$ - in fact it follows from our main result that they are equal. 

\medskip

We are ready to state our main result.
 \begin{theorem} \label{thm:mainthm1}
With notation as above, the following hold:
 \begin{itemize}
 \item[i)]   
 Let  $\beta\in[0,1]\cap\bQ$. The ring of sections $R(\sF, \lambda+\beta\Delta)$ is a finitely generated $\CC$-algebra, and hence $\sF(\beta):=\Proj R(\sF, \lambda+\beta\Delta)$ is a projective variety interpolating between $\sF^{*}=\sF(0)$ and $\gM=\sF(1)$. 
  \item[ii)] 
  The variation of models $\sF(\beta)$ on the interval $[0,1]\cap\QQ$ has a Mori chamber decomposition whose set of critical values is
  $$\left\{0,\frac{1}{8},\frac{1}{6}, \frac{1}{5}, \frac{1}{4},\frac{1}{3}, \frac{1}{2},1\right\}.$$
Hence for consecutive critical values $\beta'>\beta''$ it makes sense to let $\sF(\beta',\beta''):=\sF(\beta)$, where $\beta'>\beta>\beta''$ is arbitary.
  \item[iii)] 
 The  period map $\gp\colon \gM=\sF(1)  \dra \sF(0)=\sF^*$ is the composition of the elementary birational maps in~\eqref{flipper}. 
\begin{equation}\label{flipper}
\xymatrix @R=.07in @C=.01in{
\scriptstyle  & & \scriptstyle \sF ( 1 , \frac{1}{2} ) \ar@{<-->}[rr]\ar[ldd] \ar[rdd] & & \scriptstyle \cdots \ar[ldd] & \scriptstyle  \sF (\beta_{k-1}, \beta_k) \ar[rdd]\ar@{<-->}[rr] & &  
\scriptstyle  \sF ( \beta_k , \beta_{k+1}) \ar[ldd] &\scriptstyle \cdots \ar[rdd]   & & \scriptstyle \sF (\frac{1}{8},0)\ar@{<-->}[ll]\ar[ldd] \ar[rrdd]^{\scriptstyle \Phi}  \\
&&&&&&\\
& \scriptstyle \gM=\sF(1) & & \scriptstyle \sF( \frac{1}{2} ) & & & \scriptstyle \sF( \beta_k )  & & &  \scriptstyle \sF(\frac{1}{8})  & & & \scriptstyle \sF(0)=\sF^{*}  \\
 }
 \end{equation}
Here the critical values of $\beta$ are indexed as in~\eqref{rosetta}.
\begin{equation}\label{rosetta}
\begin{tabular}{c|c|c|c|c|c|c|c | c | c}
$k$  & $0$ & $1$  & $2$  & $3$  &  $4$  &  $5$  & $6$ & $7$ & $8$ \\
\hline
$\beta_k$ &  $1$ & $1/2$ &  $1/3$  & $1/4$ &  $1/4$ &   $1/5$ &   $1/6$ &  $1/8$ & $0$ 
\end{tabular}
\end{equation}
(The equality $\beta_3=\beta_4$ is not a misprint.) Let $\Omega_{-}(\beta_k)\subset \sF(\beta_{k-1},\beta_k)$ 
 and $\Omega_{+}(\beta_k)\subset \sF(\beta_k,\beta_{k+1})$ be the exceptional loci of 
 $\sF(\beta_{k-1},\beta_k)\to \sF(\beta_k)$ and  $\sF(\beta_{k},\beta_{k+1})\to \sF(\beta_k)$ respectively. Then $\Omega_{-}(\beta_k)$ 
  is the strict transform of $W_k\subset \sF(1)=\gM$ for the birational map $\gM\dra \sF(\beta_{k-1},\beta_k)$ and $\Omega_{+}(\beta_k)$   is the strict transform of $Z^{k+1}\subset\sF(0)= \sF^*$ if $k\neq 4$, and of $Z^4\subset \sF^*$ if $k=4$,  for the birational map $\sF^{*}\dra \sF(\beta_{k},\beta_{k+1})$. 
  \item[iv)] 
 The map $\sF(1/8,0)\to\sF^{*}$ is the $\QQ$-Cartierization associated to $H_h$. Moreover $\sF(1/8,0)$ is a moduli space of double covers of  quadrics (possibly singular) in $\PP^3$ \emph{with slc singularities}.
 \end{itemize}
 \end{theorem}
 Summarizing: Items~(i), (ii), (iii) and the first part of Item~(iv) prove that our conjectures in~\cite{log1} hold for the period space of hyperelliptic $K3$ surfaces of degree $4$, while the second part of Item~(iv) is a \lq\lq bonus\rq\rq\ result which we find very interesting. Namely, it says roughly that  $\sF(1/8,0)$ is a KSBA-like compactification for hyperelliptic quartic $K3$s and that this compactification is nothing but a small partial resolution of the Baily-Borel compactification $\sF^*$. (We refer to \cite{shahinsignificant} and \cite{kk} for some discussion of KSBA versus Hodge theoretic degenerations.) 
 
\smallskip

In addition to the results and techniques of our previous work (\cite{log1,log2}), the main new tool that allows us to prove the theorem above is VGIT for $(2,4)$ complete intersection curves $C$ in $\bP^3$ (obviously, such a curve determines a hyperelliptic quartic $X_C$ or degeneration of it). A similar case, namely $(2,3)$ complete intersections, was analyzed by the first named author and his collaborators in \cite{g4git, g4ball} in the context of the Hassett-Keel program for genus $4$ curves. Here, we follow in rough outline the strategy from loc.~cit., but as the complexity increases, some streamlining and  new ideas (such as the basin of attraction arguments in \Ref{sec}{sec-VGIT}) are necessary.

\begin{remark}
Arguably, the case of quartic $K3$ surfaces would have been of greater geometric interest, but from the perspective of the HKL program it has almost the same  complexity. More precisely, in \cite{log1}, 
we have introduced an  inductive structure on the moduli spaces considered here (quartic $K3$, hyperelliptic quartic $K3$s, etc.). We expect then (both for geometric and arithmetic arithmetic reasons) 
that  \Ref{thm}{mainthm1} for hyperelliptic quartics to imply an exact analogue for quartics. Furthermore, we expect to be able to bootstrap this and understand the relationship between GIT and periods for EPW sextics (\cite{epwperiods,epw}); this is the next step in our inductive structure. The study of the moduli of EPW sextics was our original motivation for this investigation, which in turn led us to the development of a general HKL program. We further remark a non-inductive proof for an analogue of \Ref{thm}{mainthm1}  for quartic $K3$s should be possible by using GIT/VGIT techniques similar to those used in this paper. (We thank O. Benoist for a suggestion that allows  to reduce the HKL program for quartics to a feasible GIT computation.) 
\end{remark}

\begin{remark}
One possible application of  \Ref{thm}{mainthm1} is the computation of the cohomology of the Baily-Borel compactification $\sF^*$. Specifically, Kirwan's techniques allow one to compute the cohomology of $\gM$ (see \cite{kirwanhyp} for the case of quartics in $\bP^3$), while \eqref{flipper} gives a simple wall crossing decomposition which allows one to compute (in a standard way by now) the cohomology of $\sF^*$. The considerably simpler case of degree $2$ $K3$ surfaces (where no flip occurs) was studied by Kirwan-Lee \cite{kirwanlee,kirwanlee2}.
\end{remark}

%

%
\subsection{Structure of the paper}
We start our paper with a review of the basic facts about the moduli space of hyperelliptic quartic $K3$ surfaces. In \Ref{sec}{preliminaries}, we discuss the period space $\sF=\cD/\Gamma$ for hyperelliptic polarized $K3$ surfaces of degree $4$.  We then recall the definition of the (Shimura type) loci $Z^k\subset \sF$ that were identified in \cite{log1} as conjectural centers of the birational transformations occurring in the variation of models $\sF(\beta)$. We also describe the loci $Z^k$ as the periods of double covers $X\to Q$ such that $Q$ is a quadric cone, and 
the  branch curve has a certain singularity at the vertex of $Q$ (see~\Ref{prp}{hypkgeom}) -- roughly, the higher the $k$ (corresponding to the codimension), the worse the singularity of $C$ at the vertex of $Q$. We end the section  with a discussion of the boundary components of the Baily-Borel compactification $\sF^*$.  The analogous case of polarized $K3$'s of degree $4$ was first investigated by Scattone~\cite{scattone}. For inductive purposes, it is convenient to discuss uniformly the Baily-Borel compactification for the entire $D$-tower; this is done in the appendix \Ref{sec}{bbcompact}.

In \Ref{sec}{quattro}, we discuss the other end of the period map, namely the GIT quotient $\gM$
for $(4,4)$ curves on $\bP^1\times \bP^1$. This was first studied by Shah \cite[Sect. 4]{shah4}, but we caution the reader that some of the results of 
loc.~cit. are incomplete. We then discuss the  Hodge-theoretic  stratification of $ \gM$ analogous to that of the GIT moduli space of quartics analyzed in~\cite[Sect. 3]{log2} (building on ideas of Shah \cite{shahinsignificant,shah,shah4}).
In particular, we define the tower in~\eqref{wstrata1}.   

\smallskip

In order to understand the variation of models $\sF(\beta)$, we introduce in \Ref{sec}{shyper} a  VGIT $\gM(t)$ for $t\in(1/6-\epsilon,1/2]\cap\QQ$ which interpolates between the GIT quotient $\gM$ for $(4,4)$ curves on $\bP^1\times\bP^1$ and the GIT quotient $\Chow_{(2,4)}\gquot \SL(4)$ of the Chow variety of $(2,4)$ complete intersection curves in $\bP^3$. The point is that in~\Ref{sec}{proof} we will prove that each $\sF(\beta)$ is isomorphic to $\gM(t(\beta))$, where $t(\beta)$ is given by~\eqref{tidibeta}.  
 The precise definition of $\gM(t)$ is given in \eqref{defgmt}, and  is mostly based  on  ideas  in~\cite{g4git},  where the analogous case of $(2,3)$ curves was discussed. A key point (\Ref{prp}{propssci}) that comes up in our analysis is that only complete intersections $V(f_2,f_4)$ are relevant for the GIT analysis of $\gM(t)$ (and thus, we do not have to deal with complicated schemes). We also note that for an infinite set of values of $t_m$, the moduli space  $\gM(t)$ can be identified with  GIT quotients $\Hilb^m_{(2,4)}\gquot \SL(4)$, where $\Hilb^m_{(2,4)}$ denotes the parameter space for $m$-th Hilbert point of a $(2,4)$ complete intersection. In particular, one sees that the VGIT map $\gM(\frac{1}{2}-\epsilon)\to \gM(\frac{1}{2})\cong \Chow_{(2,4)}\gquot \SL(4)$ is induced by the Hilbert-to-Chow morphism $\Hilb_{(2,4)}\to \Chow_{(2,4)}$. 

\smallskip

The actual GIT analysis for $\gM(t)$  is accomplished in \Ref{sec}{sec-VGIT}. Although the tools that we use are by now standard (we acknowledge the influence of \cite{g4git} and \cite{benoist} that studied a similar set-up to ours, and  \cite{hh2} and \cite{ah} that focus on the relationship between GIT and Hassett-Keel program in general),  the proof 
 is a somewhat indirect, multi-step argument. 
  First, via  the numerical criterion, we can destabilize various ``bad'' $(2,4)$ complete intersection curves. In particular,  curves that are relevant to the change of stability are contained in  a quadric cone.  Since a wall crossing in VGIT involves orbits with positive dimensional stabilizer,   it is not hard to identify a finite list $\{t_k\}\subset(0,\frac{1}{2})$ of  potential critical slopes (or walls), and associated (potential) critical orbits $\{C_k\}$. To check that these potential critical slopes/orbits actually occur, we proceed in two further steps. On one hand, via the numerical criterion, we see that the generic curve in $W_k$ will be destabilized for some $t\le t_k$. Conversely,  an analysis of the basin of attraction for the potential semistable curves $C_k$, shows that the generic curve in $W_k$ can not be destabilized before $t_k$. Thus, the stratum $W_k$ will change from stable to unstable precisely at $t_k$. Furthermore, it is not hard to see that the replacement stratum is (birational to) $Z^{k+1}$ (here we use the geometric characterization of the $Z$-strata given by \Ref{prp}{hypkgeom}). This concludes the GIT analysis of the quotients $\gM(t)$ (we refer the reader to Table \ref{table101} for a quick summary). A key feature of our proof is the basin of attraction argument. Once we know (semi)stability for $\gM=\gM(1/6)$, we can determine which curves are parametrized by $\gM(t)$ for $1/6<t\le 1/2$ via a careful analysis of the change in stability at the critical $t$'s. In particular one gets that 
  $\gM(\frac{1}{2}-\epsilon)$ parametrizes double covers of irreducible cones with slc singularities.

\smallskip

In~\Ref{sec}{proof} we  return to the period space $\sF$ and the variation of models $\sF(\beta)$. First, we note that $\sF$ and $\gM(t)$ for $t\in(1/6,1/2)$ are isomorphic in codimension $1$ because, by the results in~\Ref{sec}{shyper}, a $(2,4)$ curve with at worst a node on an irreducible quadric (and not passing through the vertex if the quadric is singular) is GIT $t$-stable  for $t\in(1/6,1/2)$.  Next, let
\begin{equation}\label{tidibeta}
t(\beta)=\frac{1}{4\beta+2}.
\end{equation}
and consider the composition
\begin{equation}\label{betaper0}
 \gM(t(\beta))\overset{\gp(t(\beta))}{\dashrightarrow} \sF\dashrightarrow \sF(\beta),
\end{equation}
where $\gp(t(\beta))$ is the period map. We prove that the above map matches $\lambda+\beta\Delta$ (ample on $\sF(\beta)$ by definition) and the natural ample GIT polarization of $ \gM(t(\beta))$, up to a positive factor.
It follows (by the codimension $1$ result) that we 
we have an isomorphism
\begin{equation}\label{identifygmf}
\gM(t(\beta))\cong\sF(\beta)
\end{equation}
and that the ring of sections $R(\sF,\lambda+\beta\Delta)$ (for $\beta\in[0,1]\cap \bQ$) is a finitely generated $\CC$-algebra. 

The remaining statements  of \Ref{thm}{mainthm1} follow from the isomorphism in~\eqref{identifygmf} and 
 the VGIT analysis of \Ref{sec}{sec-VGIT}. This concludes the proof (and gives a complete verification of the conjectures in~\cite{log1} for polarized hyperelliptic $K3$'s of degree $4$). As an interesting geometric fact, we note that 
specializing \eqref{identifygmf} to the case $\beta=0$ (i.e.~$t=\frac{1}{2}$), we get: 
 \begin{corollary}\label{crl:corbbchow}
The Baily-Borel compactification $\sF^*$ for hyperelliptic quartics is isomorphic to  the GIT quotient $\Chow_{(2,4)}\gquot \SL(4)$ for the Chow variety of $(2,4)$ curves in $\bP^3$.
 \end{corollary}

We were able to do the complete analysis for the hyperelliptic quartic $K3$'s because geometrically everything reduces to curves. Essentially, our paper is almost an instance of the Hassett--Keel program (quite similar to \cite{g4git,g4ball}, which is the Hassett-Keel program in genus $4$). However, this is a misleading point of view. We argue that the results in this paper have much more structure than those in the simpler geometric case discussed in \cite{g4git,g4ball} (for instance compare \Ref{crl}{corbbchow} to \cite{g4ball}). More generally, we believe that the Hassett--Keel--Looijenga program (for $K3$ type) has much more structure and  is more regular than the Hassett--Keel program (for curves). The main point (that we call {\it Looijenga's vision}) is that in the HKL setup  everything, while intricate, is controlled very regularly by arithmetic. To drive this point home, we analyze in \Ref{sec}{chow}, the GIT for Chow quotient and Hilbert scheme (for large Hilbert point) for $(2,4)$ curves. As noted in  \Ref{crl}{corbbchow}, we have $\Chow_{(2,4)}\gquot \SL(4)\cong \sF^*$. Similarly, for $m\gg0$, $\Hilb^{m}_{(2,4)}\gquot \SL(4)\cong \widehat \sF(\cong \sF(1-\epsilon))$, where $\widehat \sF$ is the Looijenga $\bQ$-Cartierization (\cite{looijengacompact}) associated to the divisor $\Delta$ (which fails to be $\bQ$-Cartier in $\sF^*$). Furthermore, the Hilbert-to-Chow morphism induces the $\bQ$-factorialization map $\widehat \sF\to \sF^*$. The structure of $\sF^*$, $\widehat \sF$, and of the map $\widehat \sF\to \sF^*$ are purely arithmetic (this is the essential content of \cite{looijengacompact}).  While it is a consequence of our results, that geometry of the variation of the GIT quotients matches the arithmetic variation, it is very striking to see explicitly how this actually happens. Essentially, various geometric phenomena magically align to give a result such as \Ref{crl}{corbbchow} (e.g. see \Ref{thm}{thmchow} and its proof).  We simply reiterate our belief that the GIT analysis of this paper would not have be possible without the prior knowledge of the expected structure which was given by the arithmetic computations of \cite{log1}. 

\section{The period space and its Baily-Borel compactification}\label{sec:preliminaries}
\subsection{Periods of hyperelliptic polarized $K3$'s of degree $4$ according to~\cite{log1}}\label{subsec:predictqh}
Let $\Lambda_N$ be the lattice $U^2\oplus D_{N-2}$, where $U$ is the hyperbolic plane, $D_{N}$ is the negative definite lattice corresponding to the  Dynkin diagram $D_{N}$ (for $N\ge 3$, and where $D_1$ is the rank $1$ lattice with generator of square $(-4)$), and 
$\oplus$ will always mean \emph{orthogonal}  direct sum. Let
\begin{equation*}
\cD(N):=\{[\sigma] \in\PP(\Lambda_{N}\otimes\CC)\mid \sigma^2=0,\ (\sigma+\ov{\sigma})^2>0\}.
\end{equation*}
Then $\cD(N)$ is a complex manifold of dimension $N$, and it has two connected components, interchanged by complex conjugation; let $\cD^{+}(N)$ be one of the two connected components. We note that  $\cD^{+}(N)$ is a Type IV bounded symmetric domain. Let $O^{+}(\Lambda_N)< O(\Lambda_N)$ be the index two subgroup mapping $\cD^{+}(N)$ to itself. In~\cite{log1} we have studied the locally symmetric variety
\begin{equation*}
\sF(N):=\Gamma(N)\backslash\cD^{+}(N)
\end{equation*}
where $\Gamma(N)=O^{+}(\Lambda_N)$ if $n\not\equiv 6\pmod{8}$, and $\Gamma(N)$ is a subgroup of index $3$ in $O^{+}(\Lambda_N)$  if $n\equiv 6\pmod{8}$, see Prop.~1.2.3 ibid.
We let $\sF(N)\subset\sF(N)^{*}$ be  the Baily-Borel compactification. 

 The period space for hyperelliptic polarized $K3$'s of degree $4$ is $\sF(18)$, see Propositions~2.2.1 and~1.4.5 ibid. In order to simplify notation, we let
 \begin{equation}
\Lambda:=\Lambda_{18},\quad \cD^{+}:=\cD^{+}(18),\quad \sF:=\sF(18),\quad \sF^{*}:=\sF(18)^{*}.
\end{equation}
We recall that  $w\in\Lambda$ is a \emph{hyperelliptic vector} if $w^2=-4$, and $\divisore(w)=2$ (see Definition~1.3.4 and Remark~1.1.3  ibid.), where $\divisore(w)$ is the \emph{divisibility} of $w$, i.e.~the positive generator of 
$(w,\Lambda)$. 
The Heegner divisor $H_h\subset\sF$ is the locus of $O^{+}(\Lambda)$-equivalence classes of points $[\sigma]\in\cD^{+}_{\Lambda}$ such that $\sigma^{\bot}$  contains a hyperelliptic vector. The boundary divisor for $\sF$ is given by $\Delta:=H_h/2$, see Definition~1.3.7 in~\cite{log1}.
 
Given a smooth  $C\in|\cO_{\PP^1}(4)\boxtimes\cO_{C}(4)|$, let  $\pi\colon X_C\to\PP^1\times\PP^1$ be the double cover ramified over $C$, and let $L_C:=\pi^{*}\cO_{\PP^1}(1)\boxtimes\cO_{\PP^1}(1)$. Then $(X_C,L_C)$ is a hyperelliptic polarized $K3$ of degree $4$. 
   By associating to $C$ the period point of $(X_C,L_C)$ we get a period map $\gp\colon \gM\dra\sF^{*}$ (Shah proved that smooth $(4,4)$ curves are stable, see Theorem~4.8 in~\cite{shah}, or~\Ref{prp}{git44}).  By Global Torelli, $\gp$ is birational. The intersection of  $\sF$ and the image  of the regular locus of $\gp$  is equal to $\sF\setminus H_h$. 
The indeterminacy locus of $\gp$ is a  subset of $\gM$ of dimension $7$, with an intricate Hodge-theoretic stratification. 
In~\cite{log1}  we have formulated precise predictions on the decomposition of $\gp$ as a composition of elementary birational maps. In particular  we defined a tower of closed subsets
\begin{equation}\label{zetahyp}
Z^8\subset Z^7\subset Z^6\subset Z^4\subset Z^3\subset Z^2\subset Z^1\subset \sF
\end{equation}
(the superscript denotes codimension), and we motivated the expectation that the centers of the elementary birational maps are birational to the $Z^k$'s. In agreement with Looijenga's vision, most of the $Z^k$'s are the images in $\sF$ of the locus of points in $\cD^{+}_{\Lambda}$ which are orthogonal to $k$ (at least) hyperelliptic vectors. 
More precisely, the $Z^k$'s are as follows:
\begin{enumerate}
\item
If $1\le k\le 4$, then $Z^k=\im f_{18-k,18}$, the locus of $O^{+}(\Lambda)$-equivalence classes of points 
 $[\sigma]\in\cD^{+}_{\Lambda}$ such that $\sigma^{\bot}$  contains $k$ pairwise orthogonal   hyperelliptic vectors
\item
$Z^6=\im(f_{13,18}\circ q_{13})$, the locus of $O^{+}(\Lambda)$-equivalence classes of points 
 $[\sigma]\in\cD^{+}_{\Lambda}$ such that $\sigma^{\bot}$  contains  pairwise orthogonal   vectors $v_1,\ldots,v_5,w$, with 
 $v_1,\ldots,v_5$ hyperelliptic, $w^2=-12$, and the divisibility of $w$ in   $\{v_1,\ldots,v_5\}^{\bot}$ equal to $4$ (see Prop.~1.6.1 
 ibid.).
\item
$Z^7=\im(f_{12,18}\circ m_{12})$,  the locus of $O^{+}(\Lambda)$-equivalence classes of points 
 $[\sigma]\in\cD^{+}_{\Lambda}$ such that $\sigma^{\bot}$  contains  pairwise orthogonal   vectors $v_1,\ldots,v_6,w$, with 
 $v_1,\ldots,v_6$ hyperelliptic, $w^2=-2$, and the divisibility of $w$ in   $\{v_1,\ldots,v_6\}^{\bot}$ equal to $2$ (see Prop.~1.5.2 
 ibid.).
\item
$Z^8=\im f_{10,18}\sqcup \im (f_{11,18}\circ l_{11})$, where $\im f_{10,18}$ is as in Item~(1) , with $k=8$, and $\im (f_{11,18}\circ l_{11})$
is the locus of $O^{+}(\Lambda)$-equivalence classes of points 
 $[\sigma]\in\cD^{+}_{\Lambda}$ such that $\sigma^{\bot}$  contains  pairwise orthogonal   vectors $v_1,\ldots,v_7,w$, with 
 $v_1,\ldots,v_7$ hyperelliptic, $w^2=-4$, and the divisibility of $w$ in   $\{v_1,\ldots,v_7\}^{\bot}$ equal to $4$ 
 (see Prop.~1.5.1 ibid.).
\end{enumerate}
Notice that $Z^1=H_h$.
\subsection{Hyperelliptic $K3$'s of degree $4$ whose periods are parametrized by $Z^k$}
In the present subsection we will freely use notation and results  contained in~\cite{log1}. 
\begin{proposition}\label{prp:hypkgeom}
Let $(X,L)$ be a hyperelliptic quartic $K3$ surface, and let  $x\in\sF$ be its period point. Let $\varphi_L\colon X\to |L|^{\vee} \cong\PP^3$ be the map associated to $L$. 
Let $Q:=\varphi_L(X)$ be the image, a quadric, and let $C\in |\cO_Q(4)|$ be the branch curve of $X\to Q$. The following hold:
\begin{enumerate}
\item
Let $1\le k\le 15$. Then $x\in \im f_{18-k,18}$ if and only if $Q$ is a quadric cone, and $C$ has an $A_{m}$-singularity at the vertex of $Q$, where $m\ge(k-1)$. If $k=8$, in addition $C$ must not contain a line.
\item
$x\in \im (f_{13,18}\circ q_{13})$  if and only if  $Q$ is a quadric cone, $C$    has an $A_{m}$-singularity at the vertex of $Q$, where $m\ge 4$, and the support of the tangent cone of $C$ at   the vertex of $Q$ is tangent to a line of $Q$.  
\item
$x\in \im (f_{12,18}\circ m_{12})$  if and only if  $C$   contains a line, and has an $A_{m}$-singularity at the vertex of $Q$, where $m\ge 5$.  
\item
$x\in \im (f_{11,18}\circ l_{11})$  if and only if  $C$  contains a line, and has an $A_{7}$-singularity at the vertex of $Q$. 
\end{enumerate}
\end{proposition}
We will prove~\Ref{prp}{hypkgeom} at the end of the present subsection.  
\begin{remark}\label{rmk:conesetup}
 Let $Q\subset\PP^3$  be a quadric cone, with vertex $v$. Let   $C\in |\cO_Q(4)|$, and suppose that $C$ has simple singularities.   Let $\varphi\colon X\to Q$ be the double cover ramified over $C$, and let  $L:=\varphi^{*}\cO_Q(1)$. Then  $(X,L)$  is a hyperelliptic quartic $K3$ surface, and the image of $\varphi_L\colon X\to |L|^{\vee}$ is identified with $Q$.   Let $\mu_0\colon Q_0\to Q$ be the blow up of the vertex $v$; thus $Q_0\cong\FF_2$.  
Let $A,F\subset Q_0$ be the negative section and a fiber of the $\PP^1$-fibration $Q_0\to \PP^1$, respectively. 
  Let    $C_0:=\mu_0^{*}C\in |4A+8F|$,  and let $\varphi_0\colon X_0\to Q_0$ be the double cover ramified over $C_0$. Then    $X_0$ is a $K3$ surface as well.  It follows that $C_0$ has simple singularities. Conversely, if  $C_0$ has simple singularities, then $X$ is a $K3$ surface.   
  
  Let $\wt{X}\to X_0$ be the minimal desingularization of $X_0$, and  let $\wt\varphi_0\colon\wt{X}\to Q_0$ be  the composition $\wt{X}\to X_0\to Q_0$. The orthogonal 
  $\wt{\varphi}_0^{*} H^2(Q_0;\ZZ)^{\bot}\subset H^2(\wt{X};\ZZ)$ is isomorphic to the lattice $\Lambda$ defined in~\Ref{subsec}{predictqh}. 
  Let
$\psi\colon \wt{\varphi}_0^{*} H^2(Q_0;\ZZ)^{\bot}\overset{\sim}{\lra}\Lambda$ be an isomorphism, and let $\psi_{\CC}$ be its $\CC$-linear extension. Then 
$\psi_{\CC}H^{2,0}(\wt{X})\in\cD_{\Lambda}$ and composing, if necessary, with a suitable automorphism of $\Lambda$, we may assume that 
$\psi_{\CC}H^{2,0}(\wt{X})\in\cD^{+}_{\Lambda}$. The period point of $(X,L)$ in $\sF$ is the $O^{+}(\Lambda)$-equivalence class of $\psi_{\CC}H^{2,0}(\wt{X})$; we will denote it by $\Pi(X,L)$.
\end{remark}
\begin{proposition}\label{prp:keyone}
Let $(X,L)$ be a hyperelliptic quartic $K3$. Then $\Pi(X,L)\in H_h$ if and only if the image quadric $\varphi_L(X)\subset |L|^{\vee}$ is a cone.   
\end{proposition}
\begin{proof}
Let  $Q:=\varphi_L(X)$. Suppose that  $Q$ is a quadric cone, and let us prove that $\Pi(X,L)\in H_h$. 
We will adopt the notation  in~\Ref{rmk}{conesetup}. Assume that the vertex $v$ of $Q$ is not in the support of the branch curve $C$.  Since $A:=\mu_0^{-1}(v)$, the support of the branch divisor $C_0$ is disjoint from $A$, and hence 
$\wt\varphi_0^{*}A=A_1+A_2$, where $A_1,A_2\subset \wt{X}$ are disjoint smooth rational curves. It follows that $[A_1-A_2]\in \wt{\varphi}_0^{*} H^2(Q_0;\ZZ)^{\bot}$.   By definition of $H_h$, it will suffice to prove that $\psi([A_1-A_2])$ is a hyperelliptic vector. First, 
$\psi([A_1-A_2])^2=[A_1-A_2]^2=-4$, secondly $\psi([A_1-A_2])$ has divisibility equal to $2$ because $[A_1-A_2]+\wt{\varphi}_0^{*}A=2A_1$ (we recall that the  divisibility of a vector in $\Lambda$ of square $-4$ is either  $1$ or $2$).  This proves that if $Q$ is a quadric cone, and  the branch curve  does not contain the vertex of the cone, then $\Pi(X,L)\in H_h$. Since $H_h$ is closed in $\sF$, the same is true whenever  $Q$ is a quadric cone. By irreducibility of $H_h$, and a dimension count, the proposition follows.
\end{proof}
\begin{remark}\label{rmk:moresetup}
Retain the notation of~\Ref{rmk}{conesetup}.  Suppose that the branch curve $C$ contains the vertex $v$, and that the double cover   $\varphi\colon X\to Q$  ramified over $C$ is a $K3$.  Then $C_0=A+D$, where $D\in | 3A+8F|$, and $C_0$ has simple singularities. In particular $D$ does not contain $A$. Now notice that 
$D\cdot A=2$. Thus, either $D$  intersects transversely $A$ in $2$ points, or it intersects $A$ in a single point $p$. Assume that the latter holds. Since $C_0=A+D$ has simple singularities, it follows that  $D$ has an $A_{n}$-singularity at $p$ (if $n=0$, this means that $C$ is smooth  at $p$, and simply tangent to $A$), and simple singularities elsewhere. Conversely, if  $D$ is as described above, then $X_0$ is a $K3$, and hence so is $X$.
For $k\ge 2$ the following are equivalent:
\begin{enumerate}
\item
$C$ has an $A_{k-1}$-singularity at $v$.
\item
$D$ has an  $A_{k-3}$-singularity at the point(s) of intersection with $A$. 
\item
$C_0$ has a  $D_{k}$-singularity at the point(s) of intersection with $A$. 
\end{enumerate}
(If $k=2$, then Item~(2) is to be interpreted as saying that $D$ intersects $A$ transversely in two points, and Item~(3) accordingly.) 
\end{remark}
\begin{definition}\label{dfn:essekappa}
Let $Q\subset\PP^3$ be a quadric cone, with vertex $v$. Let $|\cO_Q(4)|_{K3}\subset |\cO_Q(4)|$ be the (open) subset of $C$ with simple singularities (i.e.~such that the double cover $X\to Q$ ramified over $B$ is a $K3$), and  let $\cS_k(Q)\subset|\cO_Q(4)|_{K3}$  be the set of curves   with   an $A_{k-1}$-singularity at $v$. 
\end{definition}
Notice that   $\cS_k(Q)$ is locally closed in $|\cO_Q(4)|_{K3}$. By~\Ref{rmk}{moresetup},
\begin{equation}
|\cO_Q(4)|_{K3}=\bigcup_{k=1}^{\infty}\cS_k(Q).
\end{equation}
\begin{proposition}\label{prp:dikappa}
Let $Q\subset\PP^3$ be a quadric cone, with vertex $v$. Let $C\in \cS_k(Q)$, with 
$k\ge 1$. Let $\varphi\colon X\to Q$ be the double cover ramified over $C$, and let $L:=\varphi^{*}\cO_Q(1)$.  Then 
\begin{equation*}
\Pi(X,L)\in 
\begin{cases}
\im f_{18-k,18} & \text{if $k\not=8$,} \\
\im f_{10,18}\cup \im(f_{11,18}\circ l_{11}) & \text{if $k=8$.}
\end{cases}
\end{equation*}
\end{proposition}
\begin{proof}
The statement for $k=1$ has been proved in~\Ref{prp}{keyone}. Thus we may assume that $k\ge 2$. 
Let $R\subset H^{1,1}(\wt{X};\ZZ)$ be the span of the classes of curves which are mapped to a point of $A\cap C$ (notation as in~\Ref{rmk}{conesetup}) by $\wt{\varphi}_0$. Then the following hold:
\begin{enumerate}
\item
$R\subset \wt{\varphi}_0^{*}H^2(Q_0;\ZZ)^{\bot}$,
\item
$R\cong D_k$,
\end{enumerate}
Item~(1) is clear.  If $k=2$, then $C$ intersects $A$ at two points, and the intersection is transverse; thus  $R\cong A_1^2\cong D_2$. If $k\ge 3$ then $C$ intersects $A$ at a single point $p$, which is a $D_k$-singularity of $B_0$ (see~\Ref{rmk}{moresetup}), and Item~(2) follows. Let $\psi\colon \wt{\varphi}_0^{*}H^2(Q_0;\ZZ)^{\bot} \overset{\sim}{\lra}\Lambda$ be as in~\Ref{rmk}{conesetup}; then $\Pi(X,L)\subset \psi(R)^{\bot}$. By Proposition~1.7.2 in~\cite{log1}, it will suffice to prove that   $\psi(R)$ contains $k$ pairwise orthogonal hyperelliptic vectors. 

A straightforward computation shows  that
\begin{equation*}
\wt{\varphi}_0^{*}A=2\Gamma_0+2\Gamma_1+\ldots+2\Gamma_{k-2}+\Gamma_{k-1}+\Gamma_k,
\end{equation*}
where  $\wt{\varphi}_0(\Gamma_0)=A$,  $\wt{\varphi}^{-1}_0(A\cap D)=\Gamma_1\cup\ldots\cup\Gamma_k$, and the rational curves $\Gamma_0,\Gamma_1,\ldots,\Gamma_k$  form a $D_{k+1}$-Dyinkin diagram
\begin{equation}\label{dynkdien}
 \centering
\xy
\POS(-10,0) = "z",
\POS(0,0) *\cir<2pt>{}="a",
\POS(10,0) *\cir<2pt>{}="b",
\POS(20,0) *\cir<2pt>{}="c",
\POS(30,0) *\cir<2pt>{}="d",
\POS(40,5) *\cir<2pt>{}="e",
\POS(40,-5) *\cir<2pt>{}="f",
\POS(50,5) ="g",
\POS(50,-5) ="h"

\POS "a" \ar@{-}^<<{\Gamma_0} "b",
\POS "b" \ar@{.}^<<{\Gamma_1} "c",
\POS "c" \ar@{.}^<<{\Gamma_l} "d",
\POS "d" \ar@{-}^<<<<{} "e",
\POS "d" \ar@{-}^<{\Gamma_{k-2}} "f",
\POS "e" \ar@{}^<<{\Gamma_{k-1}} "g",
\POS "f" \ar@{}^<{\Gamma_{k}} "h",
\endxy 
\end{equation}
Thus $R$ is spanned by the cohomology classes $[\Gamma_1],\ldots,[\Gamma_k]$.  
The $k$ cohomology classes
\begin{equation}\label{lista}
\scriptstyle
[\Gamma_k]-[\Gamma_{k-1}],\ [\Gamma_k]+[\Gamma_{k-1}],\ [\Gamma_k]+[\Gamma_{k-1}]+2[\Gamma_{k-2}]+\ldots+2[\Gamma_{i}], \quad 1\le i\le k-2,
\end{equation}
are pairwise orthogonal, and of square $-4$. 
In order to finish the proof it suffices to show that each of the classes in~\eqref{lista} has divisibility  $2$ in $\wt{\varphi}_0^{*}H^2(Q_0;\ZZ)^{\bot}$. 
Let $1\le i\le k-1$, and $z\in \wt{\varphi}_0^{*}H^2(Q_0;\ZZ)^{\bot}$; then
\begin{equation*}
\scriptstyle
0=z\cdot \wt{\varphi}_0^{*}A=z\cdot(2\Gamma_0+2\Gamma_1+\ldots+2\Gamma_{k-2}+\Gamma_{k-1}+\Gamma_k)
\equiv z\cdot(2\Gamma_i+2\Gamma_{i+1}+\ldots+2\Gamma_{k-2}+\Gamma_{k-1}+\Gamma_k)\pmod{2}
\end{equation*}
(If $i=k-1$, the expression in parentheses is to be understood as $(\Gamma_{k-1}+\Gamma_k)$.) This proves that all of the classes listed in~\eqref{lista} have divisibility a multiple of $2$. On the other hand 
$$\Gamma_k\cdot (\Gamma_k\pm \Gamma_{k-1})=-2,\quad \Gamma_i\cdot (\Gamma_k+\Gamma_{k-1}+2\Gamma_{k-2}+\ldots+2\Gamma_{i})=-2, \quad 1\le i\le k-2,$$
and hence their divisibility is equal to $2$. 
\end{proof}
\begin{proposition}\label{prp:esiste}
Let $Q\subset\PP^3$ be a quadric cone, with vertex $v$. Then $\cS_k(Q)$  is not empty for  $1\le k\le 16$, and each of its 
 irreducible components  has codimension at most $(k-1)$ in $|\cO_Q(4)|_{K3}$.  There exist irreducible 
  components $\cZ_1(Q),\ldots,\cZ_{16}(Q)$ of $\cS_1(Q),\ldots,\cS_{16}(Q)$ respectively, such that $\cZ_k(Q)\subset\ov{\cZ_{k-1}(Q)}$. 
\end{proposition}
\begin{proof}
First we prove that  $\cS_k(Q)$  is not empty, and that its codimension is as stated. If $k=1$, the result is trivially true. 
Since $\cS_2(Q)$ is an open dense subset of $|\cI_v\otimes\cO_Q(4)|$,  the statement of the proposition is  true for $k=2$. It is equally easy to check that the statement  is true for $k=3$, if we identify $\cS_3(Q)$ with the set of $D\in|3A+8F|$  with simple singularities, which intersect $A$ at a single point, and are smooth at that point, see~\Ref{rmk}{moresetup}.
Thus we may assume that $k\ge 4$. 
By~\Ref{rmk}{moresetup} we may identify $\cS_k(Q)$ with the subset $\cT_k(Q)\subset |3A+8F|$ parametrizing divisors $D$ with simple singularities, and  with an  $A_{k-3}$-singularity at the unique point of intersection with $A$.  Once we know that  $\cT_k(Q)$ is not empty,  the codimension statement follows from general results about the versal deformation space of an $A_{k-3}$-singularity. 
Let us prove that $\cT_k(Q)$ is not empty.  There exist  $D_1\in |A+3F|$, and  $D_2\in |2A+5F|$ intersecting at a point $p\in A$  (notice that $A\cdot(A+3F)=1$) with multiplicty $7=(A+3F)\cdot(2A+5F)$, and  smooth at $p$; in fact one argues that $D_1,D_2$ exist via a simple cohomological argument. Then $D=D_1+D_2\in \cT_{16}(Q)$, and hence $\cS_{16}(Q)$ is not empty.  Let $\varphi\colon X\to Q$ be the double cover ramified over a curve $C\in \cS_{16}(Q)$, and let $\wt{X}$ be the minimal desingularization of $X$.  
Let $\Gamma_0,\Gamma_1,\ldots,\Gamma_{16}\subset\wt{X}$ be the rational curves introduced in the proof of~\Ref{prp}{dikappa}. In the versal deformation of $\wt{X}$, consider deformations that keep the classes $\wt{\varphi}_0^{*}F, [\Gamma_0+\Gamma_1],[\Gamma_2],\ldots,[\Gamma_k]$ of type $(1,1)$. The generic such deformation is the desingularization of a double cover $X'\to Q_0$ ramified over  a divisor $A+D'$, where $D'\in \cT_{15}(Q)$. Iterating, one proves that $\cT_k(Q)$  is not empty for $1\le k\le 16$. This argument also produces  irreducible  components  $\cZ_1(Q),\ldots,\cZ_{16}(Q)$ of  $\cS_1(Q),\ldots,\cS_{16}(Q)$ respectively with the stated property.
\end{proof}

\n
{\it Proof of~\Ref{prp}{hypkgeom}.\/} We have $\im f_{18-k,18}\subset H_h(18)$ for $1\le k$ (see \S1.7.1 in~\cite{log1}),  and $\Pi(X,L)\in H_h$ if and only if $\varphi_L(X)$ is a quadric cone by~\Ref{prp}{keyone}. It follows that we may  \emph{fix} a quadric cone $Q\subset\PP^3$, with vertex $v$, and consider only double covers $X\to Q$ ramified over  curves  
$C\in |\cO_Q(4)|_{K3}$. Let 
$\Phi\colon  |\cO_Q(4)|_{K3} \to H_h$
 be the period map associating to $C$ the  point  $\Pi(X,L)$, where  
$X\to Q$ is the double cover  ramified over  $C$. By Global Torelli the fibers of $\Phi$ are the orbits for the action of $\Aut(Q)$ on $ |\cO_Q(4)|_{K3}$.
The image by $\Phi$ of an open  subset of $ |\cO_Q(4)|_{K3}$ is open in $H_h$. It follows  that  the image by $\Phi$ of a closed and $\Aut(Q)$-invariant 
subset of $ |\cO_Q(4)|_{K3}$  is closed. 

Let us prove Items~(1) and~(4). Let $1\le k\le 15$. 
The subset $\Phi(\ov{\cS(Q)}_k)\subset \sF$ is closed because $\ov{\cS(Q)}_k$ is closed and $\Aut(Q)$-invariant. Moreover 
\begin{equation}\label{twooptions}
\Phi(\ov{\cS(Q)}_k)\subset 
\begin{cases}
\im f_{18-k,18} & \text{if $k\not=8$,} \\
\im f_{10,18}\cup \im(f_{11,18}\circ l_{11}) & \text{if $k=8$.}
\end{cases}
\end{equation}
 by~\Ref{prp}{dikappa}. 
 On the other hand,  by~\Ref{prp}{esiste} every irreducible component of $\ov{\cS(Q)}_k$  has codimension at most $(k-1)$ in $|\cO_Q(4)|_{K3}$.
  
  Now assume that  $k\not=8$.  Since $ \im f_{18-k,18}$ 
   is irreducible, closed, of codimension $k$ in $\sF$ (and hence of codimension $k-1$ in $H_h$), it follows from~\eqref{twooptions} that  $\Phi(\ov{\cS(Q)}_k)=\im f_{18-k,18}$. 
 
Next, assume that  $k=8$. 
 Let $\cS_8^{\rm ex}(Q)\subset \cS_8(Q)$ be the subset of divisors  $(L+D)$, where 
$L$ is a line on $Q$. Then $\cS_8^{\rm ex}(Q)$ is closed, $\Aut(Q)$-invariant,  and a straightforward parameter count shows that $\cS_8^{\rm ex}(Q)$ has codimension $7$ in $|\cO_Q(4)|_{K3}$. 
Recalling Proposition~1.7.2 in~\cite{log1}, and arguing as above, it follows that $\Phi(\cS_8^{\rm ex}(Q))$ is an irreducible component of $\im f_{10,18}\cup \im (f_{11,18}\circ l_{11})$. On the other hand, $\cS_8^{\rm ex}(Q)$ is not the whole of $\cS_8(Q)$, because  $\cS_8(Q)$ is not closed, by~\Ref{prp}{esiste}. It follows that there is a unique other irreducible component of $\cS_8(Q)$ (unicity follows from Proposition~1.7.2 ibid.), call it  $\cS_8^{\rm st}(Q)$, and that  
\begin{equation*}
\Phi(\ov{\cS}_8^{\rm st}(Q))=\im f_{10,18},\qquad 
\Phi(\cS_8^{\rm ex}(Q))=\im(f_{11,18}\circ l_{11}).
\end{equation*}
Now let us prove Item~(2). Suppose that $C\in \cS_5(Q)$, and that the support of the tangent cone of $C$ at   the vertex of $Q$ is tangent to a line of $Q$.
Let $(X,L)$ be the hyperelliptic polarized $K3$ surface of degree $4$ corresponding to $C$. 
Let us prove that $\Pi(X,L)\in \im(f_{13,18}\circ q_{13})$. Retaining the notation of~\Ref{rmk}{conesetup} and~\Ref{rmk}{moresetup}, we have $C_0=A+D$, where by our hypotheses $D$ intersects $A$ in a single point $p$,  $D$ has an $A_4$ singularity at $p$, and the tangent cone $\cC_p(D)$ is equal to $2T_p(E)$, where $E\in |F|$ is the unique divisor containing $p$. 
Thus $E$ intersects $C_0$ in the single point $p$, with multiplicity $4$, and hence $\wt{\varphi}_0^{*}E$ splits as the sum of two divisors: $\wt{\varphi}_0^{*}E=E'+E''$. Let 
$\Gamma_0,\ldots,\Gamma_5$ be as in the proof of~\Ref{prp}{dikappa}. The cohomology class $[E'-E'']$ is orthogonal to $\wt{\varphi}_0^{*}H^2(Q_0;\ZZ)$.  One checks that 
\begin{equation*}
(E'-E'')\cdot \Gamma_i=
\begin{cases}
0 & \text{if $i\in\{0,3\}$,} \\
1 & \text{if $i=4$,} \\
-1 & \text{if $i=5$.} \\
\end{cases}
\end{equation*}
Thus $\alpha:=[\Gamma_4-\Gamma_5+2E'-2E'']$ is orthogonal both  to $\wt{\varphi}_0^{*}H^2(Q_0;\ZZ)$ and to $R$ (the span of $[\Gamma_1],\ldots,[\Gamma_5]$, notation as  in the proof of~\Ref{prp}{dikappa}). Since  $R$ contains $5$ pairwise orthogonal hyperelliptic vectors (see the proof  of~\Ref{prp}{dikappa}), it suffices to show that $\alpha$ has square $-12$ and divisibility $4$ in the orthogonal of  $\wt{\varphi}_0^{*}H^2(Q_0;\ZZ)\cup R$. The equality $\alpha^2=-12$ is straightforward. Let $z\in H^2(\wt{X};\ZZ)$ be orthogonal to 
$\wt{\varphi}_0^{*}H^2(Q_0;\ZZ)\cup R$. Since  $\wt{\varphi}^{*}_0 E\equiv E'+E''\pmod{R}$, we have $z\cdot [E']=-z\cdot [E'']$, and hence $z\cdot \alpha=4z\cdot [E']$. Thus the divisibility of $\alpha$ is a multiple of $4$. Since the orthogonal to 
$\wt{\varphi}_0^{*}H^2(Q_0;\ZZ)\cup R$ is a $D$ lattice (see~\cite{log1}), the divisibility of a (primitive) vector of $\wt{\varphi}_0^{*}H^2(Q_0;\ZZ)\cup R$ is at most $4$, hence 
$\divisore(\alpha)=4$. 

Now, let  $\cZ\subset \cS_5(Q)$ be the subset parametrizing curves as above, and let $\ov{\cZ}$ be its closure in $|\cO_Q(4)|_{K3}$. We proved above that $\Phi(\cZ)\subset \im(f_{13,18}\circ q_{13})$, and since $\im(f_{13,18}\circ q_{13})$ is closed, it follows that 
$\Phi(\ov{\cZ})\subset \im(f_{13,18}\circ q_{13})$.
 On the other hand, an easy dimension count shows that $\cZ$ has codimension $5$ in $|\cO_Q(4)|_{K3}$; since  $\ov{\cZ}$ is $\Aut(Q)$-invariant, it follows that $\Phi(\ov{\cZ})$ is closed, of codimension $5$ in $H_h$. Since $ \im(f_{13,18}\circ q_{13})$ is irreducible of codimension $5$ in $H_h$, Item~(2) follows.

Lastly, we prove Item~(3). Suppose that $C\in \cS_6(Q)$ contains a line, and let $(X,L)$ be the hyperelliptic polarized $K3$ surface of degree $4$ corresponding to $C$. 
Let us prove that $\Pi(X,L)\in \im(f_{12,18}\circ m_{12})$.
Retaining the notation of~\Ref{rmk}{conesetup} and~\Ref{rmk}{moresetup}, we have $D=D'+E$ with
$D'\in |3A+7F|$ and $E\in |F|$.
Moreover there is a unique point in $A\cap D'\cap E$, call it $p$, and $\mult_p(D'\cdot E)=2$ (because $D$ has an $A_3$ singularity at $p$, see~\Ref{rmk}{moresetup}). Since $D'\cdot E=3$, there is a unique point in $(D'\cap E)\setminus A$, call it $q$, and $\mult_q(D'\cdot E)=1$.  Then $\wt{\varphi}_0^{-1}(q)$ is a smooth rational curve $N\subset \wt{X}$ (notation as in~\Ref{rmk}{conesetup}). Next, let $M\subset H^{1,1}(\wt{X};\ZZ)$ be the subgroup spanned by $\wt{\varphi}_0^{*}H^2(Q_0;\ZZ)$ and $R$ (notation as  in the proof of~\Ref{prp}{dikappa}). Notice that 
$[N]\in M^{\bot}$. We have $N\cdot N=-2$. By definition of $\im(f_{12,18}\circ m_{12})$, it will suffice to show that
\begin{equation}\label{uragano}
([N],M^{\bot})=2\ZZ.
\end{equation}
Let $\wt{E}\subset\wt{X}$ be the strict transform of $E$. Then (notation as in the proof of~\Ref{prp}{dikappa})
\begin{equation}\label{manyterms}
\wt{\varphi}_0^{*}E=N+2\wt{E}+\Gamma_1+2\Gamma_2+3\Gamma_3+4\Gamma_4+3\Gamma_5+2\Gamma_6.
\end{equation}
Since $\wt{\varphi}_0^{*}E,\Gamma_1,\ldots,\Gamma_5\in M$, Equation~\eqref{uragano} follows at once from~\eqref{manyterms}.  
This proves that $\Pi(X,L)\in \im(f_{12,18}\circ m_{12})$.

One concludes the proof of Item~(3)   arguing exactly as in the proof of Item~(2).
\qed
\subsection{The boundary of the Baily-Borel compactification}\label{subsec:sectbb}
We recall from~\cite{log1} that $\sF(19)$ is the period space of polarized $K3$'s of degree $4$. 
The boundary of  $\sF(19)^{*}$, i.e.~$\sF(19)^{*}\setminus \sF(19)$ has been studied by Scattone in~\cite{scattone}. In the appendix \Ref{sec}{bbcompact} we generalize Scattone's discussion to the initial segment  $\sF(3)\subset\ldots\subset\sF(20)$  of the $D$-tower (Camere~\cite{camere} has analyzed the boundary of  $\sF(20)^{*}$). In particular, for hyperelliptic quartic $K3$s we get

\begin{theorem}\label{thm:thmbbstrata}
The Baily-Borel compactification $\calF^*$ (associated to $D_{16}\oplus U^2$) consists of
\begin{itemize}
\item[i)] Two Type III boundary points, that we label $III_a$ and $III_b$ respectively.
\item[ii)] Eight Type II boundary components. Six of the Type II boundary components, call them of {\it type a} (labeled by $D_{16}$, $D_8\oplus E_8$, $(E_7)^2\oplus D_2$, $A_{15}\oplus D_1$, and $D_{8}^2$) are incident only to $III_a$. The remaining two Type II boundary component (label $(E_8)^2$ and $D_{16}^+$) are incident to both $III_a$ and $III_b$. The incidence diagram is given in \eqref{incidencebb}. 
\begin{equation}\label{incidencebb}
\xymatrix{
&&& \ \ \ \ \ \ \ \circ^{II(D_{16})}\\ 
&&& \ \ \ \ \ \ \circ^{II(D_7)}\\ 
 & \ \ \ \ \ \ \ \circ^{II(D_{16}^+)}\ar@{-}[rd]&& \ \ \ \ \ \ \ \circ^{II(D_8\oplus E_8)}\\
 \bullet^{III_b}\ar@{-}[ru]\ar@{-}[rd]&&\bullet^{III_a}\ar@{-}[ruuu]\ar@{-}[ruu]\ar@{-}[ru]\ar@{-}[r]\ar@{-}[rd]\ar@{-}[rdd]&\circ^{II(D_{12}\oplus D_4)}\\
&\ \ \ \ \ \ \ \ \circ_{II((E_8)^2)}\ar@{-}[ru]&& \ \ \ \ \ \ \circ_{II(D_7)}\\
&&& \ \ \ \ \ \ \ \circ_{II(D_7)}\\ 
}
\end{equation}
\item[iii)] Furthermore, each of the six Type IIa boundary components are isomorphic to the modular curve $\fh/ \SL(2,\bZ)$. Each of the two Type IIb boundary components are isomorphic to the modular curve $\fh/\Gamma_0(2)$ where $\Gamma_0(2)=\left\{\left(\begin{matrix} a&b\\c&d\end{matrix}\right), \ c\equiv 0 (2)\right\}$.
\end{itemize}
\end{theorem}
\begin{proof}
The only remaining thing to check is the stabilizer of the two special Type II components ($D_{16}^+$ and $(E_8)^2$).  They correspond specifically to the following situation $E\subset \Lambda$ is a rank $2$ isotropic lattice such that there exists $v\in E\subset \Lambda$ primitive with $\divisore_{\Lambda}(v)=2$. It is immediate to see that the image of the natural morphism (from the normalizer of the cusp to the automorphism group of $E$)
$$N_E(O^+(\Lambda))\to \Aut(E)=\GL(2,\bZ)$$
is $\Gamma_0(2)$. Simply, a basis of $E$ can be taken to be $\{v,w\}$ with $\divisore(v)=2$ and $\divisore(w)=1$. Any base change of $E$ induced by an isometry of the lattice $\Lambda$ will preserve the divisibility of $v$. Thus, we only allow base changes of type 
$v'=av+cw$ with $c\equiv 0 (2)$ (and no restriction for $w$), which gives precisely $\Gamma_0(2)$. 
\end{proof}

\begin{remark}\label{rmk:remnormalize}
As we mentioned above, $\calF(19)$ is the period space of polarized $K3$ surfaces of degree $4$. As discussed in our previous paper~\cite{log1}, the period space of  hyperelliptic polarized $K3$ surfaces of degree $4$ is  a normal Heegner divisor:
$$\calF\subset \calF(19).$$
(The embedding is induced from the natural lattice embedding $\Lambda_{18}=D_{16}\oplus U^2\subset \Lambda_{19}=D_{17} \oplus U^2$.) This extends to a morphism at the level of the Baily-Borel compactification 
$$\calF^*\to \calF(19)^*$$
which is the normalization of the image. There are $9$ Type II Baily-Borel components for $\calF(19)^*$ (cf. \cite[\S6.3]{scattone}). One of these components (labelled $E_6\oplus D_7\oplus A_{11}$) does not meet the image of $\calF^*$ and thus it does not appear as a Type II component in $\sF^*$. For the remaining $8$ Type II components we have the following behavior. First, the six Type IIa components of $\calF^*$ will map isomorphically to six Type II components in $\calF(19)^*$. The remaining two components of Type IIb (with label $D_{16}^+$ and $(E_8)^2$) will map $3$-to-$1$ to the two Type II components (labeled $E_8^2\oplus D_1$ and $D_{16}\oplus D_1$ respectively) in $\calF(19)^*$ (N.B. $3=[\SL(2,\bZ):\Gamma_0(2)]$).  Similarly, the two Type III points in $\calF^*$ will map to the unique Type III point in $\calF(19)^*$. Furthermore,  the  closure of a Type IIb component in $\sF^*$ will contain both the $III_a$ and  $III_b$ points (N.B. $(\mathfrak h/\Gamma_0(2))^*\cong \bP^1$ with $2$ cusps) with $III_b$ a ramification point, and IIIa not ramification. 
\end{remark}

\begin{remark}
The subgroup $\Gamma_0(2)$ corresponds to the moduli of elliptic curves with a choice of non-trivial torsion point of order $2$. Equivalently, we can view it as moduli of $4$ points in $\bP^1$ where the points are coming in pairs. The point at infinity for the moduli of $4$ points in $\bP^1$ corresponds to two double points in $\bP^1$. When the points are coming in pairs, there are two possibilities for the point at infinity: the double points are coming from different pairs, or the double points are coming from the same pair. 
\end{remark}

\section{Moduli of $(4,4)$ curves on a smooth quadric}\label{sec:quattro}
\subsection{GIT for $(4,4)$ curves on $\bP^1\times\bP^1$} 
The present subsection contains results of Shah~\cite{shah}, sometimes reformulated and completed. Let  $q:=x_0x_3-x_1x_2$, and let $Q:=V(q)\subset\PP^3$. In what follows we will  identify 
  $\bP^1\times\bP^1$ and $Q$ via the Segre isomorphism 
\begin{equation}
\begin{matrix}
\bP^1\times\bP^1 & \overset{\sim}{\lra} & Q \\
([u_0,u_1],[v_0,v_1]) & \mapsto & [u_0v_0,u_0v_1,u_1v_0,u_1 v_1].
\end{matrix}
\end{equation}
Given the above identification, $\Aut(\bP^1\times\bP^1)$ can be described either as the group generated by $\PGL(2)\times\PGL(2)$ and the involution exchanging the factors, or 
 as the orthogonal group  $\PO(q)$. A curve on $\bP^1\times\bP^1 $ is a \emph{line} if, with the above identification,  it is a line in $\PP^3$.

Let $r_0,\ldots,r_n$ be integers (not all zero) adding up to $0$. Let $(s^{r_0},\ldots,s^{r_n})$ be  the diagonal $(n+1)\times(n+1)$ matrix with entries $s^{r_0},\ldots,s^{r_n}$, and 
let $\diag(r_0,\ldots,r_n)$   be the $1$-PS of  $\SL(n+1)$ mapping $s\in\CC^{*}$ to  $(s^{r_0},\ldots,s^{r_n})$. 

(Semi)stability of points of $|\cO_{\PP^1}(4)\boxtimes\cO_{\PP^1}(4)|$ is analyzed by contemplating Table~\ref{quadrato}.
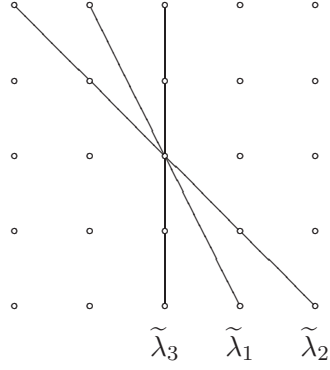
\bigskip
\begin{equation}\label{quadrato}
 \centering
\xy
(20,5) *{\wt{\lambda}_3},
(30,5) *{\wt{\lambda}_1},
(40,5) *{\wt{\lambda}_2},
\POS(0,10) *\cir<1pt>{}="a",
\POS(0,20) *\cir<1pt>{}="b",
\POS(0,30) *\cir<1pt>{}="c",
\POS(0,40) *\cir<1pt>{}="d",
\POS(0,50) *\cir<1pt>{}="e",
\POS(10,10) *\cir<1pt>{}="f",
\POS(10,20) *\cir<1pt>{}="g",
\POS(10,30) *\cir<1pt>{}="h",
\POS(10,40) *\cir<1pt>{}="i",
\POS(10,50) *\cir<1pt>{}="j",
\POS(20,10) *\cir<1pt>{}="k",
\POS(20,20) *\cir<1pt>{}="l"
\POS(20,30) *\cir<1pt>{}="m"
\POS(20,40) *\cir<1pt>{}="n"
\POS(20,50) *\cir<1pt>{}="o"
\POS(30,10) *\cir<1pt>{}="p",
\POS(30,20) *\cir<1pt>{}="q"
\POS(30,30) *\cir<1pt>{}="r"
\POS(30,40) *\cir<1pt>{}="s"
\POS(30,50) *\cir<1pt>{}="t"
\POS(40,10) *\cir<1pt>{}="u",
\POS(40,20) *\cir<1pt>{}="v"
\POS(40,30) *\cir<1pt>{}="w"
\POS(40,40) *\cir<1pt>{}="x"
\POS(40,50) *\cir<1pt>{}="y"
\POS "e" \ar@{-} "i",
\POS "i" \ar@{-} "m",
\POS "m" \ar@{-} "q",
\POS "q" \ar@{-} "u",
\POS "j" \ar@{-} "m",
\POS "m" \ar@{-} "p",
\POS "o" \ar@{-} "n",
\POS "n" \ar@{-} "m",
\POS "m" \ar@{-} "l",
\POS "l" \ar@{-} "k",
\endxy 
\end{equation}
\bigskip

The dots represent the elements of a monomial basis of $H^0(\cO_{\PP^1}(4)\boxtimes\cO_{\PP^1}(4))$. More precisely, the dot at position $(m,n)$ represents the monomial  $u^{4-m}_0,u^m_1,v^{4-n}_0,v^n_1$.  Each of the three segments is associated to a  $1$-PS  of $\SL(2)\times\SL(2)$: it passes through the monomials of weight $0$ for the action of the $1$ PS  on $H^0(\cO_{\PP^1}(4)\boxtimes\cO_{\PP^1}(4))$.
The three $1$ PS's are the following:
\begin{equation*}
\wt{\lambda}_1(s)=((s^2,s^{-2}),(s,s^{-1})),\ \wt{\lambda}_2(s)=((s,s^{-1}),(s,s^{-1})),\ \wt{\lambda}_3(s)=((s,s^{-1}),(1,1)). 
\end{equation*}
They correspond, in  Table~\ref{quadrato}, to the middle segment, the diagonal of the square, and the vertical segment respectively.
We notice that, up to isogeny, $\wt{\lambda}_1$, $\wt{\lambda}_2$ and $\wt{\lambda}_3$ correspond respectively to the  $1$-PS's of  $\SO(q)$ 
\begin{equation}\label{1234}
\lambda_1:=\diag(3,1,-1,-3),\ \lambda_2:=\diag(1,0,0,-1), \ \lambda_3:=\diag(1,1,-1,-1). 
\end{equation}
A straightforward argument gives the following results:
\begin{enumerate}
\item
If  $C\in | \cO_{\PP^1}(4)\boxtimes\cO_{\PP^1}(4)|$ is unstable, then, up to conjugation, there exists $i\in\{1,2,3\}$ such that $C$ is $\wt{\lambda}_i$-unstable.  
\item
Let  $C\in | \cO_{\PP^1}(4)\boxtimes\cO_{\PP^1}(4)|$ be properly semistable (i.e.~semistable and not stable), and $\sigma\in H^0(\cO_{\PP^1}(4)\boxtimes\cO_{\PP^1}(4))$ be a section whose divisor is $C$. Then, up to conjugation, there exists $i\in\{1,2,3\}$ such that $C$ is destabilized by $\wt{\lambda}_i$, i.e.~the limit $\lim\limits_{s\to 0}\wt{\lambda}_i(s)^{*}(\sigma)$ exists. 
\end{enumerate}
Before discussing the geometric consequences of the above result, we recall a few facts regarding singularities of double covers.
\begin{remark}\label{rmk:constrip}
A planar curve singularity $(C,p)$ has consecutive triple points if it has multiplicity $3$, and the blow up of $C$ at $p$ has a unique singular point lying over $p$, and that point has multiplicity $3$.  If $S$ is a smooth surface, and $\pi\colon T\to S$ is a double cover ramified over a reduced curve $C\subset S$, then  $T$ has  ADE singularities over $p\in S$ if and only if   $\mult_p(C)\le 3$, with the additional condition, if $\mult_p(C)= 3$, that  $C$  does \emph{not} have  consecutive triple points at $p$. 

How does one recognize whether  $(C,p)$ has consecutive triple points? First, if $(C,p)$ has consecutive triple points then the tangent cone $\cC_p(C)$ of $C$ at $p$ is a triple line. Next,  suppose that $C\subset S$ where $S$ is an open subset of the affine plane, and let $x,y$ be affine coordinates at $p$ such that $\cC_p(C)\subset\cC_p(S)=\aff^2$ is  given by $y^3=0$: then $C$ has consecutive triple points at $p$ if and only if an  equation of $C$ in a neighborhood of $p$ is
\begin{equation}
y^3+y^2f_2(x,y)+yf_4(x,y)+f_6(x,y)+\ldots+ f_n(x,y)=0,
\end{equation}
where each $f_d(x,y)$ is a homogeneous polynomial of degree $d$. 
\end{remark}
A straightforward analysis gives the following result.
\begin{lemma}\label{lmm:git44}
Let $C\in | \cO_{\PP^1}(4)\boxtimes\cO_{\PP^1}(4)|$. The following hold:
\begin{enumerate}
\item[(a)]
$C$ is desemistabilized by a $1$ PS conjugated to $\wt{\lambda}_2$ if and only if it has a point of multiplicity at least $5$, and it is destabilized  by a $1$ PS conjugated to 
$\wt{\lambda}_2$ if and only if it has a point of multiplicity at least $4$.
\item[(b)]
$C$ is desemistabilized by a $1$ PS conjugated to $\wt{\lambda}_3$ if and only if $C=3L+C'$, where $L$ is a line (on the other hand, if $C=3C'+C''$, where $C',C''$ are arbitrary, then $C$ has points of multiplicity at least $4$, hence we go to Item~(a) above). $C$ is destabilized by a $1$ PS conjugated to $\wt{\lambda}_3$ if and only if 
$C=2L+C'$, where $L$ is a line.
\item[(c)]
If $C$ is desemistabilized by a $1$ PS conjugated to $\wt{\lambda}_1$, then it is   destabilized by a $1$ PS conjugated to $\wt{\lambda}_2$ (hence we go to to Item~(a) above). $C$ is destabilized by a $1$ PS conjugated to $\wt{\lambda}_1$ if and only if  either it has a singular point $p$ with consecutive triple points and tangent cone  equal to $3T_p(L)$ where $L$ is a line, or  a point of multiplicity at least $4$ (if the latter holds,   we go to Item~(a) above).
\end{enumerate}
\end{lemma}
The  result below follows (one applies the results in~\cite{luna} in order to determine polystable non stable curves).
\begin{proposition}[cf.~Proposition 4.5 in~\cite{shah}]\label{prp:git44}
Let $C=V(\sigma)\in | \cO_{\PP^1}(4)\boxtimes\cO_{\PP^1}(4)|$. 
\begin{enumerate}
\item
$C$ is stable if and only if each of its irreducible components has multiplicity at most $2$, no component of multiplicity $2$ is a line, and  each of its points has multiplicity at most $3$, with the extra condition, if $C$ is singular at a point $p$ with consecutive triple points, that its tangent cone is \emph{not} equal to $3T_p(L)$ where $L$ is a line.  
\item
$C$ is properly semistable and polystable (i.e.~the $\Aut(\PP^1\times\PP^1)$ orbit of $\sigma$ is closed) if and only if one of the following holds:
\begin{enumerate}
\item
\emph{$\sigma$ is stabilized by $\wt{\lambda}_1$:} $\sigma=u_0 u_1(a_1 u_0 v_1^2+b_1 u_1 v_0^2)(a_2 u_0 v_1^2+b_2 u_1 v_0^2)$ in suitable homogeneous coordinates $([u_0,u_1],[v_0,v_1])$ , and $(a_1\cdot b_2,a_2\cdot b_1)\not=(0,0)$. Equivalently $C=L_1+L_2+T_1+T_2$, where $L_1,L_2$ are skew lines, $T_1$, $T_2$ are  twisted cubics (eventually singular) intersecting each line  $L_i$  tangentially at the same point $p_i$ (with $p_1$, $p_2$ not belonging to a line), satisfying the condition that no $L_i$ has multiplicity greater than $2$. The moduli space of such curves is $\PP^1$: map  $V(\sigma)$ to $[a_1^2b_2^2+a_2^2b_1^2,a_1 a_2 b_1 b_2]$. 
\item
\emph{$\sigma$ is stabilized by $\wt{\lambda}_2$:} $\sigma=\prod_{i=1}^4(a_i u_0 v_1+b_i u_1 v_0)$ in suitable homogeneous coordinates $([u_0,u_1],[v_0,v_1])$, and at most one of the $a_i$'s vanishes, and similarly for the $b_j$'s. Equivalently $C$ is the sum of four members of the pencil of conics through  two points not on a line, and no 
 reducible conic appearing in $C$ (if there are any) has multiplicity greater than one.
The moduli space of such curves  is $3$ dimensional. The moduli map is the composition of the map
\begin{equation*}
\begin{matrix}
(\PP^1\times\PP^1\times\PP^1\times\PP^1)^{ss} & \lra & \PP^5 \\
([a_1,b_1],\ldots,[a_4,b_4]) & \mapsto & [a_1 a_2 b_3 b_4,a_1 a_3 b_2 b_4,\ldots,a_3 a_4 b_1 b_2]
\end{matrix}
\end{equation*}
and the quotient map for the natural action of $\cS_4$ on the image of the above map.  
\item
\emph{$\sigma$ is stabilized by $\wt{\lambda}_3$:} $\sigma=u^2_0 u^2_1 F(v_0,v_1)$ in suitable coordinates $([u_0,u_1],[v_0,v_1])$, and $F(v_0,v_1)$ is polystable for the action of $\PGL(2)$ on $\PP(\CC[v_0,v_1]_4)$. Equivalently $C=2L+2L'+R_1+\ldots+R_4$, where $L,L'$ are distinct lines in the same ruling, $R_1,\ldots,R_4$ are lines in the other ruling, and either $R_1,\ldots,R_4$ are distinct, or $R_1=R_2\not= R_3=R_4$.
The moduli space is identified with that of binary quartics, i.e.~$\PP^1$ (we map  $V(\sigma)$ to the moduli  point corresponding to $[F]$).  
\end{enumerate}
\end{enumerate}
The remark below will be useful later on.
\begin{remark}\label{rmk:mult4}
Suppose that  $C\in | \cO_{\PP^1}(4)\boxtimes\cO_{\PP^1}(4)|$  has a point $p$ of multiplicity $4$. Then, in the closure of the orbit $\Aut(\PP^1\times\PP^1)C$  there exists $C^{*}=C_1+\ldots+C_4$ where  $C_1,\ldots,C_4$ are conics through distinct points $q_1,q_2$   not on a line - see Item~(a) of~\Ref{lmm}{git44}.  Moreover the tangent cone $\cC_p(C)$ is isomorphic to the tangent cone $\cC_{q_i}(C^{*})$ at either one of the 
(multiplicity $4$) points $q_1,q_2$ of $C^{*}$. 
\end{remark}
\end{proposition}
\subsection{Hodge-theoretic stratification of $\gM$}
We quickly recall some notions which have been treated in~\cite{log2}.
 \begin{definition}\label{dfn:definsignificant}
A reduced  (not necessarily irreducible)  projective surface $X_0$ is a \emph{degeneration of $K3$ surfaces} if it is the central fiber of a flat proper family $\sX/B$ over a pointed smooth curve $(B,0)$  such that $\omega_{\sX/B}\equiv 0$ and the general fiber $X_b$ is a smooth $K3$ surface. 
If $p\in X_0$, then  $X_0$ has an \emph{insignificant limit singularity} at $p$ if it has a semi-log-canonical singularity at $p$.
 \end{definition}
The above definition ties in with the terminology of Shah. More precisely, the list of singularities baptized as \emph{insignificant limit singularities} by Shah~\cite{shahinsignificant} coincides with the list of Gorenstein slc singularities (see~\cite{sbnef,ksb}). For a degeneration of $K3$ surfaces, the Gorenstein assumption is automatic. We recall (see Theorem~4.21 of~\cite{ksb})  that a Gorenstein surface singularity $(X,p)$ with the property that $X\setminus\{p\}$ is semi-smooth (i.e.~either smooth, normal crossings with two components, or a pinch point) is  semi-log-canonical if and only if it is semi-canonical (see Definition~4.17 ibid.), simple elliptic, cusp or a degenerate cusp. 
\begin{definition}\label{dfn:cavalier}
Let $X_0$ be a degeneration of $K3$ surfaces, and let $p\in X_0$ be an  insignificant limit singularity. Then  
\begin{itemize}
\item[i)]  $X_0$ is of Type I at $p$ if $(X_0,p)$ is an ADE singularity (this includes smooth points).
\item[ii)]   $X_0$ is of Type II at $p$ if  $(X_0,p)$ is  simple elliptic, locally normal crossings with exactly two irreducible components containing $p$, or  a pinch point. 
\item[iii)] $X_0$ is of Type III  at $p$ if   $(X_0,p)$ is either a cusp or a degenerate cusp.
\end{itemize}
\end{definition}
\begin{definition}
Let $X_0$ be a degeneration of $K3$ surfaces. 
\begin{enumerate}
\item 
If $X_0$  has  insignificant limit singularities, then  
\begin{enumerate}
\item  $X_0$ is of Type I if all its points are of Type I.
\item   $X_0$ is of Type II if all its points are of Type I and II, it does have points of Type II.
\item $X_0$ is of Type III,  if all its points are of Type I, II or III, and  it does have points of Type III. 
\end{enumerate}
\item
 $X_0$ has \emph{Type IV} if it has  significant  limit  singularities (i.e.~there exists $p\in X_0$ such that $(X_0,p)$ is  \emph{not} an   insignificant limit singularity).
\end{enumerate}
\end{definition}

\begin{remark}\label{rmk:jinv}
We note that the $1$-dimensional components in the singular locus of a Type II degeneration $X_0$ of $K3$s are either  smooth elliptic with no pinch points, or rational with $4$ pinch points. Also, recall that the resolution of a simple elliptic singularity is an elliptic curve (of negative self-intersection). Thus, one sees that in all cases, for Type II degeneration $X_0$ of $K3$ surfaces, there is an associated $j$-invariant. Typically, $X_0$ has a single Type II singularity (i.e. simple elliptic, or elliptic double curve, or rational double curve with $4$ pinch points), but even if there are multiple Type II singularities, the $j$-invariant for the various singularities coincides. 
\end{remark}
The above definitions are of interest to us because they are related to the period map $\gp\colon\gM\dra\sF^{*}$.
Before explaining this, we introduce one more piece of terminology. Let $C\in|\cO_{\PP^1}(4)\boxtimes\cO_{\PP^1}(4)|$; we let $X_C\to \PP^1\times\PP^1$ be the double cover ramified over $C$. Let  $p\in C$, and let $
\wt{p}\in X_C$ be the unique point lying over $p$.
We say that \emph{$C$ has an insignificant limit singularity at $p$} if $(X_C,\wt{p})$ is an insignificant limit singularity, and if that holds then $C$ has Type I, II or III at $p$ according to the type of $(X_C,\wt{p})$. If   $(X_C,\wt{p})$ is a significant limit singularity, we say that $C$ has Type IV at $p$. Similarly $C$ has insignificant limit singularities if all of its points are insignificant limit singularities, and if that is the case the Type of $C$ is that of $X_C$.

Let $\gM^I,\gM^{II},\gM^{III},\gM^{IV}\subset\gM$ be the subsets of points represented by polystable curves $C\in|\cO_{\PP^1}(4)\boxtimes\cO_{\PP^1}(4)|$   of Type I,   II, III  and  IV  respectively. On the other side, let $\sF^I:=\sF$, let  $\sF^{II}$ be the union of the Type II boundary components of $\sF^{*}$, and let $\sF^{III}$ be the union of the Type III boundary components of 
$\sF^{*}$. The proof Proposition 3.16 of~\cite{log2} has a straightforward extension   to our case, and gives the following result.
 \begin{proposition}\label{prp:est3}
 The  period map $\gp\colon\gM\dra\sF^{*}$  is regular away from $\gM^{IV}$, and 
\begin{equation*}
\gp(\gM^I)\subset\sF,\qquad \gp(\gM^{II})\subset\sF^{II},\qquad  \gp(\gM^{III})\subset\sF^{III}.
\end{equation*}
 \end{proposition}
\subsection{The  components of $\gM^{II}$}%

\begin{proposition}\label{prp:TypeIIShah}
The irreducible components of $\gM^{II}$ are the following: 
\begin{itemize}
\item[i)] $\gM^{II}_{D_8\oplus E_8}$, the set parametrizing  stable  reduced curves $C$ with a singularity of type $\widetilde E_8$ at a point $p$ (by~\Ref{prp}{git44} the tangent cone at $p$ is  $3L$ with $L$ \emph{not} the tangent space to a line through $p$). The dimension of $\gM^{II}_{D_8\oplus E_8}$ is $9$.
\item[ii)] $\gM^{II}_{D_{12}\oplus D_4}$,  the set parametrizing  stable   divisors $C=2C_0+D$ where $C_0$ is a smooth conic, and such that the residual curve $D$ intersects $C_0$ transversely. The dimension of $\gM^{II}_{D_{12}\oplus D_4}$ is $5$. 
\item[iii)] $\gM^{II}_{A_{15}\oplus D_1}$,  the set parametrizing  stable   divisors $C=2E$, where $E$ is smooth. The dimension of  $\gM^{II}_{A_{15}\oplus D_1}$  is $2$.
\item[iv)] $\gM^{II}_{D_{16}^+}$,  the set parametrizing   stable  divisors $C=2C_0+L_1+L_2$, where $C_0$ is a twisted cubic, and $L_1$,  $L_2$ are distinct lines intersecting $C_0$ transversally. The dimension of  $\gM^{II}_{D_{16}^+}$  is $1$.
\item[v)] $\gM^{II}_{(E_8)^2}$,   the set parametrizing   polystable  divisors as in Item~(2a) of~\Ref{prp}{git44} such that there are two $\widetilde E_8$ singularities\footnote{Recall that a surface singularity is $\wt{E}_r$ if the exceptional divisor of a minimal desingularization is a (smooth) elliptic curve of self intersection $9-r$.} (i.e.~$a_1b_2\not=0\not= a_2b_1$). The dimension of  
$\gM^{II}_{(E_8)^2}$ is $1$. 
\item[vi)] $\gM^{II}_{(E_7)^2\oplus D_2}$,  the set parametrizing   polystable  divisors as in Item~(2b) of~\Ref{prp}{git44} such that there are two $\widetilde E_7$ singularities, i.e.~the divisor is reduced. The dimension of  $\gM^{II}_{(E_7)^2\oplus D_2}$ is $3$. 
\item[vii)]  $\gM^{II}_{(D_8)^2}$,  the set parametrizing   polystable  divisors as in Item~(2c) of~\Ref{prp}{git44} such that the lines $R_1,\ldots,R_4$ are distinct. The dimension of  $\gM^{II}_{(D_8)^2}$ is $1$. 
\end{itemize}
\end{proposition}
\begin{proof}
This follows from Shah \cite[Theorem 4.8, Type II]{shah4}. The cases A-iii and A-iv are omitted in Shah, but it is clear they occur - for instance they can be viewed as limits of semistable quartics (degenerating to the double quadric), see \cite[Theorem 2.4-A-II-ii]{shah4}.

We will discuss in more detail the Type II strata (e.g. normal forms) in \Ref{sec}{chow} (esp. \S\ref{GITtype2}). In particular, the dimensions will be seen to be as in the statement of the Proposition. For the moment, we just sketch the argument for the largest possible stratum: II-A-i. Via a local computation, one can see that the space of $(4,4)$ curves versaly unfolds the $\widetilde E_8$ singularity. Thus, the codimension in the moduli is $10$ (the Milnor-Tjurina number for $\widetilde E_8$). However, there is one dimension for the equisingular deformations (corresponding to varying the modulus of the elliptic curve). Giving finally, dimension $18-10+1=9$ for the $\widetilde E_8$ stratum. 

The labels are chosen so that they match the labels from \Ref{thm}{thmbbstrata}. The justification for this is given by \cite[Prop. 7.11, Def. 7.7]{log2} (and it is based on 
 a heuristic of Friedman \cite[Sect. 5]{friedmanannals}). We will revisit the issue in \Ref{sec}{chow}.
\end{proof}

\begin{remark}[{Type III}] We will not discuss in detail the stratification of the Type III locus $\gM^{III}$. We will only note that it corresponds to degenerations of the Type II cases for which the $j$-invariant associated by \Ref{rmk}{jinv} becomes $\infty$ (e.g. simple elliptic singularity degenerate to cusp singularity, or some of the $4$ pinch points come together, but at worst with multiplicity $2$). There are $5$ strata identified by Shah \cite[Thm. 4.8]{shah4} (labeled A-III-i, A-III-ii, B-III-i, B-III-ii, and B-III-iii respectively). We note however, that there is a stratum missing in Shah's analysis, namely, the case of a double twisted cubic together with two tangent lines. This stratum is a specialization of the Type II case labeled $D_{16}^+$ above (again missing from Shah's list). We will label this case $III_b$; explicitly:
\begin{equation}
(III_b) \ :\  V(x_0x_3-x_1x_2,x_0x_2^3+2x_1^2x_2^2+x_1^3x_3).
\end{equation}
Similarly, the case where $C$ consists of $4$ double lines forming a cycle (the case B-III-iii in \cite[Thm. 4.8]{shah}) will be labeled $III_a$; explicitly, 
\begin{equation}
(III_a) \ : \ V(x_0x_3-x_1x_2,x_1^2x_2^2).
\end{equation}
Clearly, $III_a$ and $III_b$ are isolated points in $\gM^{III}$. It is not hard to see that the closure of any other Type III stratum (or similarly Type II) contains one of those two points. In other words, $III_a$ and $III_b$ are the deepest strata in $\gM^{II}\cup \gM^{III}$. The labels are chosen such that the adjacency of GIT Type II and III strata reflects the adjacency of Type II and III strata in the Baily-Borel compactification (see \eqref{incidencebb}). (For a typical picture of the behavior of the GIT vs Baily-Borel Type II and III strata see \cite[Figure 2, p. 234]{k3pairs}.)
\end{remark}

\section{Stratification of $\gM^{IV}$}\label{sec:sec-stratify}
As discussed above, the indeterminacy locus of the period map $\gM\dashrightarrow \sF^*$ is contained in the Type IV locus\footnote{A  posteriori  the indeterminacy locus of the period map  is \emph{equal} to  $\gM^{IV}$}. The main result of our paper is to decompose this period map into simple flips, and this will be achieved by a variation of GIT quotients as discussed in the subsequent sections. The purpose of this section is to define a finite (increasing) stratification $\{W_k\}_k$ (with $k$ standing for the dimension) that we will correspond to the centers of these flips.  The definitions of $W_k$ are inspired by our previous analysis \cite{log2} for quartics (which in turn is a refinement of Shah \cite{shah4}). The reader can ignore all this background information, and just regard $W_k$ as a natural stratification of the Type IV stratum in terms of the complexity of the singularities (the lower the $k$ the worse the singularity). Results of Arnold et al.~\cite{arnold1} play an essential role here, and will be reviewed below. 

\subsection{Singularity types}\label{singclassification}
Let $C\in|\cO_{\PP^1}(4)\boxtimes\cO_{\PP^1}(4)|$ be polystable, with isolated singularities, and let $X_C\to\PP^1\times\PP^1~$ be the double cover ramified over $C$. 
Suppose that $\mult_p(C)\le 3$ for all points $p$ (this condition holds for all stable $C$ by~\Ref{prp}{git44}). If $C$ does not have singular points with consecutive triple points, then $X_C$ has ADE singularities, and hence $[C]\in\gM^I$, in particular the period map is regular at $[C]$.  If instead $C$ does have consecutive triple points at $p$, then the 
initial germ of a defining equation of $C$ at $p$ is equal to $x^3$, for a suitable local parameter $x$. The isomorphism classes of such singularities have been classified, see~\cite[Ch. 16]{arnold1}. We recall the classification, and how to recognize the isomorphism class to which a given singularity belongs. Most of the isomorphism classes of such singularities are of Type IV, i.e.~the corresponding double cover has significant limit singularities. They define (together with certain  non isolated singularities) a stratification of $\gM^{IV}$ which determines the sequence of flips that resolve the period map $\gp\colon\gM\dra\sF^{*}$. The stratification is defined in~\Ref{subsec}{eccostrati}.

\begin{theorem}[{Arnold et al.~\cite[Ch. 16]{arnold1}}]\label{thm:zoo}
Let $f\in\cO_{\CC^2,0}$ be the germ of an analytic function of two variables  in a neighborhood of the origin. Suppose that $f$ has  an isolated singularity at $(0,0)$, of multiplicity $3$ with tangent cone a triple line.
 Then there exist analytic coordinates $(x,y)$ in a neighborhood of $0$ (centered at $0$) and a decomposition $f(x,y)=u(x,y)\cdot g(x,y)$, where $u(x,y)$ is a unit and $g(x,y)$ is one (and only one) of the functions  appearing in the first column of the table below. 
 \begin{table}[htb!]
\renewcommand\arraystretch{1.5}
\begin{tabular}{l| l | l|l | l  | l  }
Normal form & Leading Term & $wt(x)$ &  $wt(y)$ & Name & Type \\
\hline
$x^3+y^{3k+1}+{\bf a}xy^{2k+1}$& $x^3+y^{3k+1}$ & $\frac{1}{3}$&$\frac{1}{3k+1}$   &$E_{6k}$ & $I$ if $k=1$, $IV$  if $k\ge 2$ \\
$x^3+xy^{2k+1}+{\bf a}y^{3k+2}$& $x^3+xy^{2k+1}$ & $\frac{1}{3}$ &$\frac{2}{3(2k+1)}$  &$E_{6k+1}$  &  $I$ if $k=1$, $IV$  if $k\ge 2$ \\
$x^3+y^{3k+2}+{\bf a}xy^{2k+2}$& $x^3+y^{3k+2}$ & $\frac{1}{3}$&$\frac{1}{3k+2}$   &$E_{6k+2}$ &  $I$ if $k=1$, $IV$  if $k\ge 2$ \\
$x^3+b x^2y^k+y^{3k}+{\bf c}xy^{2k+1}$ & $x^3+b x^2y^k+y^{3k}$ & $\frac{1}{3}$&$\frac{1}{3k}$   & $J_{k,0}$ & $II$ if $k=2$, $IV$  if $k\ge 3$ \\
\hdashline
$x^3+x^2y^k+{\bf a}x^{3k+p}$ & $x^3+x^2y^k$ & $\frac{1}{3}$&$\frac{1}{3k}$  & $J_{k,p}$ &  $III$  if $k=2$, $IV$  if $k\ge 3$ \\
\hline
\end{tabular}
\vspace{0.2cm}
\caption{}\label{tavola}
\end{table}
In the first three rows of the table $k\ge 1$, in the last two rows $k\ge 2$, and in the last row $p>0$. Moreover
\begin{equation*}
{\bf a}:=a_0+\ldots+a_{k-2}y^{k-2},\quad {\bf c}:=a_0+\ldots+a_{k-3}y^{k-3},\quad    4b^3+27\neq 0,
\end{equation*}
(${\bf a}=0$ if $k=1$, and ${\bf c}=0$ if $k=2$) and in the last row $a_0\not=0$.
\end{theorem}
\begin{proof}
This is obtained by putting together  Theorems $6_k,\dots, 12_k$ in \cite[\S16.2]{arnold1}.
\end{proof} 
We explain the r\^ole of the weights appearing in the table above. First notice that the monomials in the leading term have weight $1$, and the remaining monomials in the normal form have weight strictly greater than $1$. Moreover, with the exception of singularities $J_{k,p}$ with $p>0$, the leading term has an isolated critical point at the origin. Thus, with the exception of singularities $J_{k,p}$ with $p>0$, the singularities in~\Ref{thm}{zoo} are  semiquasihomogeneous. 
\begin{theorem}[{Arnold et al.~\cite[Ch. 16]{arnold1}}]\label{thm:recognize}
Let $f(x,y)\in\cO_{\CC^2,0}$ be the germ of an analytic function of two variables  in a neighborhood of the origin, and suppose that $f$ has  an isolated singularity at $(0,0)$. Assign weights to $x$ and $y$ according to a chosen row of the table in~\Ref{thm}{zoo}. If $f=f_0+f_1$, where $f_0$ is the leading term of the chosen row and every monomial appearing in $f_1$ has weight strictly greater than $1$, then there exist  
 analytic coordinates  in a neighborhood of $0$ (which we denote again by $x,y$) such that $f=u\cdot g$, where  $u$ is a unit and $g$ is the normal form in the chosen row.
\end{theorem}
\begin{remark}
Singularities often have more than one name. Here we note that  $J_{2,p}$ is also denoted $T_{2,3,6+p}$ (cusp singularity), and $J_{2,0}=T_{2,3,6}$ is also denoted $\widetilde{E_8}$ (simple elliptic singularity).
\end{remark}
We will also make use of the following terminology for certain non isolated singularities.
\begin{definition}\label{dfn:kappainf}
The germ $(C,p)$ of a one dimensional singularity is of \emph{Type $J_{k,\infty}$} if it is isomorphic to the germ at $(0,0)$  of the planar singularity defined by 
$x^3+x^2y^k=0$.
\end{definition}
\subsection{The stratification}\label{subsec:eccostrati}
\begin{proposition}\label{prp:4unico}
Let $C\in|\cO_{\PP^1}(4)\boxtimes\cO_{\PP^1}(4)|$ be stable. Then there is at most one Type IV point of $C$. 
\end{proposition}
\begin{proof}
If $C$ is reduced, this follows from~\Ref{prp}{git44} and Lemma~4.7 of~\cite{shah4}. Suppose that $C$ is non reduced, and let 
$C_0\in|\cO_{\PP^1}(a)\boxtimes\cO_{\PP^1}(b)|$ be a component of $C$ of multiplicity at least $2$ (and hence of multiplicity $2$, 
by~\Ref{prp}{git44}). Then  $a>0$ and $b>0$  by~\Ref{prp}{git44}. It follows that there is a unique such $C_0$. Let $C=2C_0+D$. Then  $D$ is reduced, it has no component equal to $C_0$,  it has no point of multiplicity greater than $3$, and it is smooth at each point of intersection with $C_0$. Thus, if $p$ is a point of Type IV  of $C$, it belongs to $C_0$ and it is of Type $J_{3,\infty}$ or $J_{4,\infty}$, 
i.e.~$\mult_{p}(C_0\cap D)=3$ or $\mult_{p}(C_0\cap D)=4$ respectively. By B\'ezout, there is at most one such point.
\end{proof}
The definition below makes sense by the above proposition.
\begin{definition}
\begin{enumerate}
\item 
For each   type of isolated singularity $T$ appearing in~\Ref{thm}{zoo} (i.e.~isolated singularities of multiplicity $3$ with tangent cone of multiplicity $3$) \emph{of Type IV}, let $\gM^{IV}_T\subset\gM^{IV}$ be the set parametrizing stable curves which have a point of Type $T$.
\item 
Let $\gM^{IV}_{J_{k,\infty}}\subset\gM^{IV}$ be the set parametrizing stable curves which have a point of Type $J_{k,\infty}$. 
\item 
Let $\gM^{IV}_{J_{k,+}}:=\gM^{IV}_{J_{k,\infty}}\sqcup\coprod_{r>0}  \gM^{IV}_{J_{k,r}}$.
\item 
 Let $\gM^{IV}_{(3,1)}\subset\gM^{IV}$ be the set parametrizing curves $3C_0+C_1$, where $C_0,C_1\in|\cO_{\PP^1}(1)\boxtimes \cO_{\PP^1}(1)|$ are distinct,  and $C_0$ is smooth.
\item 
 Let $\gM^{IV}_{(4)}\subset\gM^{IV}$ be the  singleton whose unique point  corresponds to $4C_0$, where $C_0\in|\cO_{\PP^1}(1)\boxtimes \cO_{\PP^1}(1)|$ is  smooth.
\end{enumerate}
\end{definition}
 We have
\begin{equation}
\gM^{IV}=\coprod_{T}\gM^{IV}_T\sqcup \gM^{IV}_{(3,1)}\sqcup\gM^{IV}_{(4)}.
\end{equation}
This follows immediately from the results proved in the last two sections. 

Our next task is to describe explicitly the curves in the  subsets $\gM^{IV}_T$. This will allow us to determine which subsets are non empty, and  will also be useful in the GIT analysis later on.
\begin{lemma}[Lemma 4.6 in~\cite{shah4}]\label{lmm:norm3}
Let $C\subset\PP^3$ be a $(2,4)$ complete intersection curve,  let $Q$ be the quadric containing it, and let $p\in C$ be a point \emph{not} belonging to the singular locus of  $Q$. Then the following are equivalent:
\begin{enumerate}
\item 
There exist  homogeneous coordinates $[x_0,\ldots,x_3]$ such that $p=[1,0,0,0]$,  and 
\begin{eqnarray}
Q & = & V(x_0x_2+x_1^2+ax_3^2), \\
C & = &  V(x_0x_2+x_1^2+ax_3^2, x_0x_3^3+x_1^2g_2(x_2,x_3)+x_1g_3(x_2,x_3)+g_4(x_2,x_3)). \label{norm3}
\end{eqnarray}
\item 
 $C$ has consecutive triple points at $p$, with tangent cone $3L$, where $L$ is  \emph{not} the tangent space to  a line of $Q$. 
\end{enumerate}
\end{lemma}
\begin{proof}
Let $(x,y,z)$ be the affine coordinates, centered at $p$, given by  $x=\frac{x_1}{x_0}$, $y=\frac{x_3}{x_0}$, and $z=\frac{x_2}{x_0}$ (this agrees with   Shah's notation). 

First we check that if~(1) holds then~(2) holds.  Let    $C_{x_0}=C\setminus V(x_0)$; thus $C_{x_0}$ is an affine neighborhood of $p$ in $C$.  Projection to the $(x,y)$ plane defines an isomorphism between $C_{x_0}$ and the plane affine curve with equation
\begin{equation}\label{affeq}
y^3+x^2 g_2(-x^2-ay^2,y)+x g_3(-x^2-ay^2,y)+g_4(-x^2-ay^2,y)=0.
\end{equation}
Since $p$ is mapped to $(0,0)$, it follows that $C$ has consecutive triple points at $p$ (see~\Ref{rmk}{constrip}), with tangent cone $3V(y,z)$. The line(s) in $Q$ containing $p$  have equations $V(x\pm \sqrt{a}i y,z)$, hence $V(y,z)$ is  \emph{not} the tangent space to  a line of $Q$. 

Next, let us prove that if~(2) holds then~(1) holds. By our hypothesis on the tangent cone of $C$ at $p$, the quadric is either smooth or it has a unique singular point. Thus  we may choose homogeneous coordinates $[x_0,\ldots,x_3]$ such that $p=[1,0,0,0]$,  $Q  =  V(x_0x_2+x_1^2+ax_3^2)$, and the tangent cone to $C$ at $p$ is equal to $3L$ where $L$ is tangent to the line $V(x_2,x_3)$ 
(again use  our hypothesis on the tangent cone of $C$ at $p$). Thus $C_{x_0}$ has equations
\begin{equation*}
z+x^2+ay^2=0,\quad f(x,y,z)=0,
\end{equation*}
where $f$ is a quartic polynomial. 
Notice that we are free to modify $f$ by adding any polynomial $q(x,y,z)\cdot (z+x^2+ay^2)$, where $q$ has degree at most $2$. 
Now,  projecting to the $(x,y)$ plane (as above), we get an isomorphism between   $C_{x_0}$ and the plane affine curve with equation
\begin{equation*}
f(x,y,-x^2-ay^2)=0,
\end{equation*}
such that $p$ is mapped to $(0,0)$. By invoking~\Ref{rmk}{constrip}, one shows that there exists a choice of polynomial  $q(x,y,z)$ of degree at most $2$ such that 
\begin{equation}\label{newyear}
f+q(x,y,z)\cdot (z+x^2+ay^2)=y^3+y^2 p_2(x,y)+z q_3(x,y,z),\quad q_3(1,0,0)=0,
\end{equation}
where $p_2$, $q_3$ are homogeneous of degrees $2$ and $3$ respectively. Thus the right hand side of~\eqref{newyear} is a quartic polynomial of degree at most $2$ in $x$, and hence its homogenization  may be rewritten as the quartic polynomial appearing in~\eqref{norm3}. 
\end{proof}
The observation below will be handy when computing the dimensions of the strata $\gM^{IV}_T$.
\begin{remark}\label{rmk:qunico}
An easy argument shows that there is \emph{unique} polynomial $q$ such that~\eqref{newyear} holds. 
\end{remark}
\begin{proposition}\label{prp:norm3}
Let $C\subset\PP^3$ be a $(2,4)$ complete intersection curve,  let $Q$ be the quadric containing it, and let $p\in C$ be a point \emph{not} belonging to the singular locus of  $Q$. Suppose that $C,Q,p$ satisfy one of the (equivalent) Items~(1) or~(2) of~\Ref{lmm}{norm3}. Retain the notation of the quoted  lemma, and let $g_d(x_2,x_3)=\sum_{i+j=d} g^{i,j}_d x_2^i x_3^j$. 
Then the following hold:
\begin{enumerate}
\item If $g_2\neq 0$, then $(C,p)$ is a  $J_{2,r}$ singularity, where $r\in\NN\cup\{\infty\}$. 
\item If $g_2=0$, and $g_3^{3,0}\not=0$, then $(C,p)$ is an  $E_{12}$  singularity.
\item If $g_2=0$, $g_3^{3,0}=0$, and $g_3^{2,1}\not=0$,  then $(C,p)$ is an  $E_{13}$  singularity.
\item If $g_2=0$, $g_3^{3,0}=g_3^{2,1}=0$, and $g_4^{4,0}\not=0$,   then $(C,p)$ is an  $E_{14}$  singularity.
\item If $g_2=0$, $g_3^{3,0}=g_3^{2,1}=g_4^{4,0}=0$,  and  $g_4^{3,1}((g_3^{1,2})^2+4g_4^{3,1})\not=0$,  then $(C,p)$ is a  $J_{3,0}$  singularity.
\item If $g_2=0$, $g_3^{3,0}=g_3^{2,1}=g_4^{4,0}=g_4^{3,1}(g_3^{1,2}\cdot g_3^{1,2}+4g_4^{3,1})=0$,  and $(g_3^{1,2},g_4^{3,1})\not=(0,0)$,   then $(C,p)$ is a  $J_{3,r}$  singularity, for some $r>0$ (possibly 
$r=\infty$).
\item If  $g_2=0$, $g_3^{3,0}=g_3^{2,1}=g_4^{4,0}=g_3^{1,2}=g_4^{3,1}=0$,   and 
$g_4^{2,2}\not=0$, then $(C,p)$ is a  $J_{4,\infty}$  singularity.
\item If  $g_2=0$, $g_3^{3,0}=g_3^{2,1}=g_4^{4,0}=g_3^{1,2}=g_4^{3,1}=g_4^{2,2}=0$,  then  $C=3C_0+C_1$, where $C_0$  is a smooth conic.
\end{enumerate}
\end{proposition}
\begin{proof}
Recall that the germ of $C$ at $p$ is isomorphic to germ at $(0,0)$ of the affine plane curve with equation given in~\eqref{affeq}. 
Items~(1) - (5) follow from~\Ref{thm}{recognize}. In fact, assign weights to $x,y$ according to the table in~\Ref{thm}{zoo}, with the r\^oles of $x$ and $y$ exchanged, i.e.~$wt(x)=1/6$ when proving Item~(1), 
$wt(x)=1/7$ when proving Item~(2), 
$wt(x)=2/15$  when proving  Item~(3), $wt(x)=1/8$  when proving  Item~(4),  $wt(x)=1/9$  when proving  Item~(5), and $wt(y)=1/3$ in all  cases. Then Item~(1) holds because the leading term of the equation is  $y^3+x^2 g_2(-x^2,y)$ (in order to recognize the leading term one  needs to pass to  analytic coordinates $(x,y+\alpha x^2)$ for a suitable choice of $\alpha$),   Item~(2) holds because the leading term of the equation is $y^3-g_3^{3,0}x^7$,   Item~(3) holds because the leading term of the equation is $y^3+g_3^{2,1}x^5 y$,  Item~(4) holds because the leading term of the equation is $y^3+g_4^{4,0}x^8$, and Item~(5) holds because the leading term of the equation is $y^3-g_3^{1,2}x^3 y^2-g_4^{3,1}x^6y$, and there exist non zero distinct $\alpha,\beta$ such that 
 $y^3-g_3^{1,2}x^3 y^2-g_4^{3,1}x^6y=y(y+\alpha x^3)(y+\beta x^3)$ if and only if  $g_4^{3,1}((g_3^{1,2})^2+4g_4^{3,1})\not=0$. 

The proof of~(6), (7) and (8) is totally elementary.
\end{proof}
\begin{corollary}\label{crl:norm3}
Let $T\in\{E_{12},E_{13},E_{14},J_{3,0},J_{3,+},J_{4,\infty},(3,1),(4,0)\}$. Then $\gM^{IV}_T$ is a (non empty) irreducible locally closed subset of $\gM^{IV}$, of dimension  given   below
\begin{equation}\label{strati4}
\begin{tabular}{c|c|c|c|c|c|c|c | c }
$\gM^{IV}_T$ &  $\gM^{IV}_{E_{12}}$ & $\gM^{IV}_{E_{13}}$ &  $\gM^{IV}_{E_{14}}$  &  $\gM^{IV}_{J_{3,0}}$  &  $\gM^{IV}_{J_{3,+}}$ &  $\gM^{IV}_{J_{4,\infty}}$ & 
$\gM^{IV}_{(3,1)}$ &  $\gM^{IV}_{(4)}$ \\
\hline
$\dim \gM^{IV}_T$  & $7$ & $6$  & $5$  & $4$  &  $3$  &  $2$  & $1$ & $0$ 
\end{tabular}
\end{equation}
Moreover
\begin{equation}\label{decomp4}
\gM^{IV}=\gM^{IV}_{E_{12}}\sqcup \gM^{IV}_{E_{13}}\sqcup \gM^{IV}_{E_{14}}\sqcup \gM^{IV}_{J_{3,0}}\sqcup 
 \gM^{IV}_{J_{3,+}}\sqcup \gM^{IV}_{J_{4,\infty}}\sqcup \gM^{IV}_{(3,1)}\sqcup\gM^{IV}_{(4)},
\end{equation}
and the closure of  $\gM^{IV}_T$ is the union of  $\gM^{IV}_T$  and the strata   $\gM^{IV}_{T'}$    to the right 
 of  $\gM^{IV}_T$    in~\eqref{decomp4}.  
\end{corollary}
\begin{proof}
Let $T\in\{E_{12},E_{13},E_{14},J_{3,0},J_{3,+},J_{4,\infty}\}$, and let $[C]\in \gM^{IV}_T$, with $C$ polystable. Then $C$ is stable by~\Ref{prp}{git44} and, by the same proposition, the hypotheses of~\Ref{prp}{norm3} are satisfied, where $Q$ is the smooth quadric containing $C$, and $p\in C$ is the unique point of Type IV (see~\Ref{prp}{4unico}). By~\Ref{prp}{norm3} it follows that 
$ \gM^{IV}_T$ is non empty, irreducible and locally closed. It is elementary that $\gM^{IV}_{(3,1)}$ and $\gM^{IV}_{(4)}$ are irreducible, locally closed. The statement about the closure   follows from~\Ref{prp}{norm3}, and also the decomposition in~\eqref{decomp4}. 

It remains to prove that the  dimensions are given by~\eqref{strati4}. In the case of $\gM^{IV}_{(3,1)}$ and $\gM^{IV}_{(4)}$, the computation is straightforward. It remains to deal with $ \gM^{IV}_T$ for 
$T\in\{E_{12},E_{13},E_{14},J_{3,0},J_{3,+},J_{4,\infty}\}$. We may assume that $C$ is contained in the smooth quadric $x_0x_2+x_1^2-x_3^2$, i.e.~we set $a=-1$ in~\Ref{lmm}{norm3}. By~\Ref{lmm}{norm3} and~\Ref{prp}{norm3}, $C$ is equivalent (under $\Aut(Q)$) to a curve whose equation is given by~\eqref{norm3}, with $g_2=0$ and  such that $g_3$ and $g_4$ satisfy the conditions in the Item of~\Ref{prp}{norm3} corresponding to the chosen $T$. Let $\cS_T\subset|\cO_Q(4)|$ be the subset of such curves.
Since $C$ is stable, it follows that the dimension of  $ \gM^{IV}_T$ is equal to the difference between the dimension of $\cS_T$  and the dimension of the subgroup of $G_T<\Aut(Q)$ mapping $\cS_T$ to itself. The dimension of $\cS_T$ is easily computed (recall~\Ref{rmk}{qunico}): it is $9$ if $T=E_{12}$, and it decreases by  $1$ each time we move one step to the right. The  subgroup  $G_T<\Aut(Q)$ mapping $\cS_T$ to itself is contained in the subgroup of automorphisms $\phi$ stabilizing $p=[1,0,0,0]\in Q$, and such that the differential $d\phi(p)$  maps to itself  the tangent line $V(x_2,x_3)$. Thus $\dim G_T\le 3$. In fact  a computation shows that the connected component of the identity of $G_T$ is equal to  the subgroup of $\PGL(4)$ given by matrices
\begin{equation*}
\left(\begin{array}{rrrr}
1 & 2\beta &  -\beta^2  & 0 \\
0 & \alpha & -\alpha\beta  & 0 \\
0 & 0 & \alpha^2  & 0 \\
0 & 0  & 0 & \alpha \\
\end{array}\right)
\end{equation*}
with $\alpha\not=0$. (Isomorphic to the subgroup of $2\times 2$ matrices $\left(\begin{array}{rr}
1 & \beta \\
0 & \alpha \\
\end{array}\right)$, i.e.~the group of automorphisms of an affine  line.) 
Hence  $\dim G_T=2$, and this gives the dimensions in~\eqref{strati4}.
\end{proof}
At this point, we can give the key definition of the  $W$-stratification.
\begin{definition}\label{dfn:eccow}
For $d\in\{0,1,2,4,5,6,7\}$ (no misprint: $d=3$ \emph{is} missing), we let $W_d\subset\gM^{IV}$ be the union of all the strata $\gM^{IV}_T$  of dimension at most $d$. 
\end{definition}
\begin{remark} 
 The  stratum $ \gM^{IV}_{J_{3,+}}$ is skipped in the definition of the $W$-stratification because it  is  flipped together with  $ \gM^{IV}_{J_{3,0}}$. 
 \end{remark}
By~\Ref{crl}{norm3} we have a ladder of irreducible closed subsets indexed by dimension:
\begin{equation}\label{wstrata}
W_0\subset W_1\subset W_2\subset W_4 \subset W_5 \subset W_6 \subset W_7=\gM^{IV}\subset \gM.
\end{equation}
This is the counterpart of the stratification $Z^k$ (indexed by codimension) in $\sF^*$ (see \eqref{zetahyp}).


\section{GIT for $(2,4)$ complete intersections in $\bP^3$}\label{sec:shyper}
 Let $U$ be the parameter space for $(2,4)$ complete intersection curves in 
$\PP^3$, with the natural action of $\SL(4)$. The main tool in this paper is a variation of GIT quotients for $U$. Since $U$ is not projective, we will consider  the closure of $U$ in the Hilbert scheme, call it $\Hilb_{(2,4)}$. In order to define a GIT quotient  of  $\Hilb_{(2,4)}$ modulo the natural action of  $\SL(4)$ we must choose an $\SL(4)$-linearized  line bundle.
For large $m$, we have the (naturally linearized) ample Plucker line bundle $L_m$ corresponding to the $m$-th Hilbert point (i.e.~the fiber over a scheme  $C$ is equal to the determinant of $H^0(C,\cO_C(m))$); we let $\Hilb_{(2,4)}\gquot_{L_m} \SL(4)$ be the corresponding quotient. 

 There is a regular map $c\colon U\to\Chow$ to the Chow variety parametrizing $1$-dimensional cycles in $\PP^3$. Let $\Chow_{(2,4)}$ be the closure of the image of $c$, and let $L_{\infty}$ be the restriction to $\Chow_{(2,4)}$ of the natural polarization of the Chow variety. As suggested by the notation, for $m\to\infty$ the polarization $L_m$ approaches the pull-back of $L_{\infty}$, and hence  the quotient $\Hilb_{(2,4)}\gquot_{L_m} \SL(4)$ approaches  the GIT quotient of the  Chow variety 
  $\Chow_{(2,4)}\gquot \SL(4)$.

In the opposite direction, we may consider $m$ as small as possible, namely $m=4$. The corresponding line bundle $L_4$ is  ample on $U$ but  \emph{not} ample on  $\Hilb_{(2,4)}$. Here we recall that (semi)stability makes sense with respect to any linearized line bundle, and hence there is a quasi-projective quotient of the open set of semistable points $\Hilb_{(2,4)}^{ss}(L_4)$, see Theorem 1.10 in~\cite{GIT}.

On the other hand, we may identify  $\Hilb_{(2,4)}\gquot_{L_4} \SL(4)$ with the GIT quotient of a space birational  to  $\Hilb_{(2,4)}$, with a linearization that is \emph{ample}.  In fact, let  $E$ be the vector-bundle over $ |\cO_{\PP^3}(2)|$ defined by
\begin{equation}
\xymatrix{H^0(Q,\cO_Q(4))  \ar[d] & \subset  & E \ar[d]  \\   
 Q & \in &  |\cO_{\PP^3}(2)|.}
\end{equation}
Let $\pi\colon\PP E\to |\cO_{\PP^3}(2)|$ be the structure map. The  Picard group of $\PP E$ is generated by 
\begin{equation}\label{genero}
 \eta:=\pi^* \calO(1), \quad \xi:=\calO_{\bP E}(1).
\end{equation}
 \begin{proposition}[{\cite[Thm 2.7]{benoist}}]\label{prp:benample}
Let $t\in\QQ$.  The $\QQ$-Cartier divisor class  $\eta+t\xi$ on $\bP E$ is ample if and only if $t\in (0,\frac{1}{3})\cap\QQ$. 
 \end{proposition} 
Now, $\PP E$ contains naturally $U$, and the complement of $U$ has  codimension greater than $1$, so that the restriction of $L_4$ to $U$ extends uniquely to a line-bundle $\ov{L}_4$. A straightforward computation (see the proof of Item~(3) of~\Ref{thm}{thmgitmodels}) shows that  $c_1(\ov{L}_4)=10\eta+\xi$. Thus $\ov{L}_4$ is ample, and from this it follows that the categorical quotient 
 $\Hilb_{(2,4)}^{ss}(L_4)$ is identified with the GIT quotient  $\PP E\gquot_{\ov{L}_4} \SL(4)$. Then, a key fact is  that   $\PP E\gquot_{\ov{L}_4} \SL(4)$  is naturally isomorphic to the GIT moduli space $\gM$ of  $(4,4)$ curves on $\bP^1\times \bP^1$ (see Item~(2) of~\Ref{thm}{thmgitmodels}). 
 
On the other hand, the extension of the Chow polarization from $U$ to all of $\PP E$ corresponds to $t=1/2$ 
 in~\Ref{prp}{benample}, and one of our final goals is to prove that the
Chow quotient 
  $\Chow_{(2,4)}\gquot \SL(4)$ is isomorphic to the Baily-Borel compactification $\sF^{*}$. 
  
  Thus it is natural to study the one-parameter variation of GIT models  $\PP E\gquot_{\eta+t\xi} \SL(4)$, for 
  $t\in (0, 1/2]\cap\QQ$. Since we need to consider   values of $t$ for which   $\eta+t\xi$ is not  ample, i.e.~$t\in [1/3, 1/2]\cap\QQ$, our approach is not as straightforward as one would like. In~\Ref{subsec}{interpoldef} we will define a VGIT on an auxiliary $\SL(4)$-space $W$, which is somewhat intermediate between $\PP E$ and $\Chow_{(2,4)}$. It is natural to expect that the quotient  of $W$ with respect to a suitable polarization $N_t$ which morally corresponds to $(\eta+t\xi)$ is a  
    projective GIT moduli space $\gM(t)$ interpolating between $\gM$ and  $\Chow_{(2,4)}\gquot \SL(4)$. The main  result of the present section, i.e.~~\Ref{thm}{thmgitmodels},  establishes the expected interpolation. The proof is in~\Ref{subsec}{interpolgit}.

A very similar analysis  was done in \cite{g4git} (which in turn was influenced by Benoist \cite{benoist}) for $(2,3)$ complete intersections in $\bP^3$. We have modified somewhat the arguments of \cite{g4git}, leading in particular to a more streamlined definition for $\gM(t)$, but the main point remains the same: because of~\Ref{prp}{propssci}, the GIT analysis for $\gM(t)$ only involves complete intersection schemes. Thus, morally speaking, we are investigating a  VGIT quotient $U\gquot_{(\eta+t\xi)|_{U}} \SL(4)$, but some extra work is required since $U$ is not projective.
\subsection{Set up of the VGIT, and statement of the  main result}\label{subsec:interpoldef}
As before, let $U$ be the parameter space for $(2,4)$ complete intersection curves in 
$\PP^3$.
Thus we have a regular embedding
\begin{equation}
U\hra \PP(E)\times\Chow_{(2,4)}.
\end{equation}
\begin{definition}\label{dfn:wpergit}
Let $\cP\subset  \PP(E)\times\Chow_{(2,4)}$ be the closure of $U$. 
\end{definition}
The action of $\SL(4)$ on $\PP^3$ defines an $\SL(4)$-action on $\cP$: this is the space on which we will follow the VGIT. To do this, we need to define a variable $\SL(4)$-linearized ample line bundle on $\cP$. 
Choose
\begin{equation}\label{ecceps}
0<\delta< 1/6,\quad \delta\in\QQ.
\end{equation}
Let $p_1,p_2$ be the projections of $\cP$ onto the first and second factor respectively. For $t\in(\delta,1/2)\cap \QQ$,  let
\begin{equation}\label{eccoenne}
N_t:=\frac{1-2t}{1-2\delta} \cdot p_1^{*}(\eta+\delta\xi) +\frac{t-\delta}{2(1-2\delta)}p_2^{*}L_{\infty}.
\end{equation}
Clearly $N_t$ is  ample  for  $t\in(\delta,1/2)\cap\QQ$, and   semi-ample for $t=1/2$. Thus for each  $t\in(\delta,1/2]\cap\QQ$, we have a GIT quotient
\begin{equation}\label{defgmt}
\gM(t):=\cP\gquot_{N_t}\SL(4).
\end{equation}
Notice that, a priori, $\gM(t)$ depends also on the choice of $\delta$. Formally, we choose one $\delta$, and this justifies the absence of $\delta$ from the notation. More substantially, a posteriori we will see that  $\gM(t)$ does \emph{not} depend on the choice of $\delta$.

\begin{remark}\label{rmk:defgit12}
For $t\in(\delta,1/2)\cap\QQ$, $N_t$ is ample and thus there is no ambiguity in the definition of $\gM(t):=\cP\gquot_{N_t}\SL(4)$. For $t=\frac{1}{2}$, $N_{t}$ is only semi-ample, thus some care should be taken in defining $\gM(\frac{1}{2})$. Specifically, we define
$$\gM\left(\frac{1}{2}\right):=\Proj R\left(\cP, N_{\frac{1}{2}}\right)^{\SL(4)}.$$
By  \eqref{eccoenne}, it is clear that 
\begin{equation}\label{m12chow}
\gM\left(\frac{1}{2}\right)\cong \Chow_{(2,4)}\gquot \SL(4).
\end{equation}
Furthermore, a straightforward application of the functoriality of the numerical function $\mu$ (\cite[p. 49]{GIT}) and an application of the numerical criterion, gives that $\gM(t)$ behaves as expected as $t$ attains the value $\frac{1}{2}$. Specifically, the following hold:
\begin{itemize}
\item[i)] If $x\in \cP^{ss}(t)$ for $t\in\left(\frac{1}{2}-\epsilon,\frac{1}{2}\right)$, then $p_2(x)\in \Chow_{(2,4)}^{ss}$ (where $p_2:\cP\to \Chow_{(2,4)}$ is the projection). Conversely, if $y\in \Chow_{(2,4)}^{s}$, then $p_2^{-1}(y)\in \cP^{s}(t)$. Thus, we can write (by a slight abuse of notation)
$$\cP^s\left(\frac{1}{2}\right)\subset \cP^{s}\left(\frac{1}{2}-\epsilon\right)\subset \cP^{ss}\left(\frac{1}{2}-\epsilon\right)\subset \cP^{ss}\left(\frac{1}{2}\right)$$
(compare \cite[(3.18)]{VGIT} in the standard VGIT set-up). 
\item[ii)] In particular, there exists a natural birational map 
\begin{equation}
\gM\left(\frac{1}{2}-\epsilon\right)\to \gM\left(\frac{1}{2}\right)
\end{equation}
which is compatible with the Hilbert-to-Chow morphism $\Hilb_{(2,4)}\to \Chow_{(2,4)}$. 
\end{itemize} 
\end{remark}
 
In order to make the connection with the discussion at the beginning of the present section,  and to explain the choice of coefficients in~\eqref{eccoenne}, we prove the following result.
\begin{proposition}\label{prp:chowlb}
Let $\ov{L}_{\infty}\in\Pic(\PP E)_{\QQ}$ be the unique extension of $c^{*}L_{\infty}$ (recall that $c\colon U\to \Chow_{(2,4)}$ 
is the restriction of the Hilbert-Chow map). Then 
\begin{equation}\label{chowlb}
\ov{L}_{\infty}=4\eta+2\xi.
\end{equation}
\end{proposition}
\begin{proof}
There exist $x,y\in\QQ$ such that $\ov{L}_{\infty}=x\eta+y\xi$. We compute $x$ and $y$ by computing intersection indices with the following two projective curves in $U$:
\begin{eqnarray}
\Gamma & := & \{V(q, \mu_0 f +\mu_1 g)\mid \deg q=2,\ \deg f=\deg g=4\}, \label{curvagam} \\
\Omega & := & \{V( \mu_0 q +\mu_1 r, f)\mid \deg q=\deg r=2,\ \deg f=4\}. \label{curvaom}
\end{eqnarray}
(Here $q,r,f,g$ are chosen generically.) Notice that both $\Gamma$ and $\Omega$ are contained in $U$. The degree of $L_{\infty}|_{\Gamma}$ is equal to the number of curves parametrized by $\Gamma$ meeting a generic line in $\PP^3$, and similarly for $L_{\infty}|_{\Omega}$. 
Thus
\begin{equation}\label{duequattro}
\deg(L_{\infty}|_{\Gamma})=2,\qquad \deg\left(L_{\infty}|_{\Omega}\right)=4.
\end{equation}
On the other hand, we have the following intersection indices 
\begin{equation}\label{etaxi}
\begin{tabular}{c|c|c }
 &  $\eta$ & $\xi$  \\
\hline
$\Gamma$  & $0$ & $1$  \\
\hline
$\Omega$  & $1$ & $0$
\end{tabular}
\end{equation}
Equation~\eqref{chowlb} follows at once from~\eqref{duequattro} and~\eqref{etaxi}.
\end{proof}
\begin{corollary}\label{crl:chowlb}
With notation as above,  $\eta+t\xi$ is the unique divisor class in $\Pic(\PP E)$ which, restricted to $U$, is equal to  $N_t|_U$. 
\end{corollary}
With these preliminaries out of the way, we are almost ready to state the  main result of the present section. We will compare GIT quotients of $\cP$, $\PP E$,  $\Hilb_{(2,4)}$ and $\Chow_{(2,4)}$ modulo the natural $\SL(4)$-action. With the exception of the latter, all the listed spaces contain $U$, the parameter space for $(2,4)$ complete intersections in $\PP^3$ (with  the natural $\SL(4)$-action),  as an open dense subset. Thus $U$ induces  birational maps between  any  $\SL(4)$-quotients of  $\cP$, $\PP E$ or  $\Hilb_{(2,4)}$. Similarly, if $V\subset \Chow_{(2,4)}$ is the dense subset of Chow forms of $(2,4)$ intersections, the Hilbert-Chow map induces a birational (regular) map $U\to V$, and hence also a  
birational map between  any  $\SL(4)$-quotient of  $\cP$, $\PP E$  or $\Hilb_{(2,4)}$ and  any  $\SL(4)$-quotient
 of  $\Chow_{(2,4)}$. In the result below, whenever we state that certain moduli spaces are isomorphic, what we really mean   is that one of the birational maps discussed above is, in fact, an isomorphism. We will also compare $\SL(4)$-quotients of 
 $\PP E$ and the quotient $\gM:=|\cO_{\PP^1}(4)\boxtimes \cO_{\PP^1}(4)|\gquot\Aut(\PP^1\times\PP^1)$; when we state that they are isomorphic, it is understood that the isomorphism is induced by the inclusion $|\cO_{\PP^1}(4)\boxtimes \cO_{\PP^1}(4)|\subset U$ determined by the choice of an isomorphism between $\PP^1\times\PP^1$ and a smooth quadric in $\PP^3$.
\begin{theorem}\label{thm:thmgitmodels}
Keep notation notation as above, in particular   $\delta$ is as in~\eqref{ecceps}. The following hold:
\begin{enumerate}
\item For $t\in(\delta,1/3)$, $\gM\left(t\right) \cong \bP E \gquot_{\eta+t\xi} \SL(4)$. 
\item For $t\in(\delta,1/6)$, $\gM(t)\cong \gM$. 
\item For $ m\ge 4$, we have $\Hilb_{(2,4)} \gquot_{L_m} \SL(4)\cong \gM(t(m)) $, where
 $$t(m):= \frac{(m-3)^2}{2(m^2 -4m+5)}.$$
\item $\gM(1/2)\cong \Chow_{(2,4)}\gquot \SL(4)$.
\end{enumerate}
\end{theorem}

\Ref{thm}{thmgitmodels} is inspired by~\cite{g4git} (esp.~Theorem 6.7 in loc.~cit.),  where the analogous question for canonical genus $4$ curves (or equivalently $(2,3)$ complete intersections in $\bP^3$) was studied (in connection to the Hassett-Keel program). The main point that allows to establish the theorem is the fact that GIT (semi)stable points of $\Hilb_{(2,4)}$, or of $\PP E$ (with any of our polarizations) are contained in $U$. This was proved for $(2,3)$ curves in \cite[Prop. 5.2]{g4git}. In the next subsection we will prove that the analogous result holds for $(2,d)$ curves in $\PP^3$ for any $d$.

\subsection{GIT for $\Hilb_{(2,d)}$}
\subsubsection{The $m$-th Hilbert point}\label{subsubsec:mHilbert}
Let $U_{(2,d)}$ be the open  subset of the Hilbert scheme of $\PP^3$  parametrizing complete intersections of a quadric and a  surface of degree $d$ (in short $(2,d)$ c.i.'s), and let 
$\Hilb_{(2,d)}$ be its closure (thus $U=U_{(2,4)}$).  
Let $P_d\in\ZZ[m]$ be defined by $P_d=2dm-d^2+2d$. If $C\in U_{(2,d)}$, then 
\begin{equation*}
\rk(H^0(\PP^3,\cO_{\PP^3}(m))\lra H^0(C,\cO_{C}(m)))=P_d(m),\qquad m\ge d-1.
\end{equation*}
For $m\ge d$, let  
\begin{equation}
\begin{matrix}
U_{(2,d)} & \overset{f_{d,m}}{\lra} & \Gr(H^0(\PP^3,\cO_{\PP^3}(m)),P_d(m)) \\
C & \mapsto & H^0(\PP^3,\cO_{|PP^3}(m))\twoheadrightarrow H^0(C,\cO_{C}(m))
\end{matrix}
\end{equation}
be the natural map, where $\Gr(H^0(\PP^3,\cO_{\PP^3}(m)),P_d(m))$ is the Grassmannian parametrizing $P_d(m)$-dimensional quotients of $H^0(\PP^3,\cO_{\PP^3}(m))$. Let $\Hilb^m_{(2,d)}$
 be the closure of the image of $f_m$. A point of $\Hilb^m_{(2,d)}$ is determined by a quotient $\varphi\colon H^0(\PP^3,\cO_{\PP^3}(m))\twoheadrightarrow  V$, where $V$ is a vector space of dimension $P_d(m)$;  we denote the corresponding point by $I:=\ker\varphi$. 
 Thus $I\subset H^0(\PP^3,\cO_{\PP^3}(m))$ is a vector subspace of codimension $P_d(m)$.

 If $m\gg 0$ then  $\Hilb^m_{(2,d)}$ is isomorphic to $\Hilb_{(2,d)}$.
 
 By associating to a curve  $C\in U_{(2,d)}$ its Chow point (the hypersurface in $(\PP^3)^{\vee}\times (\PP^3)^{\vee}$ parametrizing couples $(H_1,H_2)$ of hyperplanes such that $H_1\cap H_2\cap C\not=\es$), we get an embedding $U_{(2,d)}\subset \Chow(\PP^3)$. We let    $\Chow_{(2,d)}\subset  \Chow(\PP^3)$ be the closure of $U_{(2,d)}$.

\subsubsection{(Semi)stability of points in $\Hilb^m_{(2,4)}\setminus \im f_m$}
The group $\SL(4)$ acts on  $\Hilb^m_{(2,d)}$, and on the restriction  to  $\Hilb^m_{(2,d)}$ of the Pl\"ucker (ample) line bundle on $\Gr(H^0(\PP^3,\cO_{\PP^3}(m)),P_d(m))$. Thus we have a notion of (semi)stability of elements of  $\Hilb^m_{(2,d)}$. 
We recall how to determine whether a point of  $\Hilb^m_{(2,d)}$ is (semi)stable, following~~\cite{hhl-stab}.

Let  $\lambda$ be a $1$ PS of  $\SL(4)$.  Let $[x_0,\ldots,x_3]$ be homogeneous coordinates that diagonalize $\lambda$, i.e.~$\lambda(s)=\diag(s^{r_0},s^{r_1},s^{r_2},s^{r_3})$.  Given $A=(A_0,\ldots,A_3)\in\NN^4$, we let $x^A:=x_0^{A_0}x_1^{A_1}x_2^{A_2}x_3^{A_3}$. The $\lambda$-\emph{weight} of $x^A$ is 
\begin{equation}\label{pesomon}
\weight_{\lambda}(x^A):=\sum_{i=0}^3 r_i A_i.
\end{equation}
Given pairwise distinct monomials $x^{A(1)},\ldots,x^{A(P_d(m))}$ of degree $m$, we let
\begin{equation*}
\weight_{\lambda}(x^{A(1)}\wedge\ldots\wedge x^{A(P_d(m))}):=\sum_{j=1}^{P_d(m)}\weight_{\lambda}(x^{A(j)}).
\end{equation*}
Let  $I\in\Hilb^m_{(2,d)}$. The Hilbert-Mumford index of $\lambda$ with respect to the Pl\"ucker linearization is 
\begin{equation}\label{hmindex}
\scriptstyle
\mu_m(I,\lambda)=\max\{-\weight_{\lambda}(x^{A(1)}\wedge\ldots\wedge x^{A(P_d(m))})\mid\ \text{$\{x^{A(1)},\ldots, x^{A(P_d(m))}\}\pmod{I}$ is a basis of 
$H^0(\PP^3,\cO_{\PP^3}(m))/I$}\}.
\end{equation}
By the Hilbert-Mumford Criterion, the point $I$ is semistable if and only if $\mu_m(I,\lambda)\ge 0$ for all $\lambda$, and it is stable if and only if strict inequality holds for each $\lambda$. 

Define a total ordering $\overset{\lambda}{\prec}$ on the set of monomials $\{x^A\}$ by declaring that $x^A \overset{\lambda}{\prec} x^B$ if
\begin{enumerate}
\item
$\deg x^A<\deg x^B$, or
\item
$\deg x^A=\deg x^B$, and $\weight_{\lambda}(x^A)<\weight_{\lambda}(x^B)$, or
\item
$\deg x^A=\deg x^B$,  $\weight_{\lambda}(x^A)=\weight_{\lambda}(x^B)$, and $x^A$ precedes $x^B$ in the lexicographical ordering.
\end{enumerate}
Let $f\in\CC[x_0,\ldots,x_3]$; we let  $\initial^{\lambda}_{\prec}(f)$ be the maximum (with respect to $\overset{\lambda}{\prec}$)   among monomials appearing with non zero coefficient in the expansion of $f$. 
Given a non zero vector  subspace $I\subset\CC[x_0,\ldots,x_3]$, we let  $\initial^{\lambda}_{\prec}(I)$ be the set whose elements are the monomials $\initial^{\lambda}_{\prec}(f)$ for $f\in I$. 
\begin{proposition}\label{prp:grobstab}
Let   $I\in\Hilb^m_{(2,d)}$, and let  $\lambda$ be a  $1$ PS of  $\SL(4)$. Then
\begin{equation}\label{hhlineq}
\mu_m(I,\lambda)=\sum_{x^A\in \initial^{\lambda}_{\prec}(I)}\weight_{\lambda}(x^A).
\end{equation}
\end{proposition}
\begin{proof}
This may be proved directly, or deduced from~\cite{hhl-stab} by showing that (in the notation of pp.~43-44 of loc.~cit.)
\begin{equation}\label{opposto}
\mu([X]_m,\rho')=-\sum_{i=1}^{P(m)}\weight_{\rho'}(x^{a(i)}).
\end{equation}
The above formula follows from Eqtn~(2.5) of~\cite{hhl-stab}, and the (easily checked) relation
\begin{equation*}
\weight_{\rho}(x^{a(i)})=\frac{1}{N+1}(\weight_{\rho'}(x^{a(i)})+rm).
\end{equation*}
In our case Equation~\eqref{opposto} reads
\begin{equation}\label{opposto2}
\mu_m(I,\lambda)=-\sum_{x^A\notin \initial^{\lambda}_{\prec}(I)}\weight_{\lambda}(x^A).
\end{equation}
The above equation is equivalent to~\eqref{hhlineq}, because the sum of  $\weight_{\lambda}(x^A)$ over all degree-$m$ monomials $x^A$ is equal to $0$. 
\end{proof}
The group $\SL(4)$ acts also on  $\Chow_{(2,d)}$, and on the restriction  to  $\Chow_{(2,d)}$ of the Chow polarization. There is a relation between the Hilbert-Mumford index of points in the Hilbert scheme and corresponding points of the Chow variety. First notice that the Hilbert-Chow morphism restrict to a (birational) morphism $\gamma\colon \Hilb_{(2,d)}\to \Chow_{(2,d)}$; then 
\begin{equation}\label{chowlim}
\mu(\gamma(I),\lambda)=\lim_{m\to +\infty}\frac{1}{m^2}\mu_m(I,\lambda).
\end{equation}
(The above equation  makes sense because  $\Hilb^m_{(2,d)}=\Hilb_{(2,d)}$ for $m\gg 0$.)

The result below generalizes  Proposition~5.2 in~\cite{g4git}. 
 \begin{proposition}\label{prp:propssci}
Let $d\ge 3$, and $m\ge d$. Then all points of $\Hilb^m_{(2,d)}\setminus U_{(2,d)}$ are $\SL(4)$-unstable, and similarly  all points of $\Chow_{(2,d)}\setminus U_{(2,d)}$ are $\SL(4)$-unstable. 
\end{proposition}
 \begin{proof}
Let $I\in\Hilb^m_{(2,d)}$. Then $I$ contains a  quadratic form $Q$, and a degree-$d$ form $F$ which is not a multiple  of $Q$. Now suppose that $I\notin \im f_{d,m}$; then the common zero locus of $Q$ and $F$ is not a curve. It follows that there exist homogeneous coordinates $[x_0,\ldots,x_3]$ on $\PP^3$ such that one of the following holds:
\begin{enumerate}
\item
$Q=x_0x_1$, and $F=x_0 G$, where $G\in\CC[x_0,\ldots,x_3]_{d-1}$, and $x_1\not | G$. 
\item
$Q=x^2_0$, and $F=x_0 G$, where $G\in\CC[x_0,\ldots,x_3]_{d-1}$, and $x_0\not | G$. 
\end{enumerate}
First assume that Item~(1) holds.  Let $\lambda$ be the $1$ PS of $\SL(4)$ defined by $\lambda=\diag(-3,1,1,1)$ (with respect to the chosen homogeneous coordinates). 
Let us prove that
\begin{equation}\label{negativo}
\mu_m(I,\lambda)\le 
-2(a+1)m^2-(d^2-2(2a+3)d+6(a+1))m+\frac{2}{3}(d-1)(d-2)(d-3(a+1)).
\end{equation}
Let 
\begin{equation*}
U:= x_0 x_1\CC[x_0,\ldots,x_3]_{m-2},\quad V:= x_0 G\CC[x_0,\ldots,x_3]_{m-d},\quad T:=U+V.
\end{equation*}
Notice that $U,V,T$ are subspaces of  $I$. 
We claim that
\begin{equation}\label{pesoti}
\scriptstyle 
\sum\limits_{x^A\in \initial^{\lambda}_{\prec}(T)}\weight_{\lambda}(x^A)
=-\frac{1}{2}m^3+\frac{1}{2}(2d-4a-5)m^2-\frac{1}{6}(9d^2-3(8a+13)d+36(a+1))m+\frac{2}{3}(d-1)(d-2)(d-3(a+1)).
\end{equation}
In fact, since  $\weight_{\lambda}(x_0 x_1)=-2$ and the representation of $\lambda$ on $\dim\CC[x_0,\ldots,x_3]_{m-2}$ has determinant $1$, we have
\begin{equation}\label{pesou}
%
\sum\limits_{x^A\in \initial^{\lambda}_{\prec}(U)}\weight_{\lambda}(x^A)  =  (-2)\cdot\dim\CC[x_0,\ldots,x_3]_{m-2}=(-2)\cdot{m+2\choose 3}.
\end{equation}
 Next, let $0\le a\le(d-1)$ be the maximum number such that $x_0^a | G$. Then
\begin{equation}\label{pesov}
%
\sum\limits_{x^A\in \initial^{\lambda}_{\prec}(V)}\weight_{\lambda}(x^A)  =(d-4a-4)\cdot\dim\CC[x_0,\ldots,x_3]_{m-d} =(d-4a-4)\cdot{m+3-d\choose 3}.
\end{equation}
Lastly, since $U\cap V= x_0 x_1 G\CC[x_0,\ldots,x_3]_{m-d-1}$, we have 
\begin{equation}\label{pesoint}
%
\sum\limits_{x^A\in \initial^{\lambda}_{\prec}(U\cap V)}\weight_{\lambda}(x^A)  =(d-4a-3)\cdot\dim\CC[x_0,\ldots,x_3]_{m-d-1} =(d-4a-3)\cdot{m+2-d\choose 3}.
\end{equation}
Next, we have the Grassmann-like formula 
\begin{equation*}
\sum_{x^A\in \initial^{\lambda}_{\prec}(T)}\weight_{\lambda}(x^A)=\sum_{x^A\in \initial^{\lambda}_{\prec}(U)}\weight_{\lambda}(x^A)+\sum_{x^A\in \initial^{\lambda}_{\prec}(V)}\weight_{\lambda}(x^A)-\sum_{x^A\in \initial^{\lambda}_{\prec}(U\cap V)}\weight_{\lambda}(x^A).
\end{equation*}
Equation~\eqref{pesoti} follows from the above formula, together with~\eqref{pesou}, \eqref{pesov}, and~\eqref{pesoint}.

In order to prove Equation~\eqref{negativo}, we notice that $\weight_{\lambda}(x^A)\le m$ for every monomial of degree $m$, and hence 
\begin{equation*}
\scriptstyle 
\mu_m(I,\lambda)=\sum\limits_{x^A\in \initial^{\lambda}_{\prec}(I)}\weight_{\lambda}(x^A)\le \sum\limits_{x^A\in \initial^{\lambda}_{\prec}(T)}\weight_{\lambda}(x^A)+
m\cdot (\dim I- \dim T)=\sum\limits_{x^A\in \initial^{\lambda}_{\prec}(T)}\weight_{\lambda}(x^A)+\frac{1}{2}(m^3-(2d-1)m^2+d(d-1) m).
\end{equation*}
Thus  Equation~\eqref{negativo} follows from~\eqref{pesoti}. 

Let $P(d,a,m)$ be the right hand side of~\eqref{negativo}; we claim that $P(d,a,m)<0$ for $d\ge 3$, $a\in\{0,\ldots,d-1\}$,  and $m\ge d$.
One easily checks that $P(d,a+1,m)<P(d,a,m)$ if $m>0$. Hence  it suffices to check that $P(d,0,m)<0$ for  $m\ge d$. The function $P(d,0,m)$ is decreasing for $m\ge d$, hence one is reduced to proving that $P(d,0,d)<0$ for $d\ge 3$; this is straightforward. 

This proves that  if Item~(1) holds, then $\mu_m(I,\lambda)<0$, and hence $I$ is unstable.  Moreover
\begin{equation}
\scriptstyle
\lim\limits_{m\to +\infty}\frac{1}{m^2}\mu_m(I,\lambda)\le \lim\limits_{m\to +\infty} \frac{1}{m^2}(-2(a+1)m^2-(d^2-2(2a+3)d+6(a+1))m+\frac{2}{3}(d-1)(d-2)(d-3(a+1)))=-2(a+1)<0.
\end{equation}
By Equation~\eqref{chowlim} it follows that $\mu(\gamma(I),\lambda)<0$, i.e.~the corresponding point of the Chow variety is unstable. 

Next, assume that  Item~(2) holds. Again, let $\lambda$ be the $1$ PS of $\SL(4)$ defined by $\lambda=\diag(-3,1,1,1)$.  We claim that
\begin{equation}\label{negancora}
\mu_m(I,\lambda)\le 
-(2d+6)m^2-(d^2-6d+17)m-(2d^2-6d+10).
\end{equation}
We quickly go over the proof of~\eqref{negancora}. Let 
\begin{equation*}
U:= x_0^2\CC[x_0,\ldots,x_3]_{m-2},\quad V:= x_0 G\CC[x_0,\ldots,x_3]_{m-d},\quad T:=U+V.
\end{equation*}
Then $U,V,T$ are subspaces of  $I$, and computations similar to those performed above give that 
\begin{equation}\label{pesiti}
\sum_{x^A\in \initial^{\lambda}_{\prec}(T)}\weight_{\lambda}(x^A)=-\frac{1}{2}m^3-\frac{2d+13}{2}m^2+\frac{d^2+5d-32}{2}m-2d^2+6d-10.
\end{equation}
Since $\weight_{\lambda}(x^A)\le m$ for every monomial of degree $m$, it follows that  
\begin{equation}\label{stima}
\scriptstyle 
\mu_m(I,\lambda)=\sum\limits_{x^A\in \initial^{\lambda}_{\prec}(I)}\weight_{\lambda}(x^A)\le \sum\limits_{x^A\in \initial^{\lambda}_{\prec}(T)}\weight_{\lambda}(x^A)+
m\cdot (\dim I- \dim T)=\sum\limits_{x^A\in \initial^{\lambda}_{\prec}(T)}\weight_{\lambda}(x^A)+\frac{1}{2}(m^3-(2d-1)m^2-(3d^2-7d+2) m).
\end{equation}
Inequality~\eqref{negancora} follows at once from~\eqref{pesiti} and~\eqref{stima}. The right hand side of~\eqref{negancora} is negative for $d\ge 3$ and $m\ge d$; it follows that $I$ is unstable. Proceeding as in the previous case, we also get that the corresponding point of the Chow variety is unstable. 
\end{proof}
\subsection{On (non-semi)stability of points of $\PP E$ and of $\cP$}
We will prove results on $\PP E$ and $\cP$ that parallel~\Ref{prp}{propssci} (in the case $d=4$).

We denote points of $\PP E$ as follows:
\begin{equation}\label{eqpe}
\bP E=\{([f_2],[\bar{f}_4])\mid f_d\in\Gamma(\calO_{\bP^3}(d)), \bar{f}_4=f_4|_{V(f_2)}\}
\end{equation}
The Hilbert-Mumford numerical function for points of $\PP E$ relative to $\eta+t\xi$ has been computed by Benoist. 
First we recall that the Hilbert-Mumford numerical function of a non zero homogeneous polynomial $f=\sum_A f_A x^A$ with respect to a $1$ PS $\lambda(s)=\diag(s^{r_0},\ldots,s^{r_n})$ is
\begin{equation}\label{numfun}
\mu(f,\lambda)=\max\{ \weight_{\lambda}(x^{A}) \mid f_A\not=0\}.
\end{equation}
Thus 
\begin{equation*}
\mu(f,\lambda)
\begin{cases}
<0 & \text{if and only $\lim\limits_{s\to 0}\lambda(s)f=0$,} \\
=0 & \text{if and only $\lim\limits_{s\to 0}\lambda(s)f\not=0$,} \\
>0 & \text{ if and only $\lim\limits_{s\to 0}\lambda(s)f$ does not exist.}
\end{cases}
\end{equation*}
(Recall that  $g\in\SL(n+1)$ acts on $f\in\CC[x_0,\ldots,x_n]$ by the formula $gf(x)=f(g^{-1}x)$.) 
Let $([f_2],[\ov{f}_4])\in  \bP E$ and let $\lambda$ be a $1$-PS of $\SL(4)$. Then, by   Proposition~2.15 of~\cite{benoist}
 \begin{equation}\label{benfor}
\mu^{\eta+t\xi}(x,\lambda):=\mu(f_2,\lambda)+t\min_{f\in[f_4]}\mu(f,\lambda).
\end{equation}
   The  straightforward proof of the result below is left to the reader.
\begin{proposition}\label{prp:facile}
Let $t\in(0,1/2]\cap\QQ$.  Let $([f_2],[\bar{f}_4])\in \PP E\setminus U$. Then the following hold:
\begin{enumerate}
\item
There exist homogeneous coordinates $[x_0,\ldots,x_3]$ on $\PP^3$ such that 
one of the following holds:
\begin{enumerate}
\item
$f_2=x_0x_1$, and $f_4=x_0 g$, where $g\in\CC[x_0,\ldots,x_3]_{3}$, and $x_1\not | g$. 
\item
$f_2=x^2_0$, and $f_4=x_0 g$, where $g\in\CC[x_0,\ldots,x_3]_{3}$, and $x_0\not | g$. 
\end{enumerate}
\item
Let $\lambda$ be the $1$-PS of $\SL(4)$ given by $\lambda=\diag(-3,1,1,1)$. Then $\mu^{\eta+t\xi}([f_2],[\bar{f}_4])\le -2$. 
\end{enumerate}
\end{proposition}
\begin{corollary}\label{crl:facile}
Let $t\in(0,1/3)\cap\QQ$.  Then $\PP E^{ss}(\eta+t\xi)\subset U$.
\end{corollary}
\begin{proof}
The corollary follows at once from~\Ref{prp}{facile} and the Hilbert-Mumford numerical criterion for semistability, which holds because $\eta+t\xi$ is ample for $t\in(0,1/3)\cap\QQ$.
\end{proof}
\begin{proposition}\label{prp:doppiovu}
Let $t\in(\delta,1/2]\cap\QQ$.  Then $\cP^{ss}(N_t)\subset U$.
\end{proposition}
\begin{proof}
Let $z\in(\cP\setminus U)$. There exist a curve $\cC\subset\PP^3$, parametrized by a point of $\Hilb_{(2,4)}$, and a decomposition
\begin{equation*}
H^0(\PP^3,\cI_{\cC}(4))=H^0(\PP^3,\cO_{\PP^3}(2))\cdot f_2+\CC f_4,\quad f_d\in H^0(\PP^3,\cO_{\PP^3}(d)),
\end{equation*}
such that
\begin{equation*}
z=( ([f_2],[\ov{f}_4]), c(\cC)),
\end{equation*}
(here $c\colon \Hilb_{(2,4)}\to\Chow_{(2,4)}$ is the Hilbert-to-Chow map) and either Item~(a) or Item~(b) of~\Ref{prp}{facile} holds. Let $\lambda$ be the $1$ PS of $\SL(4)$ in Item~(2)  of~\Ref{prp}{facile}. Then, by linearity of the Hilbert-Mumford numerical function
\begin{equation}\label{combnum}
\mu^{N_t}(z,\lambda)=\frac{1-2t}{1-2\delta} \mu^{\eta+\delta\xi} (([f_2],[\ov{f}_4]),\lambda)+
\frac{t-\delta}{2(1-2\delta)}\mu^{L_{\infty}}(c(\cC),\lambda).
\end{equation}
We claim that both numerical functions in the right hand side of~\eqref{combnum} are strictly negative. In fact 
$ \mu^{\eta+\delta\xi} (([f_2],[\ov{f}_4]),\lambda)<0$ by~\Ref{prp}{facile}. On the other hand, the proof of~\Ref{prp}{propssci} gives that 
$\mu^{L_{\infty}}(c(\cC),\lambda)<0$.
\end{proof}
\subsection{Proof of the main result}\label{subsec:interpolgit}
We will prove~\Ref{thm}{thmgitmodels}. 

Item~(1): We will apply Lemma 4.17 in~\cite{g4git}. For the reader's convenience,  we record below the part of that lemma that we will need.
\begin{lemma}[\cite{g4git}, Lemma 4.17]\label{lmm:quasiprojgit}
Let $X$ be a projective variety, let $G$ be a reductive group acting on $X$, and let $L$ be a $G$-linearized ample  line bundle on $X$.  Then the natural map 
\begin{equation*}
X^{ss}(L)\gquot G \lra X\gquot_L G=\Proj\left(\bigoplus_{n=0}^{\infty}H^0(X,L^{\otimes n})^G\right)
\end{equation*}
is an isomorphism.
\end{lemma}
Applying~\Ref{lmm}{quasiprojgit} to $X=\cP$, $G=\SL(4)$, $L=N_t$, and to $X=\PP(E)$, $G=\SL(4)$, $L=\eta+t\xi$,  we get two isomorphisms
\begin{equation}\label{duemod}
\cP^{ss}(N_t)\gquot \SL(4) \overset{\sim}{\lra} \gM(t),\qquad 
 \PP E^{ss}(\eta+t\xi)\gquot \SL(4) \overset{\sim}{\lra}  \PP E\gquot_{\eta+t\xi} \SL(4).
\end{equation}
By~\Ref{prp}{doppiovu} we have $\cP^{ss}(N_t)\subset U$, and by~\Ref{crl}{facile} we have $\PP E^{ss}(\eta+t\xi)\subset U$. 
On the other hand, $N_t|_{U}=(\eta+t\xi)|_{U}$ by~\Ref{crl}{chowlb}. It follows that both $\cP^{ss}(N_t)$ and 
$\PP E^{ss}(\eta+t\xi)$ are equal to the set of points in $U$ which are semistable for the action of $\SL(4)$, with respect to the 
the linearized line bundle $N_t|_{U}$. Thus $\cP^{ss}(N_t)=\PP E^{ss}(\eta+t\xi)$ and $\cP^{ss}(N_t)\gquot \SL(4)=\PP E^{ss}(\eta+t\xi)\gquot \SL(4)$. Item~(1) now follows from~\eqref{duemod}.

Item~(2): By Item~(1) it suffices to prove that
\begin{equation}\label{smallti}
\PP E\gquot_{\eta+t\xi}\SL(4)\cong \gM.
\end{equation}
Let $([f_2],[\ov{f}_4])\in \PP E^{ss}(\eta+t\xi)$. We claim that $V(f_2)$ is a smooth quadric. In fact, assume that  
 $V(f_2)$ is singular. Then there exist homogeneous coordinates $[x_0,\ldots,x_3]$ such that $f_2\in\CC[x_1,x_2,x_3]_2$. Let $\lambda=\diag(-3,1,1,1)$. Then $\mu(f_2,\lambda)=-2$, and hence (recall that $t<1/6$)
 \begin{equation}
\mu(f_2,\lambda)+t\min_{f\in[f_4]}\mu(f,\lambda)\le -2+12 t<0.
\end{equation}
This is a contradiction. Hence we have proved that if $([f_2],[\ov{f}_4])\in \PP E^{ss}(\eta+t\xi)$, then $V(f_2)$ is a smooth quadric. Now the isomorphism in~\eqref{smallti} follows from the proof of Lemma~4.18 in~\cite{g4git} (which applies verbatim).  

Item~(3): Arguing as in the proof of Item~(1), it suffices to show that
\begin{equation}\label{hilbmu}
L_m|_U=(2(m^2-4m+5)\eta+(m-3)^2\xi)|_U.
\end{equation}
Let $\Gamma,\Omega\subset U$ be the projective curves in~\eqref{curvagam} and~\eqref{curvaom} respectively. By~\eqref{etaxi}, it suffices to show that
\begin{equation}
\deg(L_m|_{\Gamma})=(m-3)^2,\qquad \deg(L_m|_{\Omega})=2(m^2-4m+5).
\end{equation}
This is a straightforward computation that we leave to the reader.

\section{The stability analysis for $\gM(t)$}\label{sec:sec-VGIT}
In the previous sections, we have defined $\gM(t)$ for $t\in(0,\frac{1}{2}]$ and we have identified it with various natural Hilbert and Chow GIT quotients for $(2,4)$ complete intersection curves. The purpose of this section is to analyze the variation of GIT describing $\gM(t)$, especially to identify the values of $t$ for which the stability conditions change. To start, we recall that the general theory of  variations of GIT quotients \cite{thaddeus,dolgachevhu} (see also \cite{VGIT}) says that the interval $(\delta,\frac{1}{2})\cap\QQ$ will be partitioned into finitely many (open) intervals, called {\it chambers}, on which $\gM(t)$ stays constant. The limits of these intervals are called {\it walls}, or critical values $t$. At such a critical value, there are birational maps $\gM(t\pm \epsilon)\to \gM(t)$. The composition $\gM(t-\epsilon)\dashrightarrow \gM(t+\epsilon)$ is typically a (generalized) flip, that we will refer to as a {\it wall crossing}. Since \cite{thaddeus,dolgachevhu} a significant number of applications of VGIT to moduli problems have appeared in the literature. Most relevant for us are \cite{hh2}, \cite{g4git}, and \cite{benoist} from which we borrow a number of techniques and results. 

 \subsection{Main GIT results and structure of the argument} In order to state the main results of the present section, we introduce the following tables.
\begin{equation}\label{valori}
\begin{tabular}{c|c|c|c|c|c|c|c | c | c}
value of $t_k$ &  $1/6$ & $1/4$ &  $3/10$  & $1/3$ &  $1/3$ &   $5/14$ &   $3/8$ &  $2/5$ & $1/2$ \\
\hline
$k$  & $0$ & $1$  & $2$  & $3$  &  $4$  &  $5$  & $6$ & $7$ & $8$
\end{tabular}
\end{equation}
(The equality $t_3=t_4$ is not a misprint.)
\begin{equation}\label{posizione}
\begin{tabular}{c|c|c|c|c|c|c|c | c }
type of sing. &  $(4,0)$ & $(3,1)$ &  $J_{4,\infty}$  &   $J_{3,+}$ &   $J_{3,0}$ &   $E_{14}$ & 
 $E_{13}$ &   $E_{12}$ \\
\hline
tag  & $0$ & $1$  & $2$  & $3$  &  $4$  &  $5$  & $6$ & $7$ 
\end{tabular}
\end{equation}
 \begin{theorem}\label{thm:improving}
  Let $t\in[1/6,\frac{1}{2})\cap\QQ$, and let $C=V(f_2,f_4)$ be a $(2,4)$ curve which is $N_t$ semistable. Let $X_C\to V(f_2)$ be the double cover ramified over $C$.  Then every   non slc singularity of $X_C$ appears in the list in~\eqref{posizione}, where notation for singularities is as in~\Ref{sec}{sec-stratify}.   More precisely,  letting $k\in\{0,\ldots,7\}$, the following holds:
\begin{enumerate}
\item 
If $t=t_k$,  then  every singularity of $X_C$ has tag  at least $k$.
\item 
If  $t_k< t<t_{k+1}$,  then every singularity of $X_C$ has tag  at least $k+1$.
\end{enumerate}
In particular, if  $2/5< t<1/2$, then every surface parametrized by $\gM(t)$ has slc singularities, and hence we have a \emph{regular} period map $\Phi\colon\gM(t)\to\sF^{*}$.  
 \end{theorem}
 \begin{theorem}\label{thm:thmcriticalt}
 The critical slopes of the VGIT for $\gM(t)$ in the interval $(\delta,\frac{1}{2})\cap\QQ$ are given by the $t_k$'s appearing in~\eqref{valori}, with the exclusion of $t_8=1/2$. Hence for $k\in\{0,\ldots,7\}\setminus\{3\}$ it makes sense to let $\gM(t_k,t_{k+1}):=\gM(t)$ for $t_k<t<t_{k+1}$ (we also set $\gM(\delta,1/6):=\gM$). The VGIT  for $\gM(t)$  gives the sequence of birational maps 
\begin{equation}\label{diagramflip}
\xymatrix @R=.07in @C=.01in{
\scriptstyle \gM\cong \gM(\delta,\frac{1}{6})\ar[rdd] & & \scriptstyle \gM ( \frac{1}{6} , \frac{1}{4} ) \ar[ll]\ar@{<-->}[rr]\ar[ldd] \ar[rdd] & & \scriptstyle \cdots \ar[ldd] & \scriptstyle  \gM (t_{k-1}, t_k) \ar[rdd]\ar@{<-->}[rr] & &  
\scriptstyle  \gM ( t_k , t_{k+1}) \ar[ldd] &\scriptstyle \cdots \ar[rdd]   & & \scriptstyle \gM (\frac{2}{5},\frac{1}{2})\ar@{<-->}[ll]\ar[ldd] \ar[rrdd]^{\scriptstyle \Phi}  \\
&&&&&&\\
& \scriptstyle \gM(\frac{1}{6}) & & \scriptstyle \gM( \frac{1}{4} ) & & & \scriptstyle \gM( t_k )  & & &  \scriptstyle \gM(\frac{2}{5})  & & & \scriptstyle \sF^{*}  \\
 }
 \end{equation}
 where each dotted arrow denotes a flip, and each solid arrow is a small contraction, with the exception of $\gM(\frac{1}{6},\frac{1}{4})\to \gM(\frac{1}{6})$ which is a divisorial contraction, and $\gM(\delta,\frac{1}{6})\to \gM(\frac{1}{6})$ which is an isomorphism. 
 Furthermore, the following holds. 
  For a critical value $t_k$, let $\Sigma_{-}(t_k)\subset \gM(t_{k-1},t_k)$ (if $k=0$, we mean $\{[4C]\}\subset \gM(\delta,1/6)$ where $C$ is a smooth conic) and $\Sigma_{+}(t_k)\subset \gM(t_k,t_{k+1})$ be the exceptional loci of 
 $\gM(t_{k-1},t_k)\to \gM(t_k)$ and  $\gM(t_{k},t_{k+1})\to \gM(t_k)$ respectively. Then $\Sigma_{-}(t_k)$ 
  is the strict transform of $W_k\subset \gM$ for the birational map $\gM\dra \gM(t_{k-1},t_k)$ (if $k=0$ we mean the inverse image of $W_0$), 
and $\Sigma_{+}(t_k)$   is the strict transform of $Z^{k+1}\subset \sF^*$ for $k\neq 4$, and of $Z^4\subset \sF^*$ for $k=4$,  for the birational map $\sF^{*}\dra \gM(t_{k},t_{k+1})$. 
 \end{theorem}

\begin{remark}\label{rmk:unmezzo}
The value $t=\frac{1}{2}$ is also a critical value. Since the behavior of $\gM(t)$ at $t=\frac{1}{2}$ is somewhat different from the other cases (e.g. the associated linearization $N_{\frac{1}{2}}$ is only semi-ample, and there are multiple critical orbits), we postpone the discussion of this case until \Ref{sec}{chow}.
\end{remark}

\begin{remark}\label{rmk:flip13}
At $t=\frac{1}{3}$, (the strict transform of) the $4$ dimensional locus $W_4\subset \gM$ is replaced by (the strict transform of) the codimension $4$ locus  $Z^4\subset \sF^*$. At $t=\frac{1}{3}$, the center of the flip is the curve parametrizing equivalence classes of  $(2,4)$ curves 
\begin{equation*}
V(x_0x_2+x_1^2,x_0x_3^3+2\alpha x_1x_2x_3^2-\beta x_2^3x_3),
\end{equation*}
i.e.~the set of $[\alpha,\beta]\in W\bP(2,1)$. Moreover $W_4$ and $Z^4$ are birationally fibered over it. 
The special case, $[\alpha,\beta]=[1,1]$ morally corresponds to (an undefined) $W_3$, but we have avoided defining $W_3$ in \Ref{dfn}{eccow}, because  this stratum is flipped together with $W_4$. Similarly, the undefined $Z^5$ is ``hidden'' in $Z^4$. This behavior is explained by the properties of automorphic forms on $\sF(N)$ (the D-tower of \cite{log1}); roughly speaking it is associated with the transition phase that occurs at $N=14$ (e.g. see \cite[Cor. 3.1.5]{log1}).
\end{remark}

\begin{remark} 
In order to make the meaning of the diagram~\eqref{diagramflip} more transparent, we recall the  heuristics governing the behavior of the VGIT $\gM(t)$ (N.B.  similar situations were previously analyzed in \cite{raduthesis} and \cite{g4git}). For $0<t\ll 1$, $\gM(t)$ is isomorphic to $\gM$, the GIT quotient for $(4,4)$ curves (see \Ref{thm}{thmgitmodels}). Thus, a curve $C$ is $t$-(semi)stable iff it is (semi)stable as a $(4,4)$ curve on $\bP^1\times \bP^1$. In particular, $C$ sits on a smooth quadric. As $t$ increases,   curves on the quadric cone become semistable. Such curves replace curves on the smooth quadric with a bad singularity. 
 The basic geometric behavior of $\gM(t)$ is as follows: as $t$ increases, $C$ will have a milder singularity at $p$, while on the other hand, $C$ is allowed to have a worse singularity at $v$.  Table~\ref{table101} gives a more precise summary of this behaviour: the critical curve $C^{*}_k$ has a non slc singular point in the smooth locus of the unique quadric cone containing it, while it has an $A_m$ singularity at the vertex of the quadric cone. As $k$ increases, the non slc singularity
 becomes milder, while $m$ increases (with the exception of $k=3$). 
\end{remark}
Before stating   the last main result of the present section, we recall that 
\begin{equation}\label{perioti}
\gp(t)\colon\gM(t)\dra \sF^{*}
\end{equation}
is the (birational) period map.
\begin{proposition}\label{prp:cod2}
For $t\in(1/6,\frac{1}{2})\cap\QQ$, the period map $\gp(t)$ is an isomorphism in codimension $1$. 
\end{proposition}
The proof of~\Ref{thm}{improving} is in~\Ref{subsec}{dimouno}, the proof of~\Ref{thm}{thmcriticalt} and~\Ref{prp}{cod2} is in~\Ref{subsec}{dimodue}. Here we 
 we outline the main ingredients in the proofs. Similar methods have previously occured in the GIT analysis associated to the Hassett-Keel program (see esp. \cite{hh2} and \cite{ah}). 

\begin{proof}[Outline of the proof of the main results]

\smallskip

\noindent{\bf Step (1)}:  We compute a set of \emph{potential} critical values $t_k$ for the VGIT $\gM(t)$ and the corresponding \emph{potential} critical curves $C^{*}_k$, listed in  Table~\ref{table101}.   If $t$ is a critical value, then there exist a curve with 
$\bC^*$ stabilizer that is $N_t$ semistable but not $N_{t\pm\epsilon}$ semistable. Based on this observation, we give in \Ref{subsec}{spotential}  a straightforward  algorithm that produces a set of $t_k$'s containing all actual critical values, and corresponding curves $C^{*}_k$ (there is one curve $C^{*}_k$ up to projectivities for all $k\not=3$, while for $k=3$ we get $1$ moduli for the $C_k^{*}$'s). The list of potential critical values coincides with the list of critical values in~\Ref{thm}{thmcriticalt},  but we will be able to prove that $C_k^{*}$ is $N_{t_k}$-semistable only at the end of the present section.

\smallskip

\noindent{\bf Step (2)}: As always in a GIT analysis, a key r\^ole is played by the numerical criterion for (semi)stability. In particular it allows us to prove that  a curve $C$ with a singularity with tag $k$ in Table~\ref{posizione} is $N_t$-unstable for $t_k<t$, see~\Ref{prp}{destabilizing}. The desemistabilizing $1$ PS is (conjugated to) the stabilizer of the corresponding curve $C_k^{*}$ (if $k\in\{0,1\}$ the stabilizer is not $1$ dimensional, one has to choose appropriate $1$ PS's of the stabilizers).  
This is the key step in the proof of~\Ref{thm}{improving}. 

\smallskip

\noindent{\bf Step (3)}:  In order to prove that the $t_k$'s are actual critical values we argue via a basin of attraction argument. This means that for each curve $C_k^{*}$ we study the curves $C$ such that $\lim_{s\to 0}\lambda(s)C=C^{*}_k$ for some $1$ PS in the stabilizer of $C^{*}_k$. Each $C_k^{*}$ lies on a quadric $V(f_2)$ of rank $3$, it has an $A_m$ singularity at the vertex of the quadric, it 
has a point in the smooth locus of $V(f_2)$  with tag $k$ in~\eqref{posizione}, and no other singularities.  
We show that if $C$ is in the basin of attraction of $C^{*}_k$ and $N_t$ semistable for $t<t_k$, then it has a point with tag $k$, while if it is $N_t$ semistable for $t_k<t$,
then it lies on a quadric $Q$ of rank $3$, it passes through the vertex of $Q$ and near the vertex it is of the same type as $C^{*}_k$. \Ref{thm}{thmcriticalt} is a straighfroward consequence of this result.
\end{proof}

 \subsection{Potential critical values and potential critical curves}\label{subsec:spotential} 
  By general results on VGIT (\cite{thaddeus}, \cite{dolgachevhu}; see also \cite{VGIT})  there exists a finite set  
  $\{t_i\}\subset (\delta,1/2)\cap\QQ$ of critical values (or {\it walls}) 
  for the VGIT $\gM(t)$. (Recall that in~\Ref{subsec}{interpoldef}  we have chosen a rational $\delta\in (0,1/6)$ in order to define our VGIT; in the end the choice of $\delta$ will make no difference.) A point   $t_0 \in 
  (\delta,1/2)\cap\QQ$ is a critical value   
 if  for all sufficiently small $\epsilon\in\QQ_{+}$, the following holds (e.g. \cite[\S3.2.1]{VGIT}) 
  \begin{equation}\label{nonuguale}
\cP^{ss}(N_{t_0-\epsilon})\cap \cP^{ss}(N_{t_0+\epsilon})\subsetneq \cP^{ss}(N_{t_0}).
\end{equation}
We let $\cP^{ss}(N_{t_0})^{\rm new}\subset \cP^{ss}(N_{t_0})$ be the complement of the left hand side of~\eqref{nonuguale}. 
Notice that 
 $\cP^{ss}(N_{t_0})^{\rm new}$  is a closed $\PGL(4)$-invariant subset of   $\cP^{ss}(N_{t_0})$, and that all its points are 
  strictly semistable (semistable but not stable) because
 $\cP^{s}(N_{t_0})\subset \cP^{s}(N_{t_0\pm\epsilon})$. 

Now let $x\in\cP^{ss}(N_{t_0})^{\rm new}$.  It follows from the results recalled above 
 that there is   a unique closed $\PGL(4)$-orbit  in the closure of 
 $\PGL(4) x$  in   $\cP^{ss}(N_{t_0})^{\rm new}$, and that if $x^{*}$ belongs to such a closed orbit, then its stabilizer is a reductive group of strictly positive dimension. In particular  $x^{*}$ is stabilized by a   $1$-PS $\lambda$. 
  
   In the present subsection we will write down a finite subset of  $(\delta,1/2)$ containing the set of critical values. The numbers in our list \emph{are} critical values, but  before we are in a position to prove that statement, they will be called \emph{potential critical values}. Moreover, for each potential critical value $t_i$ we will give a subset of $\cP^{ss}(N_{t_i})$ containing all elements of 
   $\cP^{ss}(N_{t_0})^{\rm new}$ stabilized by a $1$-PS - the elements of that subset (or any point in the same $\PGL(4)$ orbit) are the  \emph{potential critical curves} (notice that by~\Ref{prp}{doppiovu} any element of  $\cP^{ss}(N_{t_0})$ is in $U$, i.e.~is a $(2,4)$ c.i.~curve).
 \begin{proposition}\label{prp:indentifyt}
Keeping notation as above, the set of critical values  for the VGIT $\gM(t)$ is 
 included in the set (of potential critical values)
 \begin{equation}\label{valcrit}
 \left\{\frac{1}{6},\frac{1}{4},\frac{3}{10},\frac{1}{3},\frac{5}{14},\frac{3}{8}, \frac{2}{5}\right\}.
\end{equation}
For each potential critical value $t_k$, the corresponding potential critical curve(s) $C^{*}_k$
 appear in the row corresponding to $t_k$ in Table~\ref{table101}.  In that table $v=[0,0,0,1]\in \bP^3$ is the vertex of the quadric cone containing $C^{*}_k$, $p=[1,0,0,0]\in\bP^3$ is the unique singular point of $C^{*}_k$,   $\cC_v C^{*}_k$ is the tangent cone to $C^{*}_k$ at $v$, and $L$ is a line of $V(f_2)$ (in fact $L=V(x_0,x_1)$). 
\end{proposition}

\begin{table}[htb!]
\renewcommand\arraystretch{1.5}
\begin{tabular}{|c| l | l | l |c|c|}
\hline
$k$&$t_k$& Critical curve $C_k^{*}=V(f_2,f_4)$ & $1$ PS & Sing.~of $C^{*}_k$ at $p$ & Type of $C^{*}_k$ at $v$\\
\hline
$0$&$\frac{1}{6}$&$V(x_0x_2+x_1^2,x_3^4)$& $(1,\alpha,2\alpha-1,-3\alpha)$ & quadruple conic&$v\not\in C^{*}_k$\\
\hline
$1$&$\frac{1}{4}$&$V(x_0x_2+x_1^2,x_3^3x_1)$ &  $(1,\alpha,2\alpha-1,-3\alpha)$ &  triple conic&$A_1$\\
\hline
$2$&$\frac{3}{10}$&$V(x_0x_2+x_1^2,x_0x_3^3+x_2^2x_3^2)$& $(7,3,-1,-9)$ & $J_{4,\infty}$&$A_2$\\
\hline $3$ &$\frac{1}{3}$ &$V(x_0x_2+x_1^2,x_0x_3^3+2x_1x_2x_3^2-x_2^3x_3)$& $(3,1,-1,-3)$ & $J_{3,\infty}$& $2$(twisted cubic)\\
\hdashline $4$&$\frac{1}{3}$ &$V(x_0x_2+x_1^2,x_0x_3^3+2\alpha x_1x_2x_3^2-\beta x_2^3x_3)$ & 
 $(3,1,-1,-3)$   &$J_{3,0}$&$A_3$\\
\hline
$5$&$\frac{5}{14}$&$V(x_0x_2+x_1^2,x_0x_3^3+x_2^4)$& $(17,5,-7,-15)$  &$E_{14}$&$A_4$, $\cC_v C^{*}_k=2T_v(L)$\\
\hline
$6$&$\frac{3}{8}$&$V(x_0x_2+x_1^2,x_0x_3^3+x_1x_2^2x_3)$& $(11,3,-5,-9)$  & $E_{13}$& $A_5$, $C^{*}_k\supset L$\\
\hline
$7$&$\frac{2}{5}$&$V(x_0x_2+x_1^2,x_0x_3^3+x_1x_2^3)$& $(4,1,-2,-3)$  & $E_{12}$& $A_7$, $C^{*}_k\supset L$\\
\hline
\end{tabular}
\vspace{0.2cm}
\caption{Potential critical values for the VGIT $\gM(t)$}\label{table101}
\end{table}

(In the row corresponding to $k=4$, $(\alpha,\beta)$ and $(1,1)$ are linearly independent. The imprecise notation $2T_v(L)$ means the tangent cone at $v$ of  $V(x_0,x_0x_2+x_1^2)=V(x_0,x_1^2)$.)
Before proving~\Ref{prp}{indentifyt}, we go through a few auxiliary results.
\begin{lemma}\label{lmm:planarsing}
Let $x=V(f_2,f_4)\in U$. If 
\begin{enumerate}
\item
$f_2$ has rank at most $2$, or
\item
there exists a point $p\in V(f_2,f_4)$ which is  singular both for $V(f_2)$ and $V(f_4)$, 
\end{enumerate}
then $x$ is $t$-unstable for all $t\in(\delta,1/2)$. 
\end{lemma}
\begin{proof} 
If~(1) holds, the proof is  similar to that of~\cite[Prop. 4.6]{g4git}.  If~(2) holds, the proof is similar to that of~\cite[Prop. 4.7]{g4git}. For the reader's convenience, we sketch the arguments. We may assume that 
$f_2=x_3 x_4$, or $f_2=x_4^2$. Let $\lambda$ be the $1$-PS $\diag(1,1,-1,-1)$. 
Then $\mu(f_2,\lambda)=-2$, $\mu(f_4,\lambda)\le 4$. Hence
$$\mu(f_2,\lambda)+t \mu(f_4,\lambda)\le -2+4t,$$
 which is strictly negative for $t\in(\delta, 1/2)$. By~\eqref{benfor} this would conclude the proof if we were working on $\PP E$ (by~\Ref{thm}{thmgitmodels} it is a complete proof for $t\in(\delta,1/3)$). Since we are working on $\cP$, in order to finish the proof it will suffice to show that
\begin{equation}\label{nonesce}
\lim\limits_{s\to 0}\lambda(s) V(f_2, f_4)\in U.
\end{equation}
If $f_4(x_0,x_1,0,0)\not=0$,  we have
 $$\lim\limits_{s\to 0}V(f_2, f_4)=V(f_2, f_4(x_0,x_1,0,0))\in U.$$ 
This proves that if $f_4$ is generic,  then $ V(f_2, f_4)$ is $N_t$-unstable. Since the locus of $N_t$-unstable points of $\cP$ is closed, it follows that  $V(f_2, f_4)$ is $N_t$-unstable for all $f_4$. 

If Item~(2) holds, we choose homogeneous coordinates such that $p=[1,0,0,0]$. Then $f_2\in\CC[x_1,x_2,x_3]_2$ and $f_4=x_0^2 q_2+x_0 q_3+q_4$, where  $q_d\in\CC[x_1,x_2,x_3]_d$. 
Let $\lambda$ be the $1$-PS $\diag(3,-1,-1,-1)$. 
Then 
$$\mu(f_2,\lambda)+t \mu(f_4,\lambda)\le -2+4t,$$
 which is strictly negative for $t\in(\delta, 1/2)$. 
Arguing as above, we conclude that $V(f_2,f_4)$ is $N_t$-unstable. Alternatively, by the previous case we may assume that $f_2$ is irreducible, and then~\eqref{nonesce} holds for any choice of $f_4\not=0$.
\end{proof}
\begin{remark}\label{rmk:stabcomm}
We will repeatedly use the function $\mu(f_2,\lambda)+t \mu(f_4,\lambda)$ to destabilize curves $C=V(f_2,f_4)$ at specific values of $t$. An attentive reader might notice that this is in fact different from the numerical function $\mu^{N_t}\left((f_2,f_4),\lambda\right)$  that we should use in the application of the numerical criterion for $\cP$ with linearization $N_t$. In fact, $\mu(f_2,\lambda)+t \mu(f_4,\lambda)$ is the numerical function (for the linearization $\eta+t\xi$) on the $\bP E$ model of $\cP$. The point however is that $\mu^{N_t}\left((f_2,f_4),\lambda\right)=\mu(f_2,\lambda)+t \mu(f_4,\lambda)$ as long as 
\begin{equation}\label{conditionstar}
\lim_{t\to 0}\lambda(t)V(f_2,f_4)\in U
\end{equation} (recall $U$ is a common open subset of $\bP E$ and $\cP$, and that the linearizations agree over $U$, cf. \Ref{crl}{chowlb}); this will be always the case in our computations. In other words, assuming \eqref{conditionstar}, our arguments saying $\mu(f_2,\lambda)+t \mu(f_4,\lambda)<0$ implies $t$-unstable are valid. (In \cite{g4git}, we have used $\mu(f_2,\lambda)+t \mu(f_4,\lambda)<0$ to define {\it numerically unstable} (\cite[Def. 4.2]{g4git}), and then  argued that numerically unstable implies unstable (e.g. \cite[Prop. 6.2]{g4git}). The arguments here are similar, but without explicitly defining numerical stability. The key point here (as well as in \cite{g4git}) is \Ref{prp}{propssci} which guarantees that in the cases of interest, the condition \eqref{conditionstar} is satisfied.)
\end{remark}
 \begin{lemma}\label{lmm:zerozero}
Keeping notation as above, let $t_0$ be a critical value  for the VGIT $\gM(t)$. Let $C\in \cP^{ss}(N_{t_i})^{\rm new}$  be a minimal orbit (notice that $\cP^{ss}(N_{t_i})\subset U$  by~\Ref{prp}{doppiovu}). Then, there exists  $\lambda$ $1$-PS of $\PGL(4)$ stabilizing $C$ and equations $C=V(f_2,f_4)$ such that 
\begin{equation}\label{nonzero}
\mu(f_2,\lambda)\not=0,\qquad  \mu(f_4,\lambda)\not=0.
\end{equation}
(and $\mu(\cdot,\lambda)$ is minimized by $f_4$ among representatives of $f_4\pmod{f_2}$).
\end{lemma}
\begin{proof}
Let $t_0\in(\delta,1/2)$ be a critical value, and $C$ an associated critical polystable orbit. Since, the stability of $C$ changes at $t_0$, it is clear that we can choose a $1$-PS $\lambda$ and equations $C=V(f_2,f_4)$ such that $\mu^t((f_2,f_4),\lambda)=\mu(f_2,\lambda)+t\mu(f_4,\lambda)$ and that $\mu^t((f_2,f_4),\lambda)$ changes sign at $t_0\neq 0$. It follows that conditions \eqref{nonzero} are satisfied. 
Finally, the condition $\mu^{t_0}((f_2,f_4),\lambda)=0$ means that 
$$\lim_{s\to 0} \lambda(s)\cdot (f_2)^{\otimes n} \otimes (f_4)^{\otimes m}$$
(for some integers $n,m$ with $t_0=\frac{m}{n}$) exists and it is non-zero (compare \eqref{eqlimittensor} below). Replacing $(f_2,f_4)$ by the limit $(\overline f_2,\overline f_4)$,  we get that $\lambda$ stabilizes $V(\overline f_2,\overline f_4)$ and that $(f_2,f_4)$ and $(\overline f_2,\overline f_4)$ are in the same $\SL(4)$-orbit (since the orbit is closed). Finally, $\overline f_2=\lim_{s\to 0} s^{c m}\cdot \lambda(s)\cdot f_2$ and $\overline f_4=\lim_{s\to 0} s^{-cn}\cdot \lambda(s)\cdot f_4$ for appropriate $n$ and $m$  as before (and a constant c). We get $\mu(\overline f_k,\lambda)=\mu(f_k,\lambda)$ for $k=2,4$ (i.e. the monomial computing the $\lambda$-weight agree for $\overline f_k$ and $f_k$). 
\end{proof}

\begin{proof}[Proof of~\Ref{prp}{indentifyt}]
 Let $t_0\in(\delta,1/2)\cap\QQ$ be a critical value   for the VGIT $\gM(t)$, and let $V(f_2,f_4)\in U(N_{t_0})^{ss}$ be a critical curve for  $t_0$. Thus there exists a $1$ PS $\lambda$ of $\SL(4)$ stabilizing $x$, i.e.~$\lambda(s) f_2=s^m f_2$ and 
 $\lambda(s) f_4\equiv s^n f_4 \pmod{f_2}$ for some $m,n$. Replacing $f_4$ by a suitable multiple of $f_2$, we may assume that  
  $\lambda(s) f_4= s^n f_4$. Since $x$ is $N_{t_0}$-semistable, $\lambda$ acts trivially on the fiber of $\cO_{\cP}(N_{t_0})$ over $x$, i.e.
  \begin{equation}\label{tizero}
\mu(f_2,\lambda)+t_0 \mu(f_4,\lambda)=0.
\end{equation}
(Notice that $\mu(f_2,\lambda)=m$ and $\mu(f_4,\lambda)=n$, where $m,n$ are as above.) By~\Ref{lmm}{zerozero} we know that~\eqref{nonzero} holds. Since a smooth quadric is semistable, it follows that $f_2$ is degenerate. On the other hand, we may suppose  $f_2$ has rank at least $3$ 
by~\Ref{lmm}{planarsing}, and hence it has rank equal to $3$. A straightforward argument shows that there exist  coordinates $(x_0,\ldots,x_3)$ on $\CC^4$  such that 
$$\lambda(s)=\diag(s^{r_0},\ldots,s^{r_3}),\qquad f_2=x_0 x_2+x_1^2.$$
 Since  $\lambda(s) f_2=s^m f_2$, we have 
 $$2 r_1= r_0+r_2.$$
 It follows that
 $$3r_1+r_3=0$$
 because $r_0+\ldots+ r_3=0$. 
 By interchanging $\lambda$ and $\lambda^{-1}$, we can assume $r_1\ge 0\ge r_3$. Interchanging the variables $x_0$ and $x_2$, we can assume $r_0\ge r_1=\frac{r_0+r_2}{2}\ge r_2$ (in particular, $r_0>0$).   At this point  we may rescale  the $r_i$'s so that $r_0=1$ (we will get a virtual $1$ PS, it makes no difference as far as our proof is concerned), and we get
 $r_0=1$, $r_1=\alpha$, $r_2=2\alpha-1$, $r_3=-3\alpha$, where $\alpha\in[0,1]\cap\QQ$. 
Thus we let $\lambda_{\alpha}$ be the virtual $1$ PS 
\begin{equation}\label{alphaps}
\lambda_{\alpha}(s):=\diag(s,s^{\alpha},s^{2\alpha-1},s^{-3\alpha}), \qquad \alpha\in[0,1]\cap\QQ.
\end{equation}
Consider separately the two cases:
\begin{enumerate}
\item
$f_4$ is a multiple of a monomial.
\item
$f_4$ is not a multiple of a monomial.
\end{enumerate}
Suppose that Item~(1) holds.  The numerical function $\mu(f_4,\lambda_{\alpha})$ is a polynomial in $\alpha$ of degree $1$:
\begin{equation}
\mu(f_4,\lambda_{\alpha})=c\alpha+d,\qquad c,d\in\QQ,\quad c\not=0.
\end{equation}
We are assuming that~\eqref{tizero} holds for a certain $t_0\in[0,1/2]\cap\QQ$ and $\lambda=\lambda_{\alpha_0}$. Since $\mu(f_2,\lambda_{\alpha})=-2\alpha$, it follows that $c>0$ and $d=0$. In fact, if $d=0$ and $c\le 0$, then clearly~\eqref{tizero} cannot hold for $t_0>0$, and if   $d\not=0$,  then there exist (many!) values of $\alpha\in[0,1]\cap\QQ$ with the property that $\mu(f_2,\lambda_{\alpha})+t_0 \mu(f_4,\lambda)<0$ or 
$\mu(f_2,\lambda_{\alpha})+t_0 \mu(f_4,\lambda)>0$, i.e.~(after reparametrization, if $\lambda_{\alpha}$ is a virtual $1$ PS) $\lambda_{\alpha}$ fixes $V(f_2,f_4)$ and acts non trivially on the fiber of $\cO_{\cP}(N_{t_0})$ over $V(f_2,f_4)$. That is a contradiction because $V(f_2,f_4)$ is assumed to be $N_{t_0}$-semistable. This proves that $c>0$ and $d=0$. A straightforward computation then shows that, after rescaling, $f_4\in\{x_3^4,x_1x_3^4,x_1^2x_3^2,x_0x_2x_3^2\}$. The critical value for $f_4=x_3^4$ is $t_0=1/6$, 
the critical value for $f_4=x_1 x_3^3$ is $t_0=1/4$, while $f_4\in\{x_1^2x_3^2,x_0x_2x_3^2\}$ is impossible, because 
of~\Ref{lmm}{planarsing}.

Now suppose that Item~(2) holds.  Thus in the expansion of 
 $f_4$ there are two (at least)  monomials $x_0^{i_0}\dots x_3^{i_3}$ and $x_0^{j_0}\dots x_3^{j_3}$ with non-zero coefficients. Let  $k_l=i_l-j_l$ for $l=0,\dots, 3$. Since $\lambda(s) f_4=s^n f_4$ for some $n$, we have
\begin{equation}\label{eqalpha}
k_0+\alpha k_1+(2\alpha-1) k_2-3\alpha k_3=0.
\end{equation}
The above equation determines $\alpha$, and hence we get $t_0$ upon replacing  $\lambda$ by $\lambda_{\alpha}$ in~\eqref{tizero}, and solving  for $t_0$. We get the potential critical values in~\eqref{valcrit} other than $1/6$ by listing all couples of degree $4$ monomials and going through the steps described above (or programming a computer to do it in our place). Once we have the potential critical values, it is clear how to compute the potential critical curves associated to each (potential) critical value.  
\end{proof}

\subsection{Relations between singularities of $C$ and $N_t$-(semi)stability}\label{selementaryss}  
The following proposition is an adaptation to the case of singular quadrics of some of the content of ~\Ref{prp}{git44}. 
\begin{proposition}\label{prp:gitcone}
Let $C\in U$, and assume that   $C$ has consecutive triple points at $p$, and that it has a significant limit singularity at $p$.  Let $Q$ be the unique quadric containing $C$. Suppose that $Q$    is a quadric cone, and that it is smooth at $p$. 
Lastly, let $t\in(\delta,1/2)$, and suppose that  $C$ is $N_t$-semistable.
Then  the tangent cone to $C$ at $p$ is \emph{not} equal to $3T_p(L)$, for  $L$  the unique line in $Q$ through $p$. 
\end{proposition}
\begin{proof}
We may choose homogeneous  coordinates 
$[x_0,\ldots,x_4]$ so that $p=[1,0,0,0]$, and the quadric containing $C$ is $V(x_0x_2+x_1^2)$.   
Then, since $\mult_p(C)=3$ and the tangent cone to $C$ at $p$ is equal to $3T_p V(x_1,x_2)$, we have $C=V(f_2,f_4)$ where
\begin{equation}
f_2=x_0x_2+x_1^2,\qquad f_4= x_0 x_1^3+P_4(x_1,x_2,x_3), \quad P_4\in\CC[x_1,x_2,x_3]_4.
\end{equation}
Moreover, since  $C$ has a significant limit singularity at $p$, the monomials $x_3^4, x_1 x_3^3, x_1^2 x_3^2$ are missing in $P_4$. In fact, by~\Ref{thm}{recognize},  if $x_3^4$ appears then $C$ has an $E_6$ singularity at $p$, if  $x_3^4$ does not appear and $x_1 x_3^3$  appears then $C$ has an $E_7$ singularity at $p$,  if  $x_3^4,x_1 x_3^3$ do not appear and $x_1^2 x_3^2$  appears then $C$ has a $J_{2,k}$ (including $k=\infty$) singularity at $p$.

Now let $\lambda$ be the $1$ PS defined by $\lambda(s)=\diag(s^5, s^{-1},s^{-7},s^3)$. Since 
\begin{equation*}
\mu(f_2,\lambda)+t\mu(f_4,\lambda)\le -2+2t,
\end{equation*}
which is strictly negative  for $t\in(\delta,1/2]$, we are done.
\end{proof}
\begin{proposition}\label{prp:destabilizing}
Let $C\in U$, and let $p\in C$ be a point contained in the smooth locus of the unique quadric containing $C$. Suppose  that $C$ has a singularity at $p$  appearing 
in~\eqref{posizione}, with tag $k$. Then $C$ is $N_t$-unstable for $t\in(t_k,1/2]$, and it is $N_t$-desemistabilized by a $1$-PS conjugated to the one appearing in the row of Table~\ref{typeinp} with index $k$. 
\end{proposition}
\begin{proof} 
Let us assume that $\mult_p(C)=4$, i.e.~we are in one of the first two cases in Table~\ref{typeinp}. 
We may choose homogeneous  coordinates 
$[x_0,\ldots,x_4]$ so that $p=[1,0,0,0]$ and $C=V(f_2,f_4)$ where
\begin{equation}
f_2=x_0x_2+x_1^2+ax_3^2,\qquad  a\in\CC,\quad f_4\in\CC[x_1,x_2,x_3]_4.
\end{equation}
In fact, choose coordinates so that $f_2$ is as above, and let $C=V(f_2,\wt{f}_4)$. Then, since $\mult_p(C)=4$, we can add a suitable multiple of $f_2$ to $\wt{f}_4$ so that we get a quartic polynomial in $x_1,x_2,x_3$. 
 Choose affine coordinates $x_i/x_0$ around $p$, i.e.~set $x_0=1$. Then $(x_1,x_3)$ are local coordinates on $V(f_2)$ centered at $p$, and  we have an embedding  $\cC_p(C)\subset\aff^2$ as the cone $V(f_4(x_1,0,x_3))$. It follows that
\begin{enumerate}
\item
  if $\cC_p(C)=4A$ then \emph{in the generic case} we may make another change of coordinates so that 
 \begin{equation*}
f_4=x_3^4+x_2 g_3,\qquad g_3\in\CC[x_1,x_2,x_3]_3,
\end{equation*}
\item
 if $\cC_p(C)=3A+B$ (with $A\not=B$) then \emph{in the generic case} we may make another change of coordinates so that 
 \begin{equation*}
f_4=x_3^3(c x_3+x_1)+x_2 g_3,\qquad c\in\CC,\quad g_3\in\CC[x_1,x_2,x_3]_3.
\end{equation*}
\end{enumerate}
Since the locus of $N_t$-unstable points is closed, it will suffice to prove $N_t$-unstability if~(1) or~(2) above holds. Let $\lambda_{\alpha}$ be the virtual $1$ PS in~\eqref{alphaps} with $\alpha\le 1/5$. Then the exponents are ordered as follows:
\begin{equation}\label{ordinamento}
1\ge\alpha\ge -3\alpha\ge 2\alpha-1.
\end{equation}
Now suppose  that~(1) holds. Then
\begin{equation}
\mu(f_2,\lambda_{\alpha})+t\mu(f_4,\lambda_{\alpha})=2\alpha+t\max\{-12\alpha,5\alpha-1\}.
\end{equation}
Let $0<\alpha\le 1/17$. Then   $\mu(f_2,\lambda_{\alpha})+t\mu(f_4,\lambda_{\alpha})=2\alpha(1-6t)$, and hence it is strictly negative for $t>1/6$. Since $\lambda_{\alpha}(s)f_2=s^{2\alpha}f_2$, this proves that $C$ is $N_t$-unstable if $t\in(1/6,1/2]\cap\QQ$, see~\Ref{rmk}{stabcomm}. 

Next suppose  that~(2) holds. Then
\begin{equation}
\mu(f_2,\lambda_{\alpha})+t\mu(f_4,\lambda_{\alpha})=2\alpha+t\max\{-8\alpha,5\alpha-1\}.
\end{equation}
Let $0<\alpha\le 1/13$. Then   $\mu(f_2,\lambda_{\alpha})+t\mu(f_4,\lambda_{\alpha})=2\alpha(1-4t)$, and hence it is strictly negative for $t>1/4$. Since $\lambda_{\alpha}(s)f_2=s^{2\alpha}f_2$, this proves that $C$ is $N_t$-unstable if $t\in(1/4,1/2]\cap\QQ$. 

Now we suppose that the singularity of $C$ at $p$ appears in one of the remaining rows of Table~\ref{typeinp}. 
By \Ref{lmm}{norm3}, we may choose homogeneous  coordinates 
$[x_0,\ldots,x_4]$ so that $p=[1,0,0,0]$ and $C=V(f_2,f_4)$ where
\begin{eqnarray}
f_2&=&x_0x_2+x_1^2+ax_3^2 \nonumber\\
f_4&=& x_0x_3^3+x_1g_3(x_2,x_3)+g_4(x_2,x_3)\label{eqf4normal}
\end{eqnarray}
Let $\lambda$ be the $1$ PS appearing in Table~\ref{typeinp} on the corresponding row. Then $\mu(f_2,\lambda)$ and  $\mu(f_2,\lambda)$ are as in Table~\ref{pendenze}. If $t>t_k$, then $\mu(f_2,\lambda)+t\mu(f_2,\lambda)<0$, and hence  $C$ is $N_t$-unstable because $\lambda(s)f_2=s^{2r_1}f_2$, where $\lambda=\diag(r_0,r_1,r_2,r_3)$.

 \begin{table}[htb!]
\renewcommand\arraystretch{1.5}
\begin{tabular}{|c|c|l|c|c|}
\hline
$k$& Sing.~type at $p$  & Desemistabilizing $1$ PS & $t_k$\\
\hline
0& $\mult_p(C)=4$, $\cC_p(C)=4A$ & $\lambda_{\alpha}$, $0<\alpha\le 1/17$ &$\frac{1}{6}$\\
1&    $\mult_p(C)=4$, $\cC_p(C)=3A+B$   &   $\lambda_{\alpha}$, $0<\alpha\le 1/13$  &$\frac{1}{4}$\\ 
2&$J_{4,\infty}$&$(7,3,-1,-9)$&$\frac{3}{10}$\\ 
3&$J_{3,r}$&$(3,1,-1,-3)$&$\frac{1}{3}$\\
5&$E_{14}$&$(17,5,-7,-15)$&$\frac{5}{14}$\\ 
6&$E_{13}$&$(11,3,-5,-9)$&$\frac{3}{8}$\\ 
7&$E_{12}$&$(4,1,-2,-3)$&$\frac{2}{5}$\\
\hline
\end{tabular}
\vspace{0.2cm}
\caption{Singularities that eventually get desemistabilized}\label{typeinp}
\end{table}

 \begin{table}[htb!]
\renewcommand\arraystretch{1.5}
\begin{tabular}{|c|c|c|c|c|c|}
\hline
 Sing.~type at $p$&$J_{4,\infty}$&$J_{3,r}$&$E_{14}$&$E_{13}$&$E_{12}$\\
 \hline
 $\mu(f_2,\lambda)$ &$6$& $2$ & $10$ &$6$ &  $2$\\
 \hline
$\mu(f_4,\lambda)$&  $-20$&  $-6$ & $-28$ &  $-16$ &  $-5$ \\
\hline
$t_k$&$\frac{3}{10}$&$\frac{1}{3}$&$\frac{5}{14}$&$\frac{3}{8}$&$\frac{2}{5}$\\
\hline
\end{tabular}
\vspace{0.2cm}
\caption{}\label{pendenze}
\end{table}

\end{proof}
 \subsection{Proof of the first main result}\label{subsec:dimouno}
In the present subsection we prove~\Ref{thm}{improving}. The following key remark (which follows from Arnold's results) wil be useful.
\begin{remark}\label{rmk:quintaa}
A non slc singularity which is an arbitrary small deformation of a singularity appearing in~\eqref{posizione}  is  again a singularity appearing in~\eqref{posizione}.
\end{remark}
\begin{proof}[Proof of~\Ref{thm}{improving}.]
 Let us prove that Item~(1) holds for $k=0$. Let $C$ be 
$N_{1/6}$-semistable. By the classification of potential critical values of the VGIT $\gM(t)$, i.e.~\Ref{prp}{indentifyt}, either $C$ is $N_{1/6-\epsilon}$-semistable, or $C\in\cP^{ss}(N_{1/6})^{\rm new}$. In the former case $C$ sits on a smooth quadric, and defines a semistable point of $\gM$ by~\Ref{thm}{thmgitmodels}, hence 
Item~(1) holds by~\Ref{crl}{norm3}. In the latter case, the closure of the orbit of $C$ contains the curve  $C_0^{*}$ in Table~\ref{table101}. Since $C$ is a quadruple conic, Item~(1) follows from~\Ref{rmk}{quintaa}. 

Let us prove that Item~(2) holds for $k=0$. Let $C$ be $N_{t}$-semistable, where $1/6<t<1/4$. By~\Ref{prp}{indentifyt}, either $C$ is $N_{1/6-\epsilon}$-semistable, or $C\in\cP^{ss}(N_{1/6})^{\rm new}$. Thus, every non slc singularity of $C$ appears in~\eqref{posizione}. Moreover, by~\Ref{prp}{destabilizing} $C$ is not a quadruple conic. This proves that 
 Item~(2) holds for $k=0$. 
 
 Let $k_0\in\{1,\ldots,7\}$, and assume that Items~(1) and~(2) hold for all $0\le k<k_0$. We prove that  Items~(1) and~(2) hold for $k=k_0$.
 
(1): Let $C$ be $N_{t_{k_0}}$-semistable. Then either $C$ is $N_{t_{k_0}-\epsilon}$-semistable, or $C\in\cP^{ss}(N_{t_{k_0}})^{\rm new}$. In the former case, Item~(1) holds for $k=k_0$ because Item~(2) holds for $k=k_0-1$. In the latter case, the closure of the orbit of $C$ contains the curve $C_{k_0}^{*}$ in Table~\ref{table101} (if $k_0=3$ there is more than one choice for $C_{3}^{*}$). Since the unique non slc singularity of the curve in  Table~\ref{table101}  has tag $k_0$, Item~(1) holds by~\Ref{rmk}{quintaa}. 

(2): Let $C$ be $N_{t}$-semistable, where $t_{k_0}<t<t_{k_0+1}$. By~\Ref{prp}{indentifyt}, either $C$ is $N_{t_{k_0}-\epsilon}$-semistable, or 
$C\in\cP^{ss}(N_{t_{k_0}})^{\rm new}$. In the former case, Item~(2) holds for $k=k_0$ because Item~(2) holds for $k=k_0-1$. In the latter case, 
arguing as above, we get that the non slc singularities of $C$ have tag at least $k_0$. On the other hand $C$ does not have  singularities with tag $k_0$ 
by~\Ref{prp}{destabilizing}.
\end{proof}
 \subsection{Basin of attraction for the potential semistable orbits}\label{subsec:sectbasin}
 We recall the following general VGIT behavior: assume that $x$ changes stability (say goes from $t$-semistable to $t$-unstable) at some critical slope $t$ (or wall). Then there exist some $x^*$ which gives a minimal orbit at $t$ such that $\overline{G\cdot x}\supset G\cdot x^*$. As always, the orbit closures can be tested by $1$-PS subgroups. This leads to the notion of basin of attraction, which plays an important role in the GIT analyses related to the Hassett-Keel program (e.g. \cite{hh2}). 
 
 Let $x^*\in \cP$ and $t\in(\delta,1/2)\cap\QQ$; we set
 \begin{equation}
G_{N_t}(x^{*}):=\{g\in\SL(4) \mid \text{$g(x^*)=x^{*}$ and $g$ acts trivially on the fiber of $N_t$ at $x^{*}$}\}.
\end{equation}
 \begin{definition}
Let $x^*\in \cP$, and let $t\in(\delta,1/2)\cap\QQ$. Suppose that  $\lambda$ is a $1$ PS of $\SL(4)$ stabilizing $x^{*}$ and acting trivially on the fiber of $N_t$ at $x^{*}$ (the last hypothesis is satisfied if $x^{*}$ is $N_t$-semistable).   
    The $\lambda$-\emph{basin of attraction of $x^{*}$} is equal to 
 $$A_{\lambda}(x^*)=\{x\in \cP\mid \lim_{s\to 0} \lambda(s)\cdot x=x^* \}.$$
The \emph{basin of attraction of $x^{*}$} is equal to 
 $$A(x^*)=\{x\in \cP\mid \text{$\lim_{s\to 0} \lambda(s)\cdot x=x^*$ for some $1$ PS $\lambda$ of $G_{N_t}(x^*)$} \}.$$
 \end{definition}
\begin{remark}\label{rmk:sollevo}
Suppose that $x\in A_{\lambda}(x^*)$, and let $\wt{x}$ be a \emph{non zero} element of the fiber of $N_t$ at $x$. Then, since the action of  $\lambda$  on the fiber of $N_t$ at $x^{*}$ is trivial, 
$\lim_{s\to 0}\lambda(s)\wt{x}$ exists and is a non zero element of the fiber of $N_t$ at $x^*$.
\end{remark}

 \subsubsection{Transition at $t=1/6$}
Let $C^{*}=V(f_2,f_4)$, where
\begin{equation}
f_2=x_0 x_2 +x_1^2,\qquad f_4=x_3^4.
\end{equation}
\begin{proposition}\label{prp:unsesto}
Keep notation as above, and let $C\in\cP^{ss}(N_t)$ be in  the basin of attraction of $C^{*}$. The following hold:
\begin{enumerate}
\item
If  $t\in (\delta,1/6)\cap\QQ$, then $C$ has a point $p$ of multiplicity $4$, belonging to the smooth locus of the unique quadric containing $C$, and such that $\cC_p(C)=4A$, i.e.~the singularity of $C$ at $p$ is as in the first row of Table~\ref{typeinp}. 
\item
If  $t\in (1/6,1/4)\cap\QQ$ then $C=Q\cap S$, where $Q$ is an irreducible quadric, and $S$ is a quartic surface  not containing singular points of $Q$. 
\end{enumerate}
\end{proposition}
\begin{proof}
For $\alpha\in[0,1]\cap\QQ$, let $\lambda_{\alpha}$ be the virtual $1$ PS of $\SL(4)$ given by
\begin{equation}\label{lamalfa}
\lambda_{\alpha}(s)=(s,s^{\alpha},s^{2\alpha-1},s^{-3\alpha}).
\end{equation}
Every virtual $1$ PS fixing $C^{*}$ is equal to $\lambda_{\alpha}$ for some $\alpha\in[0,1]\cap\QQ$. 
Thus $C\in A_{\lambda_{\alpha}^{\pm 1}}(C^{*})$ for some  $\alpha\in[0,1]\cap\QQ$. 
\medskip

\noindent{\it $\lambda_{\alpha}$-basin of attraction of $C^{*}$:\/}  We determine which $(2,4)$-curves $C=V(f_2+f_2',f_4+f'_4)$ (where $f'_d\in\CC[x_0,\ldots,x_3]_d$)
are in the $\lambda_{\alpha}$-basin of attraction  of $C^{*}$ ($N_t$-semistable points of $\cP$ are actual $(2,4)$ curves by~\Ref{prp}{doppiovu}). 
The fiber of $\cO_{\cP}(N_{1/6})^{\otimes -6}$ at $C$ is identified with $(f_2+f_2')^{\otimes 6}\otimes (f_4+f'_4)$, and we must determine for which $(f'_2,f'_4)$ we have
\begin{equation}\label{eqlimittensor}
\lim\limits_{s\to 0}\lambda_{\alpha}(s)\left((f_2+f_2')^{\otimes 6}\otimes (f_4+f'_4)\right)=f_2^{\otimes 6}\otimes f_4.
\end{equation}
(See~\Ref{rmk}{sollevo}.) Now
\begin{equation}
\lambda_{\alpha}(s)\left((f_2+f_2')^{\otimes 6}\otimes (f_4+f'_4)\right)=(f_2+s^{2\alpha}\lambda_{\alpha}(s)f'_2)^{\otimes 6}
\otimes (f_4+s^{-12\alpha}\lambda_{\alpha}(s)f'_4).
\end{equation}
and hence  $C$ 
is in the basin of attraction  of $C^{*}$ for $\lambda_{\alpha}$ if and only if
\begin{equation}\label{duelimiti}
\lim\limits_{s\to 0}s^{2\alpha}\lambda_{\alpha}(s)f'_2=0,\qquad \lim\limits_{s\to 0}s^{-12\alpha}\lambda_{\alpha}(s)f'_4=0.
\end{equation}
Now notice that the ordering of the weights of $\lambda_{\alpha}$ changes as we cross the value $\alpha=1/5$. In fact
\begin{eqnarray}
1\ge \alpha\ge 2\alpha-1\ge -3\alpha & \text{if $1/5\le\alpha$,} \label{piuquinto}\\
1\ge \alpha\ge -3\alpha\ge 2\alpha-1 & \text{if $0\le\alpha \le 1/5$.} \label{menoquinto}
\end{eqnarray}
It follows that 
\begin{equation}\label{duecasi}
f'_2\in
\begin{cases}
\la x_0x_3,x_1 x_2,x_1 x_3,x_2^2,x_2 x_3,x_3^2\ra & \text{if $1/5<\alpha$,} \\
\la x_1 x_2,x_1 x_3,x_2^2,x_2 x_3,x_3^2\ra & \text{if $0\le\alpha \le 1/5$.} \\
\end{cases}
\end{equation}
Thus  $p=[1,0,0,0]$ is a smooth point of $Q':=V(f_2+f_2')$, and    local parameters on  $Q'$ around $p$ are $(x_1|_{Q'},x_3|_{Q'})$. Moreover  \emph{in $\cO_{Q',p}$} the following holds:
\begin{equation}\label{aneloc}
x_2|_{Q'}\equiv 
\begin{cases}
bx_3|_{Q'}\pmod{\gm_p^2},\ b\in\CC & \text{if $1/5<\alpha$,} \\
0\pmod{\gm_p^2} & \text{if $0\le\alpha\le 1/5$.}
\end{cases}
\end{equation}
On the other hand, if $1/5\le\alpha\le 1$, then the second equation in~\eqref{duelimiti} holds if and only if $f'_4=0$, while if $0\le\alpha<1/5$ and it holds for $f'_4$, then
\begin{equation}\label{rebus}
f'_4=x_2 P_3(x_1,x_2,x_3)+x_0 x_2^2 P_1(x_1,x_2,x_3),
\end{equation}
 where $P_d\in\CC[x_1,x_2,x_3]_d$. It follows (for all $\alpha$) that $C$ has  multiplicity $4$ at $p=[1,0,0,0]$, and that the tangent cone at $p$ is equal to $4V(x_3)$.

\medskip

\noindent{\it $\lambda^{-1}_{\alpha}$-basin of attraction of $C^{*}$:\/}  
 Let $C=V(f_2+f_2',f_4+f'_4)$. Arguing as above, we see that  $C$ 
is in the basin of attraction  of $C^{*}$ for $\lambda^{-1}_{\alpha}$ if and only if
\begin{equation}\label{ancoradue}
\lim\limits_{s\to 0}s^{-2\alpha}\lambda_{\alpha}(s^{-1})f'_2=0,\qquad \lim\limits_{s\to 0}s^{12\alpha}\lambda_{\alpha}(s^{-1})f'_4=0.
\end{equation}
It follows that 
\begin{equation}\label{altridue}
f'_2\in
\begin{cases}
\la x^2_0,x_0 x_1\ra & \text{if $1/5\le \alpha\le 1$,} \\
\la x^2_0,x_0 x_1, x_0 x_3\ra & \text{if $0\le\alpha < 1/5$,} \\
\end{cases}
\end{equation}
and hence $V(f_2+f'_2)$ has rank $3$. Moreover  $p\notin V(f_4+f'_4)$; in fact $C=V(f_2+f'_2,f_4+f'_4)$ is projectively equivalent to curves arbitrarily close to $C^{*}=V(f_2,f_4)$, and since $C^{*}$ does not contain the vertex of $V(f_2)$, it follows that $C$ does not contain the vertex of $V(f'_2)$.

\medskip

Let us prove Item~(1). If $t\in(\delta,1/6)$, then $V(f_2)$ is a smooth quadric, and hence $C$ is in the $\lambda_{\alpha}$-basin of attraction of $C^{*}$. Then Item~(1) holds by~\Ref{prp}{gitcone}.  Now assume that $t=1/6$. If  $C$ is in the $\lambda_{\alpha}$-basin of attraction of $C^{*}$, the same argument applies. Thus we may assume that $C$ is in the $\lambda^{-1}_{\alpha}$-basin of attraction of $C^{*}$, and hence $V(f_2)$   is singular. By~\Ref{lmm}{planarsing} the rank of $f_2$ is equal to $3$, and hence there exist coordinates $(x_0,\ldots,x_3)$ such that $f_2=x_0 x_2+x_1^2$. Let $\lambda(s)=\diag(s^{-1},s^{-1},s^{-1},s^3)$. Then $\lim_{s\to 0}\lambda(s) f_2^{\otimes 6}\otimes f_4=f_2^{\otimes 6}\otimes f_4(0,0,0,x_3)$, and we are done. 

Let us prove Item~(2). Since curves in the $\lambda_{\alpha}$-basin of attraction of $C^{*}$ are $N_t$-semistable for $t\le 1/6$, it follows from general results that $C$  is in the $\lambda^{-1}_{\alpha}$-basin of attraction of $C^{*}$.  Thus Item~(2) holds by the analysis carried out above. 
\end{proof}
 \subsubsection{Transition at $t=1/4$}
Let $C^{*}=V(f_2,f_4)$, where
\begin{equation}
f_2=x_0 x_2 +x_1^2,\qquad f_4=x_3^3 x_1.
\end{equation}
\begin{proposition}\label{prp:unquarto}
Keep notation as above, and let $C\in\cP^{ss}(N_t)$ be in the basin of attraction of $C^{*}$. Let $Q$ be the unique quadric containing $C$. The following hold:
\begin{enumerate}
\item
If  $t\in (1/6,1/4)\cap\QQ$, then there exists a point $p\in C$ of multiplicity $4$, with  tangent cone $\cC_p(C)=3A+B$ ($A\not=B$), and such $Q$ is smooth point  at $p$, i.e.~the singularity of $C$ at $p$ is as in  the second row of Table~\ref{typeinp}. 

\item
If  $t\in (1/4,3/10)\cap\QQ$ then $Q$ is a quadric cone (of rank $3$ by~\Ref{lmm}{planarsing}), $C$ contains the vertex $v$ of $Q$, and has an $A_1$ singularity at $v$.
\end{enumerate}
\end{proposition}
 \begin{proof}
We proceed as in the proof of~\Ref{prp}{unsesto}. Let $\lambda_{\alpha}$ be the virtual $1$ PS of $\SL(4)$ given by~\eqref{lamalfa}. Every virtual $1$ PS fixing $C^{*}$ is equal to $\lambda_{\alpha}$ for some $\alpha\in[0,1]\cap\QQ$. 
Thus $C\in A_{\lambda_{\alpha}^{\pm 1}}(C^{*})$ for some  $\alpha\in[0,1]\cap\QQ$. 
\medskip

\noindent{\it $\lambda_{\alpha}$-basin of attraction of $C^{*}$:\/}  
 We have 
\begin{equation}
\lambda_{\alpha}(s)\left((f_2+f_2')^{\otimes 4}\otimes (f_4+f'_4)\right)=(f_2+s^{2\alpha}\lambda_{\alpha}(s)f'_2)^{\otimes 4}
\otimes (f_4+s^{-8\alpha}\lambda_{\alpha}(s)f'_4).
\end{equation}
and hence  $C$ 
is in the basin of attraction  of $C^{*}$ for $\lambda_{\alpha}$ if and only if
\begin{equation}
\lim\limits_{s\to 0}s^{2\alpha}\lambda_{\alpha}(s)f'_2=0,\qquad 
\lim\limits_{s\to 0}s^{-8\alpha}\lambda_{\alpha}(s)f'_4=0.
\end{equation}
The first equality holds if and only if~\eqref{duecasi} holds. Thus   $p=[1,0,0,0]$  is a smooth point of $V(f_2+f_2')$, and    local parameters on  $Q':=V(f_2+f_2')$ around $p$ are $(x_1|_{Q'},x_3|_{Q'})$. Moreover~\eqref{aneloc} holds. 

Next we consider $f'_4$. If $1/5\le \alpha\le 1$, then $f'_4\in\la x_3^4,x_2 x_3^3,x_2^2 x_3^2\ra$. By~\eqref{aneloc} it follows
 that in local coordinates $(x_1,x_3)$ the tangent cone $\cC_p(C)$ has equation $x_3^3(mx_3+x_1)$ for some $m\in\CC$.  If 
 $0\le \alpha< 1/5$, then $f'_4\in\la x_3^4,x_2 x_3^3,x_2^2 x_3^2\ra$.  On the other hand, if $0\le \alpha< 1/5$, then~\eqref{rebus} holds. By~\eqref{aneloc} it follows that in local coordinates $(x_1,x_3)$ the tangent cone $\cC_p(C)$ has equation $x_3^3 x_1$.

\medskip

\noindent{\it $\lambda^{-1}_{\alpha}$-basin of attraction of $C^{*}$:\/}  
The curve $C=V(f_2+f_2',f_4+f'_4)$ 
is in the basin of attraction  of $C^{*}$ for $\lambda^{-1}_{\alpha}$ if and only if
\begin{equation}\label{aridue}
\lim\limits_{s\to 0}s^{-2\alpha}\lambda_{\alpha}(s^{-1})f'_2=0,\qquad \lim\limits_{s\to 0}s^{8\alpha}\lambda_{\alpha}(s^{-1})f'_4=0.
\end{equation}
It follows that~\eqref{altridue} holds. Now suppose that $1/5\le\alpha\le 1$. Then, by~\eqref{altridue} $V(f_2+f'_2)$  has rank $3$, with singular point $v=[0,0,0,1]$, and
\begin{equation*}
f'_4= ax_0x_3^3+\sum\limits_{d=2}^4  g_{d} x_3^{4-d}, \qquad g_d\in \CC[x_0,x_1,x_2].
\end{equation*}
It follows by a straightforward computation that $V(f_2+f'_2,f_4+f'_4)$ has an $A_1$-singularity at $v$. 

If $0\le\alpha< 1/5$, then by~\eqref{altridue} $f_2+f'_2=x_0(ax_0+bx_1+x_2+cx_3)+x_1^2$ for some $a,b,c\in\CC$. Thus  $V(f_2+f'_2)$  has rank $3$, with singular point $v=[0,0,-c,1]$. By the second equation in~\eqref{aridue} we get  
\begin{equation*}
f'_4=x_0 g_3+x_1 h_3,\qquad g_3,h_3\in\CC[x_0,\ldots,x_3]_3,\quad h_3(0,0,-c,1)=0.
\end{equation*}
In particular $v\in V(f_4+f'_4)$, and hence $v\in C$. The (projective) tangent plane to  $V(f_4+f'_4)$ at $v$ has equation $g_3(0,0,-c,1)x_0+x_1=0$, and hence it intersects $V(f_2+f'_2)$ in the union of two \emph{distinct} lines. It follows that $C$ has an $A_1$ singularity at $v$.

\medskip

Now we are ready to finish the proof of the proposition. If $t\in[1/4,3/10)$, then  by~\Ref{prp}{destabilizing} $C$ cannot have a significant limit singularity of multiplicity $4$ belonging to the smooth locus of the unique quadric containing $C$. By the analysis above, it follows that if  if $t\in(1/6,1/4]$, then $C$ is in 
$\lambda_{\alpha}$-basin of attraction of $C^{*}$. It follows that  if  if $t\in(1/4,3/10)$, then $C$ is in 
$\lambda^{-1}_{\alpha}$-basin of attraction of $C^{*}$. 
 \end{proof}
 \subsubsection{Transition at the remaining potential critical values}
Table~\ref{leonardo}  will be handy.  Regarding that table: In the row corresponding to $k=4$, $(a,b)$ and $(1,1)$ are supposed to be  linearly independent, and (as in Table~\ref{table101}) $L$ is a line contained in $V(f_2)$.
\begin{table}[tbp]
\renewcommand\arraystretch{1.5}
\begin{tabular}{|c| l | l | l |c|c|}
\hline
\multirow{2}*{$k$} & \multirow{2}*{$C^{*}_k=V(f_2,f_4)$} & \multirow{2}*{$\alpha_k$} & Sing.~of $C^{*}_k$ away  & Type of $C^{*}_k$ at  & \multirow{2}*{$m$}\\
 & & & from vertex of $V(f_2)$ & vertex of $V(f_2)$ & \\
\hline
$2$ & $V(x_0x_2+x_1^2,x_0x_3^3+x_2^2x_3^2)$ & $3/7$ & $J_{4,\infty}$&$A_2$ & $2$ \\
\hline $3$  & $V(x_0x_2+x_1^2,x_0x_3^3+2x_1x_2x_3^2-x_2^3x_3)$& $1/3$ & $J_{3,\infty}$& $2$(twisted cubic) & $3$ \\
\hdashline $4$ &$V(x_0x_2+x_1^2,x_0x_3^3+2a x_1x_2x_3^2-b x_2^3x_3)$ &  $1/3$   &$J_{3,0}$&$A_3$ & $3$ \\
\hline
$5$ &$V(x_0x_2+x_1^2,x_0x_3^3+x_2^4)$& $5/17$  &$E_{14}$ & $A_4$, $\cC_v C=2T_v(L)$ & $4$ \\
\hline
$6$ & $V(x_0x_2+x_1^2,x_0x_3^3+x_1x_2^2x_3)$ & $3/11$  & $E_{13}$& $A_5$, $C^{*}_k\supset L$ & $5$ \\
\hline
$7$ & $V(x_0x_2+x_1^2,x_0x_3^3+x_1x_2^3)$& $1/4$  & $E_{12}$& $A_7$, $C^{*}_k\supset L$ & $7$ \\
\hline
\end{tabular}
\vspace{0.2cm}
\caption{}\label{leonardo}
\end{table}
\begin{proposition}\label{prp:basintk}
Let  $k\in\{2,\ldots,6,7\}$, and let $C^{*}_k$ is as in the row of Table~\ref{leonardo} corresponding to $k$. Suppose that   $C\in\cP^{ss}(N_t)$ is in the basin of attraction of $C^{*}_k$. Then the following hold:
\begin{enumerate}
\item
If $t\in (t_{k-1},t_k)\cap\QQ$,   there exists a point $p\in C$  such that the unique quadric containing $C$ is smooth  at $p$ and the singularity of $C$ at $p$ is of the same type as the unique singularity of $C_k^{*}$ away from the vertex of $V(f_2)$. 
\item
If $t\in (t_{k},t_{k+1})\cap\QQ$,  the unique quadric containing $C$ has rank $3$, and its vertex $v$ is contained in $C$. Moreover, if $k\not=3$, then  $C$ has the same type at $v$  as $C^{*}_k$ has at the vertex of $V(f_2)$. If $k=3$, then   $C$ has a singularity at $v$ of type $A_l$, where $l\ge 3$ (possibly $l=\infty$). 
\end{enumerate}
\end{proposition}
 \begin{proof}
Let $\alpha_k$ be as in  Table~\ref{leonardo}. Then $\lambda_{\alpha_k}$ is the component of the identity in the stabilizer of $C^{*}_k$ (consult Table~\ref{table101}). Thus $C =V(f_2+f'_2,f_4+f'_4)\in\cP^{ss}(N_t)$ is either in the $\lambda_{\alpha_k}$-basin of attraction of $C^{*}_k$ or in the $\lambda^{- 1}_{\alpha_k}$-basin of attraction of $C^{*}_k$. 

\medskip

\noindent{\it $\lambda_{\alpha_k}$-basin of attraction of $C^{*}_k$:\/}   We will prove that 
\begin{equation}\label{gallinari}
\text{if $C$ is in the $\lambda_{\alpha_k}$-basin of attraction of $C^{*}_k$, then Item~(1) holds.}
\end{equation}
Noting  that $t_k=2\alpha_k/(9\alpha_k-1)$, we get that $C$ is in the 
$\lambda_{\alpha_k}$-basin of $C^{*}_k$ if and only if 
\begin{equation}\label{deux}
\lim\limits_{s\to 0}s^{2\alpha_k}\lambda_{\alpha_k}(s)f_2'=0,\qquad 
\lim\limits_{s\to 0}s^{-(9\alpha_k-1)}\lambda_{\alpha_k}(s)f_4'=0.
\end{equation}
Since $1/5<\alpha_k$, the first equation of~\eqref{deux} gives that the first alternative in~\eqref{duecasi} holds. In particular $p:=[1,0,0,0]$ is a smooth point of the quadric $V(f_2+f'_2)$, local coordinates on $Q'$ centered at $p$ are
\begin{equation*}
x:=x_3|{Q'},\quad y:=x_1|{Q'}.
\end{equation*}
Moreover $x_2|_{Q'}=\varphi$, where $\varphi$ is an analytic function such that $\varphi\equiv bx\pmod{\gm_p^2}$ for some $b\in \CC$  (see~\eqref{aneloc}). 
Thus a local equation of $C\subset Q'$ centered at $p$ is given by
\begin{equation*}
f_4(1,y,\varphi(x,y),x)+f'_4(1,y,\varphi(x,y),x)=0.
\end{equation*}
Let us prove that the singularity of $C$ at $p$ is as in  the  row of Table~\ref{table101} corresponding to $k$. 
 
We assign weights to $x$ and $y$ as follows:
\begin{equation}\label{bilancia}
\weight_k(x):=\frac{1}{3},\qquad \weight_k(y):=\frac{1-\alpha_k}{3(3\alpha_k+1)}.
\end{equation}
Notice  the following:
\begin{enumerate}
\item
The weights of $x$ and $y$ in~\eqref{bilancia} are equal to the weights of $x$ and $y$ in Table~\ref{tavola} corresponding to the singularity type listed in Table~\ref{table101} (on the row with index $k$) - this is a straightforward computation. (Warning: the index $k$  in Table~\ref{tavola} has no relation to the index $k$ in Table~\ref{table101}.)
\item
 $\varphi(x,y)=y^2+\psi(x,y)$  (see~\eqref{aneloc}) where all monomials appearing in $\psi(x,y)$ have weight greater than $\weight_{k}(y^2)$ - this because $2\weight_{k}(y)<\weight_{k}(x)$. 
\end{enumerate}
One easily checks that $f_4(1,y,\varphi(x,y),x)=h(x,y)+g(x,y)$ where $h(x,y)$ is homogeneous of weight $1$ and all monomials appearing in $g(x,y)$ have weight strictly greater than $1$, and moreover $h(x,y)$ is equal to the leading term appearing in Table~\ref{tavola} on the row with index $k$ (with the exception of the case $k=3$ in Table~\ref{table101}, where it is equal to the leading (and \lq\lq unique\rq\rq) term of the $J_{3,\infty}$ singularity). 

Thus, by~\Ref{thm}{recognize},  in order to prove~\eqref{gallinari}, it suffices to check that all monomials appearing in 
$f'_4(1,y,\varphi(x,y),x)$ have weight strictly greater than $1$. Every   such  monomial is obtained from a  monomial  $x_0^{i_0}x_1^{i_1}x_2^{i_2}x_3^{i_3}$ appearing in $f'_4$ by setting $x_0=1$, $x_1=y$, $x_2=\varphi(x,y)$ and $x_3=x$. By item~(2) above it suffices to check that the weight of $x^{i_3}y^{i_1+2i_2}$ is strictly greater than $1$, i.e.~that 
\begin{equation*}
(i_1+2i_2)(1-\alpha_k)+i_3 (3\alpha_k+1)>3(3\alpha_k+1).
\end{equation*}
The above inequality follows from the second equation in~\eqref{deux}. 

\medskip

\noindent{\it $\lambda^{-1}_{\alpha_k}$-basin of attraction of $C^{*}_k$:\/} We will prove that 
\begin{equation}\label{baffetto}
\text{if $C$ is in the $\lambda^{-1}_{\alpha_k}$-basin of attraction of $C^{*}_k$, then Item~(2) holds.}
\end{equation}
First $C$ is in the 
$\lambda^{-1}_{\alpha_k}$-basin of $C^{*}_k$ if and only if 
\begin{equation}\label{dos}
\lim\limits_{s\to 0}s^{-2\alpha_k}\lambda_{\alpha_k}(s^{-1})f_2'=0,\qquad 
\lim\limits_{s\to 0}s^{9\alpha_k-1}\lambda_{\alpha_k}(s^{-1})f_4'=0.
\end{equation}
Since $1/5<\alpha_k$ it follows from~\eqref{altridue} that there exist $c,d\in\CC$ such that 
\begin{equation}\label{gradodue}
f_2+f'_2=x_0(cx_0+dx_1+x_2)+x_1^2.
\end{equation}
Thus $V(f_2+f'_2)$ has rank $3$, and its singular point is $v=[0,0,0,1]$. In order to examine the consequences of the second equation in~\eqref{dos}, we introduce some notation. Given a $1$ PS $\lambda(s)=\diag(s^{r_0},\ldots,s^{r_3})$
we define $\weight_{\lambda}(g)$ for $0\not=g\in\CC[x_0,\ldots,x_3]$ as
\begin{equation}
\weight_{\lambda}(g):=\min\{\weight_{\lambda}(\text{monomial in $g$})\},
\end{equation}
where $\weight_{\lambda}$ of a monomial is given by~\eqref{pesomon}. Then the second equation in~\eqref{dos} is equivalent to 
\begin{equation}\label{almeno}
\weight_{\lambda_{\alpha_k}}(f_4')>1-9\alpha_k.
\end{equation}
Inequality~\eqref{almeno} gives that $v$ is a singular point of $V(f'_4)$. Since $v$ is a smooth point of $V(f_4)$, with tangent plane $V(x_0)$, it follows that  $v$ is a smooth point of $V(f_4+f'_4)$, with tangent plane $V(x_0)$. 
Now set $x_3=1$, and hence $(x_0,x_1,x_2)$ are local coordinates in a neighborhood of $v$ in $\PP^3$; by the Implicit Function Theorem, the restrictions of $x_1,x_2$ to  in $V(f_4+f'_4)$ are local coordinates in a neighborhood of $v$ in $V(f_4+f'_4)$. Abusing notation, we use the same symbol for $x_1,x_2$ and their restrictions. There exists an analytic function $\varphi$ of two variables defined in a neighborhood of  $(0,0)$ such that $x_0=\varphi(x_1,x_2)$ \emph{on $V(f_4+f'_4)$}. By~\eqref{gradodue}, a local equation of the plane singularity $(C,v)$ is given by
\begin{equation}\label{eqloc}
x_1^2+\varphi(x_1,x_2)(c\varphi(x_1,x_2)+dx_1+x_2)=0.
\end{equation}
Define 
\begin{equation}
\weight_m(x_1):=1/2,\quad \weight_m(x_2):=1/(m+1),
\end{equation}
 and extend it to a weight function on non zero elements of $\CC[[x_1,x_2]]$ by defining $\weight_m(g)$ as the minimimum of weights of monomials appearing in $g$. In order to prove~\eqref{baffetto} it suffices to check that
\begin{equation}
\weight_m(\varphi)\ge\frac{m}{m+1},
\end{equation}
and the extra condition involving a line $L\subset V(f_2+f'_2)$ if $k\in\{5,6,7\}$.

Since $f_4(\varphi(x_1,x_2),x_1,x_2,1)+f'_4(\varphi(x_1,x_2),x_1,x_2,1)=0$, we get that 
\begin{equation}
\weight_m(\varphi)=\weight_m(f_4(0,x_1,x_2,1)+f'_4(0,x_1,x_2,1)).
\end{equation}
A straighforward computation shows that $\weight_m(f_4(0,x_1,x_2,1))= m/(m+1)$. Thus it suffices to check that
\begin{equation}
\weight_m(f'_4(0,x_1,x_2,1))\ge\frac{m}{m+1}.
\end{equation}
This follows from the values in Table~\eqref{trerighe} - notice that it suffices to consider monomials $x_3^{i_3}x_2^{i_2}x_1^{i_1}$ with $i_1\in\{0,1\}$, because if $i_1\ge 2$  the  associated weight  is at least $1$. (It helps to notice that $\alpha_k$'s are decreasing.)
\begin{equation}\label{trerighe}
\begin{tabular}{c|c|c|c|c|c|c }
monomial  $x^{I}$ &  $x_3^2x_2^2$ & $x_3^2x_2 x_1$ & $x_3 x_2^3$ &  $x_3 x_2^2 x_1$ &     $x_2^4$ &   $x_2^3 x_1$  \\
\hline
$\weight_{\lambda_{\alpha_k}}(x^I)$  & $-2\alpha_k-2$ & $-3\alpha_k-1$  & $3\alpha_k-3$  &  $2\alpha_k-2$  &   $8\alpha_k-4$  & $7\alpha_k-3$  \\
\hline
$\weight_{\lambda_{\alpha_k}}(x^I)>1-9\alpha_k$ iff  & no $k$ & $k=2$  &  $k=2$  & $k\in\{2,\ldots,5\}$  &   $k\in\{2,3,4\}$  & $k\in\{2,\ldots,6\}$  \\
\end{tabular}
\end{equation}
Lastly we prove that,  if $k\in\{5,6,7\}$, the extra condition involving a line $L\subset V(f_2+f'_2)$ holds. First notice that $L:=V(x_0,x_1)$ is a line contained in $V(f_2+f'_2)$.  If $k=5$ the condition $\cC_v C=2T_v(L)$ holds by the local equation of $C$, see~\eqref{eqloc}.  If $k\in\{6,7\}$ the line $L$ belongs to $V(f_4)$ by the computations above (see Table~\ref{leonardo}).

\medskip

Let us finish the proof of the proposition. Suppose that  $t\in (t_{k},t_{k+1})$.  By the analysis above, either the thesis of Item~(1) or the thesis of Item~(2) holds. The former is excluded by~\Ref{prp}{destabilizing}. 

Now suppose that  $t\in (t_{k-1},t_{k})$. By the analysis above, either the thesis of Item~(1) or the thesis of Item~(2) holds. Assume that the latter holds. By the analysis of the 
$\lambda_{\alpha_k}^{\pm}$-basin of attraction of $C^{*}_k$,  it follows that there exist $N_t$ semistable curves $C$, with $t_k<t$, which are in the basin of attraction of $C^{*}_k$ for which the thesis of Item~(1) holds. This contradicts~\Ref{prp}{destabilizing}. Hence  the thesis of Item~(1) holds.
\end{proof}

\subsection{Proof of the remaining main results of the section}\label{subsec:dimodue}
\begin{proof}[Proof of~\Ref{thm}{thmcriticalt}]
Let us prove that the critical values for the VGIT $\gM(t)$ in the interval $(\delta,1/2)\cap\QQ$ are given by the $t_k$'s appearing in~\eqref{valori}, with the exclusion of $t_8=1/2$. By~\Ref{prp}{indentifyt}, every critical value is equal to one of the $t_k$'s.  Hence it remains to prove that $t_k$ is  a critical value, for each $k\in\{0,\ldots,7\}$. There exists an $N_{1/6-\epsilon}$-semistable $(4,4)$ curve $C_k$ on a smooth quadric $Q$,  such that the double cover $X_{C_k}\to Q$ ramified over $C_k$  has a non slc   singularity with tag $k$ and no non slc singularity with tag strictly less than $k$. This is immediate for $k\in\{0,1\}$, and it follows from~\Ref{prp}{4unico} for $k\in\{2,\ldots,7\}$. (The argument would work equally well if we knew that  the curve $C_k^{*}$ in Table~\ref{table101} is $N_{t_k}$-semistable, but at this stage we have not yet proved this.) By~\Ref{prp}{destabilizing} $C_k$ is $N_t$-unstable for $t_k<t$. Thus it will suffice to show that $C_k$ is $N_t$-semistable for $t<t_k$. Suppose the contrary. 
Since $C_k$ is  $N_{1/6-\epsilon}$-semistable, there exists $0\le k_0<k$ such that $C_k$ is in the basin of attraction (form the left) of a curve projectively equivalent to $C_{k_0}^{*}$. By the results of~\Ref{subsec}{sectbasin} this implies that $C_k$ has a point with tag $k_0$, contradicting our choice of $C_k$.  This proves that $C_k$ is $N_t$-semistable for $t<t_k$, and hence proves that $t_k$ is critical value.

It remains to prove that the exceptional loci of 
 $\gM(t_{k-1},t_k)\to \gM(t_k)$ and  $\gM(t_{k},t_{k+1})\to \gM(t_k)$ are as claimed. We have already established the fact that the exceptional loci of $\gM(t_{k-1},t_k)\to \gM(t_k)$ are naturally birational to $W_k$. Specifically, the generic point $\zeta_k$ in $W_k\subset \gM$ is $N_t$ stable for $t<\frac{1}{6}$, it becomes unstable for $t>t_k$ (cf. \Ref{prp}{destabilizing}), but via the basin of attraction argument, $\zeta_k$ can not become unstable before $t_k$. For the exceptional loci of $\gM(t_{k},t_{k+1})\to \gM(t_k)$, we note first that the cases $k=0,1$ are discussed in detail in \Ref{prp}{unsesto} and \Ref{prp}{unquarto}. For the cases $k=2,4$, the minimal orbit $C_k^*$ at $t_k$ has a singularity of type $A_2$ and respectively $A_3$ at the vertex $v$ of the quadric cone containing $C_k^*$. For the cases $k=5,6,7$, there is a singularity at $v$ of type $A_4$, $A_5$, and $A_7$ respectively,  and additionally a line $L$ in special position with respect to the curve $C_k^*$ and the singularity at $v$ (see Table \ref{table101}).  As previously discussed, the exceptional locus $\Sigma_+(t_k)$ of $\gM(t_{k},t_{k+1})\to \gM(t_k)$ is obtained via the basin of attraction of $C_k^*$. By the arguments given in the previous subsection, it follows that the generic point $\xi_k$ of $\Sigma_+(t_k)$
will correspond to a curve having the same type of singularity at $v$ (and position of $L$) as $C_k^*$ (see also \Ref{subsec}{e14case} for further details in one of the cases). Furthermore, this curve will have at worst some additional nodes (imposed by the special position of the line $L$). The resulting conditions are exactly the conditions that have been used to define the loci $Z^{k+1}$ in $\sF^*$ (see \Ref{prp}{hypkgeom}). In other words, there are natural rational maps $\Sigma_+(t_k)\dashrightarrow Z^{k+1}$ (and clearly one-to-one onto the image). To conclude that the two spaces are birational, we note that $Z^{k+1}$ are irreducible and that $\Sigma_+(t_k)$ and $Z^{k+1}$ have the same dimension. As discussed, the index of $Z$ corresponds to the codimension. On the other hand, the dimension of $\Sigma_+(t_k)$ can be computed as being complementary to that of $\Sigma_-(t_k)$, or equivalently $W_k$. (See \Ref{rmk}{flip13} for a discussion of the dimensions in the special case $t=\frac{1}{3}$.)
  \end{proof}

\medskip

\begin{proof}[Proof of~\Ref{prp}{cod2}]
Let $U_{0}\subset U$ be the (open) subset parametrizing curves $C=V(f_2,f_4)$ such that the associated double cover  is a $K3$ surface with canonical singularities, and  such that if $\rk f_2=3$ (by the first condition $\rk f_2\in\{4,3\}$), then  
 $C$ does not contain the vertex of $V(f_2)$. We claim that
\begin{equation}\label{furbata}
\text{If  $t\in\left(\frac{1}{6},\frac{1}{2}\right)\cap\QQ$, then   every curve in $U_{0}$ is  $N_t$ semistable.}
\end{equation}
In fact, since  for $t\in \left(2/5,1/2\right)\cap\QQ$ the period map $\Phi\colon\gM\to\sF^{*}$ is regular and surjective, \eqref{furbata} holds for  $t\in \left(2/5,1/2\right)$. Next, let $C\in U_0$. Then,  by the results of~\Ref{subsec}{sectbasin}, $C$ is not in the basin of attraction (from the right) of any critical value $t_k$ with $k\in\{1,\ldots,7\}$. It follows that $C$ remains semistable for all $t\in(1/6,1/2)\cap\QQ$. We have proved~\eqref{furbata}.

Let
\begin{equation}
\gM(t)_{0}:=U_{0}\gquot_{N_t} \SL(4).
\end{equation}
Thus $\gM(t)_{0}$ is an open subset of $\gM(t)$. A dimension count shows that 
\begin{equation}\label{cod2}
\cod(\gM(t)\setminus\gM(t)_{0},\gM(t))\ge 2.
\end{equation}
Since  $\gM(t)_{0}$ is contained in the regular locus of the period map $\gp(t)$,  all that remains to prove is that the complement of $\gp(t)(\gM(t)_{0})$ in $\sF^{*}$ has codimension at least $2$. By~\Ref{prp}{hypkgeom}, the  complement of $\gp(t)(\gM(t)_{0})$ in $\sF$ is equal to $Z^2$, which has codimension $2$. Since the boundary $\sF^{*}\setminus\sF$ has codimension $17$, we are done.
 \end{proof}

 \subsection{An example - The local structure of the flip at $t=\frac{5}{14}$ (the $E_{14}$ case)}\label{subsec:e14case}
We discuss here in greater detail the structure of the flip occurring at $5/14$.  Let
 $$C=V(x_0x_2+x_1^2,x_2^4+x_0x_3^3)$$
 be a curve in the minimal orbit where semistability changes for $t=5/14$. 
 Thus $C$ has a singularity of type $E_{14}$ at $p=[1:0:0:0]$ and one of type  $A_4$ at the vertex $v$ of the quadric cone $Q=V(x_0x_2+x_1^2)$. 
 
  As expected from the general VGIT theory in this setup, the two fibers $\Sigma_{-}(5/14)$  and $\Sigma_{+}(5/14)$ are weighted projective spaces of complementary dimensions. Below, we will identify them explicitly, and see the connection to the deformation of the $E_{14}$ singularities (and Pinkham's theory of deformations of singularities with $\mathbb G_m$ action).

  The relevant $1$-PS $\lambda$ in this example is $(17,5,-7,-15)$. We are interested in computing the induced $\bC^*$ action on $T_{C}U$. We can choose an affine slice $\mathbf A$ near $C$
 such that the curves parametrized by $\mathbf A$ are given by equations
 $$V(f_2+g_2,f_4+g_4)$$
 where
  \begin{eqnarray*}
  g_4&=&a_1x_0^4+\dots+a_{24}x_1x_3^3\\
  g_2&=&b_1x_0^2+\dots+b_9x_2^2
  \end{eqnarray*}
 given by sums of all  monomials that do not contain $x_1^2$, and excluding $x_2^4$ from the degree $4$ monomials. Simply, we require that the coefficient of $x_1^2$ in $f_2+g_2$ and the coefficient of $x_2^4$ in $f_4+g_4$ are both $1$ (thus we fix affine coordinates). Also, the equation of the quartic can be modified by a multiple of the quadric - and thus the monomials containing $x_1^2$ (i.e. $x_1^2 q(x_0,x_1,x_2,x_3)$) can be cancelled by an appropriate multiple of the quadric (i.e. $(f_4+g_4)-(f_2+g_2)q$ will contain no monomial divisible by $x_1^2$). 
 
 To understand the action induced by $\lambda$ on the affine slice $\mathbf A$, note
 \begin{eqnarray}
 s\cdot(f_2+g_2,f_4+g_4)&\to&\left(\lambda(s)\cdot(f_2+g_2),\lambda(s)\cdot(f_4+g_4)\right) \label{actionslice}\\
 &=&\left(s^{-10}\lambda(s)\cdot(f_2+g_2),s^{28}\lambda(s)\cdot(f_4+g_4)\right)\nonumber\\
 &=&(f_2+s^{-10}\lambda(s)g_2, f_4+s^{28}\lambda(s)g_4)\nonumber
 \end{eqnarray}
 (i.e. we rescale the equations to preserve the coefficient $1$ for $x_1^2$ and $x_2^4$ respectively).  A straight-forward computation gives the weights: 
 \begin{itemize}
 \item[(1)] directions perturbing $f_4$ (coefficients $a_i$): $96$ (the weight corresponding to $x_0^4$,  i.e. $4\cdot 17+28=96$, where $17$ is the weight of the $\lambda$ on the $x_0$ coordinate), $84$, $72$, $64$, $60$, $52$, $48$, $40$, $36$, $32$, $28$, $24$, $20$, $16$, $12$, $8$, $4$, $0$, $-4$, $-8$, $-12$, $-16$, $-24$, $-32$.
 \item[(2)] directions perturbing $f_2$ (coefficients $b_i$): : $24$, $12$, $0$, $-8$, $-12$, $-20$, $-24$, $-32$, $-40$.
\end{itemize}
 These $33$ weights are the weights of the $\bC^*$ action induced by $\lambda$ on $T_{C}U$. 
On the other hand, the orbit to $C$ can be identify with $G/\lambda$. The tangent space to $G$ is the Lie algebra $\mathfrak{g}$ and the induced action by $\lambda$ has weights $\lambda_i-\lambda_j$ (N.B. there are three $0$ weights corresponding to the maximal torus in $\mathfrak{sl}(4)$). Explicitly, we have
\begin{itemize}
\item[(3)] The weights of the $\bC^*$ action on the tangent space to the orbit are  $\pm 24$, $\pm 20$, $\pm 12$, $\pm 12$, $\pm 8$, $0$, $0$.
\end{itemize}
In conclusion, we get
\begin{itemize}
\item[(4)] The positive weights on the normal slice to the orbit of $C$ are:  $4$, $16$,$24$, $28$, $36$, $40$, $48$, $52$, $60$, $64$, $72$, $84$, $96$
\item[(5)] The negative weights on the normal slice to the orbit of $C$ are:  $-40$, $-32$, $-24$, $-16$, $-8$, $-4$
\end{itemize}
By Luna's slice theorem, if $C$ represents a critical orbit at $t_5=5/14$, the two exceptional fibers for $\gM(\frac{5}{14}\pm \epsilon)\to \gM(\frac{5}{14})$ over $C$ are $W\bP(1,4,6,7,9,10,12,13,15,16,18,21,24)$ and $W\bP(1,2,4,6,8,10)$ respectively (we just normalized the two sets of weights from (4) and (5) by dividing by the gcd's) -- e.g. see \cite[\S4.2]{VGIT}. 

Also note that the two sets of weights from (4) and (5) correspond to the basins of attraction $A_\lambda(C)$ and $A_{\lambda^{-1}}(C)$ respectively. 

 \begin{itemize}
 \item[(6)] The monomials giving the positive weights from (4) can be chosen to be: $x_0^4$, $x_0^3x_1$, $x_0^3x_2$, $x_0^3x_3$, $x_0^2x_1x_2$, $x_0^2x_1x_3$, $x_0^2x_2^2$, $x_0^2x_2x_3$, $x_0x_1x_2^2$, $x_0x_1x_2x_3$, $x_0x_2^3$, $x_0x_2^2x_3$, $x_1x_2^2x_3$. These monomials give a normal slice in the positive direction, which is equivalent to saying that the basin of attraction $A_{\lambda}(C)$ is 
 given by perturbing the original equations $(f_2,f_4)$ by monomials in the list above. 
  \item[(7)] Similarly, the monomials giving the negative weights from (5): $x_3^2$, $x_3^4$, $x_2x_3^3$,  $x_2^2 x_3^2$, $x_2^3x_3$, $x_1x_2x_3^2$. 
  As before, this gives an explicit description of $A_{\lambda^{-1}}(C)$.
 \end{itemize}

To understand the geometric behavior, we consider first the negative weight deformations:
$$V(x_0x_2+x_1^2+cx_2^2, x_2^4+x_0x_3^3+a_1 x_1x_2x_3^2+a_2x_2^3x_3+a_3x_2^2 x_3^2+a_4x_3^4)$$
In affine coordinates at $p=[1,0,0,0]$, we get
$$x_2=-(x_1^2+c x_3^2)$$
Plugging it into the quartic equation, we get the local equation near $p=[1,0,0,0]$:
$$x_3^3+(x_1^2+cx_3^2)^4+g_4(1,x_1,x_1^2+cx_3^2,x_3)$$
Giving weights $\frac{1}{8}$ for $x_1$ and $\frac{1}{3}$ for $x_3$, we note that the equation becomes
$$x_3^3+x_1^8+\textrm{h.o.t}$$
 
Let us consider the positive directions. It is immediate to see that perturbations of $(f_2,f_4)$ in the positive directions (cf. (4) and (6) above) preserve the $A_4$ singularity at the vertex $v$ of the quadric cone. On the other hand, one gets a versal deformation of the $E_{14}$ singularity. 
We are interested at the behavior at $p=[1,0,0,0]$. Thus, in affine coordinates ($x_0=1)$, we get 
 $$x_2=-x_1^2$$
Writing $x_1=y$, $x_3=x$, $x_2=y^2$, $x_0=1$, we get that the equation of the perturbed quartic 
$$f_4+g_4=x_2^4+x_0x_3^3+g_4$$
becomes
$$x^3+y^8+g_4(1,y,y^2,x)$$ 
Now consider the monomials of positive weight (see (6) above). Plugging-in into equation above, we get indeed the versal deformation of $E_{14}$ in the negative weight directions (i.e. smoothing directions).  By Luna's slice theorem, modulo the $\bC^*$-action, this can be identified with the exceptional fiber of $\gM(\frac{5}{14}+\epsilon)\to \gM(\frac{5}{14})$. By Looijenga \cite{ltriangle2}, this weighted projective space can be also be identified with the moduli of $T_{3,3,4}$-marked $K3$s. This confirms the predictions of \cite[Sect. 6]{log2} (see loc. cit. for further discussion). 
 
 \begin{table}[htb!]
\renewcommand\arraystretch{1.5}
\begin{tabular}{ccccccccccccccc}
\hline
4&16&24&28&36&40&48&52&60&64&72&84&96\\
$x_1x_2^2x_3$&$x_0x_2^2x_3$&$x_0x_2^3$& $x_0x_1x_2x_3$&$x_0x_1x_2^2$&$x_0^2x_2x_3$&$x_0^2x_2^2$&$x_0^2x_1x_3$& $x_0^2x_1x_2$&$x_0^3x_3$& $x_0^3x_2$&$x_0^3x_1$&$x_0^4$\\
$xy^5$&$xy^4$&$y^6$&$xy^3$&$y^5$&$xy^2$&$y^4$&$xy$&$y^3$&$x$&$y^2$&$y$&$1$\\
\hline
 \end{tabular}
 \end{table}

Also recall that the defining equation of an $E_{14}$ singularity is 
$$f(x,y)=x^3+y^8$$
Thus a natural basis for the versal deformation (consider $\bC[x,y]/\langle \partial_x f,\partial_y f\rangle$) is
$$1,y,y^2,y^3,y^4,y^5,y^6,x,xy,xy^2,xy^3,xy^4,xy^5,xy^6$$
which is exactly the third column of the table above.

\section{Proof of the main result}\label{sec:proof}

 In the present section, we prove~\Ref{thm}{mainthm1}. Item~(i) needs no proof. In fact, the identification of $\sF(0)$ with $\Proj$ of the ring of sections of the Hodge bundle is a celebrated result of Baily and Borel, while the equality $\gM=\sF(1)$ is proved in Section 4 of~\cite{log1}.

The proofs of Items~(ii) and~(iii) involve our GIT models $\gM(t)$ for $t\in(1/6,1/2)$.  Let
\begin{equation}\label{betaper}
\gp(t)\colon \gM(t)\dashrightarrow \sF
\end{equation}
be the period map: if $C=V(f_2,f_4)$ represents a generic stable point  $x\in\gM(t)$, then $\gp(t)(x)$ is the period point of the double cover of $V(f_2)$ branched over $C$ (a hyperelliptic quartic $K3$ surface). The key ingredients in our proofs of Items~(ii)-(iii) are~\Ref{prp}{cod2} and results about the Picard groups of $\sF$ and $\gM(t)$ that we discuss in the following subsection.

\subsection{Divisor Classes on the locally symmetric and GIT models}
The locally symmetric variety $\sF=\cD/\Gamma$ is a $\bQ$-factorial quasi-projective variety. 
 Let $H_n$ and $H_h$ be the nodal and hyperelliptic divisors of $\sF$, respectively. (See Definition 1.3.4 in~\cite{log1}.) Informally, $H_n$ is the closure of the locus of periods of hyperelliptic quartic $K3$'s which are double covers $X\to V(f_2)$ ramified over a $(2,4)$ curve $V(f_2,f_4)$ which is smooth except for a node at a smooth point of $V(f_2)$, while $H_h$ is the locus of periods of hyperelliptic quartic $K3$'s which are double covers $X\to V(f_2)$ of a quadric cone ramified over a $(2,4)$ curve with ADE singularities. We recall that $\lambda$ is the Hodge (or automorphic) $\QQ$ divisor class.

If $Z$ is an algebraic variety, we let $\Pic(Z)_{\QQ}:=\Pic(Z)\otimes_{\ZZ}\QQ$. A $\QQ$-Cartier divisor $D$ on $Z$ determines an element $[D]$ of $\Pic(Z)_{\QQ}$.
\begin{proposition}[{cf.~\cite[Sect. 3]{log1}}]\label{prp:borchrel}
Let $\sF$ be the period space of hyperelliptic polarized $K3$'s of degree $4$. Then
\begin{enumerate}
\item $\Pic(\sF)_\bQ=\QQ [H_n]\oplus \QQ[H_h]$, and
\item $136 \lambda\equiv H_n+16 H_h$.
\end{enumerate}
\end{proposition}
Next, we discuss the Picard group of $\gM(t)$ for $t\in(1/6,1/2)$. Let $\cP^{ss}(N_t) \subset\cP$ be the locus of $N_t$-semistable points. Let $U$ be the parameter space for $(2,4)$ complete  intersection curves in $\PP^3$. Then $\cP^{ss}(N_t)\subset U$ by~\Ref{prp}{doppiovu}. Since $U\subset \PP E$, it makes sense to restrict $\eta$ and $\xi$ (see~\eqref{genero}) to $\cP^{ss}(N_t)$.
\begin{proposition}\label{prp:discesa}
Keeping notation as above, the following hold:
\begin{enumerate}
\item Let $t\in(1/6,1/2)\cap\QQ$, and assume that $t$ is \emph{not} one of the critical slopes for the VGIT $\gM(t)$ (see~\Ref{thm}{thmcriticalt}). Then both 
$\eta|_{\cP^{ss}(N_t)}$ and $\xi|_{\cP^{ss}(N_t)}$ descend to $\QQ$-Cartier divisor classes $\ov{\eta}(t)$ and $\ov{\xi}(t)$ on $\gM(t)$, and   $\Pic(\gM(t))_{\QQ}=\QQ\ov{\eta}(t)\oplus \QQ\ov{\xi}(t)$.
\item Let  $t_k\in[1/6,1/2)\cap\QQ$ be one of the critical slopes for the VGIT $\gM(t)$. Then the restriction  
$(\eta(t_k)+t_k\xi(t_k))|_{\cP^{ss}(N_{t_k})}$ descends to 
a $\QQ$-Cartier divisor class  on $\gM(t_k)$ which generates $\Pic(\gM(t_k))_{\QQ}$.
\end{enumerate}
\end{proposition}
\begin{proof}
(1):  One checks easily that $\Pic^G(\cP^{ss}(N_t))_{\QQ}=\QQ[\eta|_{\cP^{ss}(N_t)}]\oplus\QQ[\xi|_{\cP^{ss}(N_t)}]$. In order to prove that both 
$\eta|_{\cP^{ss}(N_t)}$ and $\xi|_{\cP^{ss}(N_t)}$ descend to $\QQ$-Cartier divisor classes  on $\gM(t)$ we apply  Theorem 2.3 in~\cite{dreznara}. One has to check that if  $C=V(f_2,f_4)$ is  $N_t$-polystable, then the stabilizer $\Stab(C)$ acts trivially on the fiber of $\eta$ or $\xi$ at $C$. Since $G$ is a linearly reductive group, it suffices to check that any $1$ PS $\lambda$ contained in $\Stab(C)$ acts trivially on the fiber of $\eta$ or $\xi$ at $C$. Now, $t$ is not one of the   critical slopes for the VGIT $\gM(t)$,  hence 
$\mu(f_2,\lambda)=0$ and $\mu(f_4,\lambda)=0$. The fiber of $\eta$ at $C$ is identified with $(\CC f_2)^{\vee}$, and 
the fiber of $\xi$ at $C$ is identified with $(\CC f_4)^{\vee}\pmod{f_2}$; it follows that  $\lambda$  acts trivially both on the fiber of $\eta$ and of  $\xi$ at $C$. (2): One applies  Theorem 2.3 in~\cite{dreznara} and proceeds as in Item~(1). 
\end{proof}
\begin{definition}
For $t\in(1/6,1/2)\cap\QQ$, let $D(t)$ be the $\QQ$-Cartier divisor class on $\gM(t)$ obtained by descent from the divisor class $(\eta(t)+t\xi(t))|_{\cP^{ss}(N_{t})}$.
\end{definition}
Note that if $t$ is not a critical slope for the VGIT $\gM(t)$, then  $D(t)=\ov{\eta}(t)+t\ov{\xi}(t)$, but if $t$ is a critical slope for the VGIT 
$\gM(t)$ such an equality does not make sense (by Item~(2) of~\Ref{prp}{discesa}). 
\begin{remark}\label{rmk:ampio}
Let $t\in(1/6,1/2)\cap\QQ$. Then  $D(t)$ is the ample $\QQ$-Cartier divisor class on $\gM(t)$ obtained by descent from the $\SL(4)$-linearized ample divisor class $N_t$. In fact, this follows from~\Ref{crl}{chowlb}.
\end{remark}

Our next task is to compare $\QQ$-Cartier divisors on $\gM(t)$ with the pull-back via $\gp(t)$ of $\QQ$-Cartier divisors on $\sF$. Let 
\begin{equation}\label{perzero}
\gp(t)_0\colon \gM(t)_0\lra \sF
\end{equation}
be the (restriction of the) period map.
 Let  $t\in(1/6,1/2)\cap\QQ$. Since $U_0$ is contained in the locus of $N_t$-stable points, 
  both 
$\eta|_{U_0}$ and $\xi|_{U_0}$ descend to $\QQ$-Cartier divisor classes $\eta(t)_0$ and $\xi(t)_0$ on 
$\gM(t)_0$. 
\begin{lemma}\label{lmm:basta}
Let  $t\in(1/6,1/2)\cap\QQ$. In $\Pic(\gM(t)_0)_\bQ$, we have
\begin{equation*}
\gp(t)_0^{*} H_h = 4 \eta(t)_0,\qquad  \gp(t)_0^{*} H_n = 72 \eta(t)_0 +68 \xi(t)_0.
\end{equation*}
\end{lemma}
\begin{proof}
Let $\wt{\gp}_0\colon U_0\to \sF$ be the period map. It suffices to prove that 
\begin{equation}\label{dadim}
\wt{\gp}^{*}_0 H_h= 4 \eta|_{U_0},\qquad \wt{\gp}^{*}_0 H_n= 72 \eta|_{U_0}+68 \xi|_{U_0}.
\end{equation}
Now
$$\wt{\gp}^{*}_0H_h=\{([f_2],[\bar f_4])\in  U_0\mid V(f_2)  \textrm{ is singular}\},$$
and
$$\wt{\gp}^{*}_0H_n=
\text{closure of $\{([f_2],[\bar f_4])\in  U_0\mid C=V(f_2,f_4) \textrm{is singular at a smooth point of $V(f_2)$}\}$}.$$
 Since the locus of singular quadrics is a degree $4$ hypersurface in $|\cO_{\PP^3(2)}|$, the first equality in~\eqref{dadim} is clear. 
 The second equality in~\eqref{dadim} is proved by a  computation analogous to the one done in~\cite[Prop. 1.1]{g4ball} for $(2,3)$ complete intersections in  $\bP^3$. We omit the details. 
\end{proof}
For $t\not=0$, let $\beta(t)=\frac{1-2t}{4t}$.
\begin{proposition}\label{prp:morise}
In $\Pic(\gM(t))_\bQ$ we have the relation
\begin{equation}\label{enough}
D(t) = 2t \gp(t)^{*}(\lambda+\beta(t) \Delta).
\end{equation}
\end{proposition}
\begin{proof}
By~\eqref{cod2}, it suffices to prove that the restrictions of the left and right hand sides of~\eqref{enough} to $\gM(t)_0$ are equal. Since   $D(t) |_{\gM(t)_0}=\eta(t)_0+t\xi(t)_0$, we must prove that  
\begin{equation}\label{liverpool}
\eta(t)_0+t\xi(t)_0= 2t \gp(t)_0^{*}(\lambda+\beta(t) \Delta).
\end{equation}
Since $\Delta=\frac{1}{2}H_h$, we have $\gp(t)^{*}_0\Delta=2\eta(t)_0$ by~\Ref{lmm}{basta}. On the other hand,   by~\Ref{prp}{borchrel} and by~\Ref{lmm}{basta}, 
$$\gp(t)_0^{*}\lambda=\frac{1}{136}\gp(t)_0^{*}(H_n+16H_h)=\frac{1}{136}(72 \eta(t)_0+68 \xi(t)_0+64\eta(t)_0)=
\eta(t)_0+\frac{1}{2}\xi(t)_0.$$
 Thus, 
$$\gp(t)_0^{*}(\lambda+\beta(t)\Delta)= \eta(t)_0+\frac{1}{2}\xi(t)_0+\left(\frac{1}{2t}-1\right) \eta(t)_0=
\frac{1}{2t}(\eta(t)_0+t\xi(t)_0).$$
This proves~\eqref{liverpool}.
\end{proof}
\subsection{Proof of Items~(ii) and~(iii)}
Let us prove Item~(ii). By Item~(i) there is nothing to prove for $t\in\{0,1\}$. Thus we assume that $t\in(0,1)\cap\QQ$. Let $\sF_0:=\im(\gp(t)_0)$. The (restriction of the) period map defines an isomorphism $\gp(t)_0\colon\gM(t)_0\overset{\sim}{\lra} \sF_0$. Moreover
\begin{equation}
\cod(\gM(t)\setminus \gM(t)_0,\gM(t))\ge 2,\qquad \cod(\sF\setminus \sF_0,\sF)\ge 2.
\end{equation}
In fact, the first inequality is~\eqref{cod2}, the second one has been proved in the proof of~\Ref{prp}{cod2}. By~\Ref{prp}{morise} it follows that the period map induces an isomorphism of rings
\begin{equation}\label{isomorfi}
\gp(t)^{*}\colon R(\sF,\lambda+\beta(t)\Delta)\overset{\sim}{\lra} R(\gM(t),D(t)).
\end{equation}
(Both $\sF$ and $\gM(t)$ are normal varieties.) Since $D(t)$ is ample on $\gM(t)$ (see~\Ref{rmk}{ampio}), the right hand side 
of~\eqref{isomorfi} is a finitely generated $\CC$-algebra. This proves Item~(ii), because $\beta$ defines an invertible function $[1/6,1/2]\cap\QQ\to [0,1]\cap\QQ$. In fact the inverse is given by
\begin{equation*}
t(\beta):=\frac{1}{4\beta+2}.
\end{equation*}
In order to prove Item~(iii) we take the $\Proj$ of both sides of~\eqref{isomorfi}, and we get  
  an isomorphism
\begin{equation}\label{gitisoper}
\gp(t(\beta))^{-1}\colon  \sF(\beta)\overset{\sim}{\lra} \gM(t(\beta)).  
\end{equation}
Thus Item~(iii) follows from~\Ref{thm}{thmcriticalt}, except  that we do not know yet whether $\beta=1$ is a critical value. For this we must show that  for $\epsilon\in\QQ_{+}$ small the period map $\gp(t(\epsilon))\colon \sF(t(\epsilon))\dra \sF^{*}$ is not an isomorphism.  Suppose that it is an isomorphism. Then $\lambda+t(\epsilon)\Delta$ is a $\QQ$-Cartier divisor, and since $\lambda$ is $\QQ$-Cartier, so is $\Delta$. Thus $H_h$ is $\QQ$-Cartier, and this is a contradiction, one knows that $H_h$ is not $\QQ$-Cartier (e.g. it follows from \cite[Cor. 3.5]{looijengacompact}).  
\subsection{Proof of Item~(iv)} By the discussion above, we have $\sF(\epsilon)\cong \gM\left(\frac{1}{2}-\epsilon'\right)\cong \Hilb_{(2,4)}^{\gg 0}\gquot \SL(4)$ (see \Ref{thm}{thmgitmodels}) (with $\frac{1}{2}-\epsilon'=t(\epsilon)$, and $0<\epsilon,\epsilon'\ll 1$). Similarly, $\sF^*\cong \sF(0)\cong \gM\left(\frac{1}{2}\right)\cong \Chow_{(2,4)}\gquot \SL(4)$. Furthermore, with these identifications, $\sF(\epsilon)\to \sF^*$ is compatible with the natural Hilbert-to-Chow map  (see \Ref{rmk}{defgit12}). 

By \Ref{prp}{discesa}, it follows that $\sF(\epsilon)$ is $\bQ$-factorial with Picard number $2$. As already noted, $\sF^*$ is not $\bQ$-factorial with (the closure of) $H_h$ being a Weil divisor, which is not $\bQ$-Cartier. By the GIT description, it is clear that $\sF(\epsilon)\to \sF^*$ is a small map (e.g. \Ref{prp}{cod2}). It follows then that $\sF(\epsilon)$ is isomorphic to the $\bQ$-factorialization of Looijenga associated to the divisor $H_h$ (see \cite[Lemma 6.2]{km}). 

%

\section{The structure of the Chow and (asymptotic) Hilbert GIT Quotients}\label{sec:chow}
As previously discussed, the period map induces an isomorphism 
$\gM(t(\beta))\cong \sF(\beta)$ for $\beta\in[0,1]\cap\QQ$. 
The purpose of this section is to discuss the geometric meaning of this isomorphism for  $\beta$ close to $0$ (or equivalently $t$ close to $\frac{1}{2}$).  Specifically, we are interested in the following diagram:
\begin{equation}\label{diagram12}
\xymatrix{
\gM(\frac{1}{2}-\epsilon)\cong \Hilb_{(2,4)}^{m\gg0}\gquot \SL(4)\ar@{->}[r]^{\ \ \ \ \ \ \ \ \ \ \ \cong}\ar@{->}[d]^{ \Psi}&\widehat\sF\cong \sF(\epsilon)\ar@{->}[d]^{ \Pi}\\ 
\gM(\frac{1}{2}) \cong \Chow_{(2,4)}\gquot \SL(4)\ar@{->}[r]^{\ \ \ \ \ \ \ \ \ \cong}& \sF^*\cong \sF(0)&
}
\end{equation}
where
\begin{itemize}
\item[i)] 
$\sF^*$ is the Baily-Borel compactification of $\sF$, $\widehat \calF$ is  Looijenga's $\bQ$-factorialization of $\sF^{*}$, and $\Pi:\widehat \sF\to \sF^*$ is the structural morphism constructed by Looijenga~\cite{looijengacompact},
\item[ii)] 
$\Hilb_{(2,4)}^{m\gg0}\gquot \SL(4)$ is the GIT quotient of the Hilbert scheme for $(2,4)$ complete intersections (see~\Ref{subsubsec}{mHilbert} and~\Ref{thm}{thmgitmodels}), and similarly $\Chow_{(2,4)}\gquot \SL(4)$ is the Chow quotient. The map $\Psi$ is induced by the Hilbert-to-Chow morphism, 
\item[iii)] the horizontal isomorphisms are those of \eqref{gitisoper} (i.e. induced by the period map). 
\end{itemize}

An immediate consequence of the above identification is the following (which amplifies the statement of \Ref{thm}{improving} for $t\in(2/5,1/2)$). 
\begin{corollary}\label{crl:adestable}
Let $C=V(f_2,f_4)$ be a reduced $(2,4)$ complete intersection such that the associated quadric $Q=V(f_2)$ is either smooth or a quadric cone. Assume that the double cover $S$ of $Q$ branched over $C$ has at worst ADE singularities (i.e. $S$ is a $K3$ with canonical singularities). Then $C$ is Chow GIT stable and asymptotic Hilbert GIT stable. 
\end{corollary}
\begin{proof}
The identification $\Chow_{(2,4)}\gquot \SL(4)\cong \sF^*$ corresponds to (up to passing to Veronese subrings) an identification between the rings of $\SL(4)$-invariant sections on $\Chow_{(2,4)}$ and the ring of automorphic forms on the period domain. If $S$ is $K3$ surface with canonical singularities, by the Torelli theorem and the Baily-Borel theorem, there are enough automorphic forms to separate $S$. Since $S$ determines $C$, and it has finite automorphisms, it follows that $C$ is Chow properly stable. The claim on the Hilbert stability follows from the behavior of VGIT (see \Ref{rmk}{defgit12}).
\end{proof}


\begin{remark}
We emphasize that diagram \eqref{diagram12} is remarkable for (at least) the following reasons:
\begin{itemize}
\item[(1)] It is quite rare to exhibit an isomorphism between GIT quotients and Baily-Borel compactifications. The standard example is that of elliptic curves $(\fh/\SL(2,\bZ))^*\cong |\calO_{\bP^2}(3)|\gquot \SL(3)$. Other examples include some of the Deligne--Mostow examples (weighted points in $\bP^1$), and the moduli of cubic surfaces (cf. Allcock--Carlson--Toledo). However, we are not aware of any example of Baily-Borel compactifications of moduli of $K3$s to be known to be isomorphic to GIT quotients. 
\item[(2)] Even more interesting is the realization of the Looijenga compactification $\widehat \calF$ as a GIT quotient, and of the structural morphism $\Pi$ as a VGIT morphism. 
\end{itemize}
\end{remark}

\subsection{Structure of  Looijenga's $\bQ$-factorization $\widehat \sF$} 
We have already discussed the structure of the Baily-Borel compactification $\sF^*$ in \Ref{subsec}{sectbb}. We recall that while $\sF$ is $\bQ$-factorial, its compactification $\sF^*$ is typically not. For this reason, Looijenga \cite{looijengacompact} has introduced the semitoric compactifications (that offer common generalization of both Baily-Borel and toroidal compactifications) that for appropriate choices give  the $\bQ$-factorializations of the closures of Heegner divisors in $\sF$. As already used elsewhere in the paper, we denote by $\widehat \sF$ the $\bQ$-factorialization associated to the divisor $\Delta$, and we have $\widehat \sF\cong \sF(\epsilon)$.

By definition, $\widehat \sF\to \sF^*$ is a small map, which is an isomorphism over $\sF$ (recall $\sF$ is $\bQ$-factorial). One of the main points of Looijenga \cite{looijengacompact} is that  the structure of $\widehat \calF\to \calF^*$ is purely arithmetic, reflecting how the hyperplanes associated to the divisor $\Delta$ intersect at the boundary. 
In \cite[Sect. 7]{log2}, we have analyzed the $\bQ$-factorialization $\widehat{\sF}(19)\to \sF(19)^*$ for quartic surfaces. In particular, we have determined the dimensions of the Type II boundary components (\cite[Prop. 7.6]{log2}), and the geometric meaning (see \cite[Def. 7.7]{log2} and \cite[Prop. 7.11]{log2}). The same arguments as in loc. cit. give the following result:
\begin{proposition}\label{prp:structureqfact}
Let $\widehat \sF\to \sF^*$ be the Looijenga $\bQ$-factorialization associated to the hyperelliptic divisor. Then
\begin{itemize}
\item[i)] The dimensions of the pre-images in $\widehat \sF$ of the eight Type II components in $\sF^*$ are as given in Table \ref{tabletype2}.
\item[ii)] With the exception of the component labeled $D_{16}$, the remaining $7$ Type II boundary components in $\widehat \sF$ are naturally birational to $7$ Type II boundary components in the GIT quotient for $(4,4)$ curves in $\bP^1\times \bP^1$ (see \Ref{prp}{TypeIIShah} - the labeling are chosen compatibly). The geometric meaning is given in Table \ref{tabletype2}. 
\end{itemize}
\end{proposition}
\begin{table}[htb!]
\renewcommand{\arraystretch}{1.50}
 \begin{tabular}{l|c|c|l||l}
 Label&Type&Dim. in $\widehat \sF$&Geometric Meaning&Quartic Case\\
 \hline
 \hline
$D_{16}$&a&1&double twisted cubic  (on quadric cone)&$D_{17}$\\
 \hline
$D_8\oplus E_8$&a&$9$&$\widetilde E_8$, double line&$D_9\oplus E_8$\\
 $D_{12}\oplus D_4$&a&$5$&double conic&$D_{12}\oplus D_5$\\
$(E_7)^2\oplus D_2$&a&$3$&two $\widetilde E_7$ singularities& $(E_7)^2\oplus D_3$\\
$A_{15}\oplus D_1$&a&$2$& double elliptic quartic& $A_{15}\oplus D_2$ \\
$(D_8)^2$&a&1& two skew double lines;  smooth quadric&  $(D_8)^2\oplus D_1$\\
  \hdashline
$(E_8)^2$&b&$1$&two $\widetilde E_8$ singularities; smooth quadric &$(E_8)^2\oplus D_1$\\
$(D_{16})^+$&b&$1$& double twisted cubic; smooth quadric& $D_{16}\oplus D_1$\\
\hline
 \end{tabular}
 \vspace{0.2cm}
\caption{The Type II boundary components of $\widehat \sF\to \sF^*$}\label{tabletype2}
 \end{table}

\begin{proof}
The first item has the same proof as \cite[Proposition 7.6]{log2}. The only difference is that one loses one dimension (one replaces $D_7$ for quartics, by $D_8$ for hyperelliptic quartic $K3$s). The Type a or b corresponds to the case that one has a primitive or not embedding of $D_8$ into the corresponding Niemeier lattice (see Table \ref{tabledkembed}). 

For the second item, the geometric meaning is obtained via the heuristic of Friedman \cite{friedmanannals} (see \cite[Def. 7.7]{log2}) and then one proceeds as in the proof of \cite[Prop. 7.11]{log2}). We also recall that the GIT quotient $\gM$ for $(4,4)$ curves can be identified with the exceptional divisor in the Kirwan blow-up of the point $\omega$ corresponding to the double quadric. It is not hard to see that a Type II stratum in the GIT for quartic surfaces which contains in its closure the double quadric (i.e. all strata except for the stratum corresponding to the cone over a cubic curve union a transversal plane) will restrict to a Type II stratum (with one less dimension) in the GIT for $(4,4)$ curves on $\bP^1\times\bP^1$. Thus, the matching of GIT and Baily-Borel Type II components for hyperelliptic quartic is in fact a consequence of the matching that we obtained in \cite{log2} for quartics.  
\end{proof}

\subsection{Arithmetic dictates the structure of the Chow and Hilbert quotients}\label{subsec:chowhilbstruct} In conclusion, the isomorphisms of \eqref{diagram12}, the structure of Baily-Borel compactification (\Ref{thm}{thmbbstrata}), and the structure of the $\bQ$-factorialization (\Ref{prp}{structureqfact}), give the following information on the structure of the Chow GIT quotient $\Chow_{(2,4)}\gquot \SL(4)$ and asymptotic Hilbert quotients $\Hilb_{(2,4)}^{\gg 0}\gquot \SL(4)$:
 \begin{enumerate}
 \item There are $8$ one-dimensional Type II strata in the Chow GIT $\Chow_{(2,4)}\gquot \SL(4)$. Additionally, there are two Type III points in $\Chow_{(2,4)}\gquot \SL(4)$. The union of the Type II and III boundary strata is the complement of the ADE locus (as in \Ref{crl}{adestable}) in $\Chow_{(2,4)}\gquot \SL(4)$.
 \item In the Hilbert GIT $\Hilb_{(2,4)}^{m}\gquot \SL(4)$ ($m\gg0$), there are  $8$ Type II boundary components of dimensions between $1$ and $9$ according to the 
  third column of Table \ref{tabletype2}. 
 \item Seven of the $8$ Type boundary II components in $\Hilb_{(2,4)}^{m}\gquot \SL(4)$ are birational to the seven Type II boundary components of $\gM$ identified by 
 \Ref{prp}{TypeIIShah}. More precisely, the VGIT $\gM(t)$ for $t\in (\delta,1/2)$  affects these seven components only birationally. In particular, they have the same dimension in $\gM$ as in $\Hilb_{(2,4)}^{m}\gquot \SL(4)$ (given by  Table \ref{tabletype2}). 
  \item Finally, the eighth Type II component (label $D_{16}$) only exists in the range $t\in (1/3,1/2]$.  (It appears in the exceptional locus of the flip at $t=\frac{1}{3}$, when the locus $W_4\subset \gM$ is replaced by the stratum $Z^4\subset \sF$; see \Ref{rmk}{flip13}).
 \end{enumerate}
 
The above results are much more involved and subtle than the existing comparable literature (\cite{shah}, \cite{looijengacompact}, \cite{lcubic}, \cite{cubic4fold}). Especially, item (4) is completely new, and it offers an elegant explanation to an apparent contradiction to the Shah/Looijenga study of GIT versus Baily-Borel for quartic surfaces (i.e. that the number of Type II components in the two natural compactifications does not agree).

\subsection{The GIT analysis of the Chow and asymptotic Hilbert quotients}\label{GITtype2} The following result is the geometric counterpart of the discussion from  \Ref{subsec}{chowhilbstruct}. The essential new aspect here is the geometric explanation for the drop in dimensions for Type II components as we pass from the Hilbert to Chow GIT quotient. Somewhat surprisingly, there are four different geometric behaviors (labeled (A)--(D) in the proof below)  that occur here. 

\begin{theorem}\label{thm:thmchow}
Consider $\gM(\frac{1}{2})=\Chow_{(2,4)}\gquot \SL(4)(\cong \sF^*)$. Then the following hold:
\begin{itemize}
\item[a)] Let $C$ be $(2,4)$ complete intersection with only planar singularities of type ADE (on an irreducible quadric). Then $C$ is GIT stable (wrt the Chow polarization). In fact, $C$ is $t$-stable for $t\in(\frac{2}{5},\frac{1}{2}]$. Consequently, we can view $\sF$ as an open subset in $\gM(\frac{1}{2})$.
\item[b)] The boundary of $\sF$ in the Chow GIT $\Chow_{(2,4)}\gquot \SL(4)$ is the union of $8$ rational curves, meeting as in diagram \eqref{incidencebb}. 
\end{itemize}
Furthermore, 
\begin{itemize}
\item[(II)] The equations for the polystable curves parametrized by the $8$ Type II boundary components are given in Table \ref{tabletype2git} (the geometric meaning is given by the fifth column of Table \ref{tabletype2}). 
\item[(III)] The polystable orbits corresponding to the two Type III points have equations: 
\begin{eqnarray*}
(III_a) &:& V(x_0x_3,x_1^2x_2^2)\\
(III_b) &:& V(x_0x_3-x_1x_2,x_0x_2^3+2x_1^2x_2^2+x_1^3x_3)
\end{eqnarray*}
\end{itemize}
\end{theorem}

\begin{table}[htb!]
\renewcommand{\arraystretch}{1.50}
 \begin{tabular}{l|l|c}
 Label&Equations&$\dim$ in $\gM(\frac{1}{2}-\epsilon)$\\
 \hline
 \hline
$D_{16}$& $V\left(x_1^2+x_0x_2,(x_3+x_1+ax_2)(x_0x_3^2+2x_1x_2x_3+x_2^2x_3)\right)$&$1$\\
 \hline
$D_8\oplus E_8$&$V\left(x_1^2+x_0x_2, x_0x_3^3+x_1^2x_2^2+a x_1^2x_2x_3\right)$&$9$\\
 $D_{12}\oplus D_4$&$V\left(x_1^2+x_0x_2,x_1(x_0+ax_2)x_3^2\right)$&$5$\\
$(E_7)^2\oplus D_2$& $V\left(x_0x_3,x_1x_2(x_1-x_2)(x_1-ax_2)\right)$&$3$\\
$A_{15}\oplus D_1$&$2E$ (with $E=V(q_1,q_2)$)& $2$\\
$(D_8)^2$&$V(u_0^2u_1^2v_0v_1(v_0-v_1)(v_0-av_1))\subset\bP^1\times \bP^1$& $1$\\
  \hdashline
$(E_8)^2$&$V\left(u_0u_1(u_0v_1^2+u_1v_0^2)(u_0v_1^2+a u_1v_0^2)\right)\subset\bP^1\times \bP^1$&$1$\\
$(D_{16})^+$&$V\left((u_0+u_1)(u_0+a u_1)(u_0v_1^2+u_1v_0^2)^2\right)\subset\bP^1\times \bP^1
$&$1$\\
\hline
 \end{tabular}
 \vspace{0.2cm}
\caption{The Type II boundary components of $\gM(\frac{1}{2})$}\label{tabletype2git}
 \end{table}

\begin{proof}
From \cite[Thm. 4.8]{shah4}, we see that a curve with $ADE$ singularities on a smooth quadric is GIT stable. Such a curve can not become $t$-semistable for any $t\in(0,\frac{1}{2}]$ since if it were the case, $C$ would have in its orbit closure one of the critical orbits identified in \Ref{thm}{thmcriticalt}, but then $C$ would have a worse than ADE singularity (see \Ref{subsec}{sectbasin}). A similar argument applies to curves on the quadric cone (e.g. the ADE curves that do not pass through the vertex of the cone become semistable, and then stable, at $t=\frac{1}{6}$, and can not be subsequently destabilized).  

Some of the Type II GIT orbits in $\gM(t)$ for $t<\frac{1}{2}$ have dimension larger than $1$ (see \Ref{prp}{TypeIIShah}; they will be affected only birationally by the VGIT for $t\in(0,\frac{1}{2})$). They will drop to dimension $1$ (as expected from the isomorphism $\gM(\frac{1}{2})\cong \sF^*$) due to the occurrence of new polystable orbits at $t=\frac{1}{2}$. To identify these new minimal orbits at $t=\frac{1}{2}$, we proceed as in  \Ref{sec}{sec-VGIT}. First, we run our algorithm to identify potential critical orbits at $t=\frac{1}{2}$. We obtain three cases: 
\begin{itemize}
\item[i)] 
\begin{equation}\label{orb12_case1}
V(x_1^2+x_0x_2,x_1^2f_2(x_2,x_3)+x_0f_3(x_2,x_3))
\end{equation}
with stabilizer $\lambda=(5,1,-3,-3)$. This curve has a double line passing through the vertex $v$ of the cone, and an $\widetilde E_8$ singularity at the point $p=[1,0,0,0]$. 
\item[ii)] \begin{equation}\label{orb12_case2}
V(x_0x_3,f_4(x_1,x_2))
\end{equation}
stabilized by $\lambda=(3,1,1,-1)$. In this situation, the quadric becomes reducible, and it is cut out by $4$ planes (that share an axis). 
\item[iii)] \begin{equation}\label{orb12_case3}
V(x_1^2+x_0x_2,q(x_0,x_1,x_2)x_3^2)
\end{equation}
stabilized by $\lambda=(1,1,1,-3)$. In this situation, we have a double conic, together with $4$ lines passing through the vertex $v$. 
\end{itemize}
The same arguments as in \Ref{sec}{sec-VGIT} will show that these potential critical orbits actually occur at $t=\frac{1}{2}$. The effect of the occurrence of these new orbits is to collapse some of the Type II strata in the Hilbert GIT $\Hilb_{(2,4)}^m\gquot \SL(4)$ (see third column of Table \ref{tabletype2git}) to $1$ dimensional strata in $\Chow_{(2,4)}\gquot \SL(4)$. Note also that the basin of attraction argument of \Ref{sec}{sec-VGIT} shows that no Type II strata will collapse in the VGIT $\gM(t)$ for $t<\frac{1}{2}$.

\medskip

In conclusion, we identify the following cases for the behavior of the Type II GIT boundary:

\smallskip

\noindent {\bf Case A}: {\it The $1$-dimensional Type II boundary components in \Ref{prp}{TypeIIShah} (label $(D_8)^2$, $(E_8)^2$, and $(D_{16})^+$)}. The stability in these cases does not changes in the interval $(0,\frac{1}{2}]$ (two of the cases are strictly semistable at all time, while the third one is stable). In all cases, the curves are on the smooth quadric, and only have Type II singularities (isolated singularities of Type $\widetilde E_r$, or double rational curves with $4$ distinct pinch points). In their closure, they will contain either $III_a$ or both $III_a$ and $III_b$. 

\medskip

\noindent {\bf Case B}: {\it The stratum $A_{15}\oplus D_1$ (double elliptic normal curve).} Such curves $C=2C'$ are stable at all time (roughly, GIT will see such a singularity as a double point -- GIT only detects linear geometric features). The difference to the case A is that these dimension of the stratum drops by one via the Hilbert-to-Chow morphism. In the Hilbert scheme, the quadric containing $C$ is recorded (N.B. there is a pencil of quadrics containing $C'$), while in the Chow variety it is not, leading to a $1$-dimensional drop. 

\medskip

\noindent {\bf Case C}: {\it Strata $D_8\oplus E_8$, $D_{12}\oplus D_4$ and $(E_7)^2\oplus D_2$}. In these cases the stability is not affected in the interval $(0,\frac{1}{2})$, but at $t=\frac{1}{2}$, new orbits become semistable and absorb the polystable orbits (of the given $3$ types). The arguments and computations are very similar to those of \Ref{sec}{sec-VGIT}. 

For instance, assume that we are in the situation of the stratum with a single $\widetilde E_8$ and no special line. Then, \Ref{lmm}{norm3} gives the normal form: 
\begin{eqnarray*}
f_2&=&x_0x_2+x_1^2+ax_3^2\\
f_4&=&bx_0x_3^3+x_1^2g_2(x_2,x_3)+x_1g_3(x_2,x_3)+g_4(x_2,x_3)
\end{eqnarray*}
As usual, we are interested in singularity at $p=[1,0,0,0]$. In affine coordinates, $x_2=x_1^2+ax_3^2$, and once we substitute ($x_0=1$, $x_3=v$, $x_1=w$, $x_2=v^2+aw^3$) the leading term of $f_4$ becomes 
$$v^3+w^2g_2(w^2,v)$$ 
If the leading term defines an isolated singularity, we get the singularity $J_{2,0}=\widetilde{E_8}$. If this is not the case, we get $J_{2,p}$, which is the same as $T_{2,3,6+p}$ (a cusp singularity, still insignificant, but of Type III). Here $g_2$ is assumed non-vanishing, otherwise we get  Type IV case discussed previously.  Very similarly to the $E_{12}$ case discussed in \Ref{sec}{sec-VGIT}, the potential critical orbit (case \eqref{orb12_case1} above) is
$$V(x_0x_2+x_1^2, x_0x_3^3+x_1^2g_2(x_2,x_3))$$
with stabilizer 
$$\lambda=(5,1,-3,-3)$$
This can be semistable only at $t=\frac{1}{2}$. For $t>\frac{1}{2}$, we see $\nu^t(x,\lambda)<0$ for all points in the $\widetilde E_8$ stratum. At $t=\frac{1}{2}$, the limit of $x$ in this stratum with respect to $\lambda$ is the orbit with $\bC^*$ as above. 

\smallskip

The case of $2\widetilde E_7$ is similar. The polystable orbits $V(f_2, f_4(x_1,x_2))$ (i.e. a quadric cut by $4$-coaxial planes) will further degenerate to the case of $f_2=x_0x_3$ (as discussed  in~\Ref{lmm}{planarsing}, the reducible quadric case can not be semistable until $t=\frac{1}{2}$). This leads to \eqref{orb12_case2} above.

\smallskip

In the case of double conic, the residual curve will be (in general) an elliptic curve $E$ (Type (2,2) in the smooth quadric case) cutting the double conic in $4$ points. At $t=\frac{1}{2}$ this will degenerate to the curve of arithmetic genus $1$ with a $4$-tuple elliptic point (i.e. $4$ lines passing through the origin in $\bA^3$). This type of degeneration is not allowed until $t=\frac{1}{2}$ (N.B. by \Ref{lmm}{planarsing}, for $t<\frac{1}{2}$ only planar singularities are allowed). This corresponds to \eqref{orb12_case3} above.

\medskip

\noindent {\bf Case D}: {\it This is the stratum $D_{16}$ that is not visible in GIT quotient $\gM$ for $(4,4)$ curves. } Geometrically, the relevant polystable curves sit only on the quadric cone. It consists of a double twisted cubic, together with a residual conic. 
The minimal orbits at $t=\frac{1}{3}$ are 
$$V(x_1^2+x_0x_2,x_0x_3^3+\alpha x_1x_2x_3^2+\beta x_2^3x_3).$$
For generic $\alpha$, there will be a singularity of type $A_3$ at the vertex of the cone $v=[0,0,0,1]$, and a singularity of type $E_{3,0}$ at $p=[1,0,0,0]$. For the special value of $\alpha=1$, we obtain the double twisted cubic. The singularity at $v$ will be of type $A_{\infty}$, while the singularity at $p$ is of type $J_{3,\infty}$. The curves that will have this polystable orbit in their orbit closure at $t=\frac{1}{2}$ will have either a singularity of type $J_{3,k}$ ($k>0$) at $p$ or a singularity of type $A_k$ ($k>3$, and such that the curve doesn't split a line through the vertex). At $t$ increases, we have seen that all $J_{3,k}$ (we allow also $k=0,\infty$) are destabilized, while the curves with $A_k$ singularity (allow also $k=\infty$, but require that it doesn't split a line) at the vertex are allowed to become stable. In practice, we modify the equation $f_4=x_0x_3^3+\alpha x_1x_2x_3^2+\beta x_2^3x_3$ by adding monomials of lower weight with respect to $\lambda=(3,1,-1,-3)$. 
We are interested in preserving the double twisted cubic thus
$$f_4=x_3(x_0x_3^2+2x_1x_2x_3+x_2^2x_3)$$
is modified by moving the conic $V(x_1^2+x_0x_2,x_3)$, i.e. modify the linear form $x_3$ to $\ell(x_0,x_1,x_2,x_3)$. It needs to avoid passing through the vertex, which finally gives the form in the table. 
\end{proof}

\appendix
\section{The Baily-Borel compactification for the $D$ tower}\label{sec:bbcompact}
Any locally symmetric variety $\Gamma\backslash\cD$ has a canonical projective compactification, the Baily-Borel compactification $(\Gamma\backslash\cD)^*=\Proj R(\Gamma\backslash\cD,\cL)$, where $\cL$ is the automorphic (or Hodge) line bundle. In the case of Type IV domains $\cD$, the boundary consists of a union of modular curves, the {\it Type-II} boundary components, and some isolated points, the {\it Type-III} boundary components, which are in the closure of these modular curves. In the present section we will describe  the boundary components of  $\sF(N)^{*}$ for $N\le 20$. The  Baily-Borel boundary components for  $N=19$ (quartic $K3$s)  and $N=20$ (EPW sextics modulo duality) have been described by Scattone~\cite{scattone} and  Camere~\cite{camere} respectively. We extend their analysis to cover the lower dimensional cases (this is necessary for our inductive study). The number of boundary components are listed in Table~\ref{table:tabletype2}. 
  \begin{table}[htb!]
\caption{Number of  boundary components of $\sF(N)$}\label{table:tabletype2}
\renewcommand{\arraystretch}{1.60}
\begin{tabular}{c|c|c|c|c|c|c|c|c|c}
N &    $\le 9$   &10     &11--13  &14   & 15--16 &17&18&19&20    \\
\hline
Type II   &1  &2&2  & 4 & 3   &5   &8   &9  &13\\
\hline
Type III &1&2&1&1&1&1&2&1&1\\\end{tabular}
\vspace{0.2cm}
\end{table}
Within the section we will provide a more detailed description of the boundary components, see~\Ref{prp}{classifyType3} and~\Ref{thm}{bbtype2comp}. 

\subsection{Type II and Type III boundary components in general}\label{subsec:ingenerale}
Let $\Lambda$ be a lattice of signature $(2,m)$, and  $\Gamma< O^{+}(\Lambda)$ be a finite-index subgroup. For $k\in\{1,2\}$, we let $\cI_k(\Lambda)$ be the set of  rank-$k$ saturated isotropic subgroups of $\Lambda$.
 The set of   Type-III boundary components  of $\sF_{\Lambda}(\Gamma)$ is in one-to-one correspondence with  $\Gamma\backslash\cI_1(\Lambda)$, and  the set of   Type-II boundary components  of $\sF_{\Lambda}(\Gamma)$ is in one-to-one correspondence with  $\Gamma\backslash\cI_2(\Lambda)$. Thus our task will be to enumerate $\Gamma_{\xi_N}$-orbits of elements of $\cI_k(\Lambda_N)$, for $k\in\{1,2\}$. 

\subsection{Type III boundary components}
\begin{proposition}\label{prp:classifyType3}
Let $3\le N$. 
\begin{enumerate}
\item 
If $N\not\equiv 2\pmod{8}$,  there is one Type-III boundary components of $\sF(N)$.
\item 
If $N\equiv 2\pmod{8}$,     there are two Type-III boundary components of $\sF(N)$, one corresponds to the $\Gamma_{\xi_N}$-orbit in $\cI_1(\Lambda)$ whose elements are those $L$ such that  
$(L,\Lambda)=\ZZ$ and  the other to the $\Gamma_{\xi_N}$-orbit  whose elements are those $L$ such that    $(L,\Lambda)=2\ZZ$.  
\end{enumerate}
\end{proposition}
\begin{definition}\label{dfn:treatreb}
Let $3\le N$. A Type III boundary component of $\sF(N)$ is of \emph{class a} if it corresponds to the   $\Gamma_{\xi_N}$-orbit of $L\in\cI_1(\Lambda)$ such that  
$(L,\Lambda)=\ZZ$, and it is  of \emph{class b}  if it corresponds to the  $\Gamma_{\xi_N}$-orbit of $L\in\cI_1(\Lambda)$ such that  
$(L,\Lambda)=2\ZZ$. 
\end{definition}
\begin{remark}\label{rem:normalizationbb}
We have an injection $f_{N+1}\colon\sF(N)\hra \sF(N+1)$ with image $H_h(N+1)$. 
The map $f_{N+1}$  extends to a finite map 
\begin{equation}\label{effennestar}
f_{N+1}^{*}\colon\sF(N)^{*}\longrightarrow\sF(N+1)^{*},
\end{equation}
 mapping Type II (Type III) boundary components to Type II (Type III) boundary components. Since the number of Type III boundary components is $1$ if    $N\not\equiv 2\pmod{8}$, and $2$ if $N\equiv 2\pmod{8}$, it follows that if $N\equiv 2\pmod{8}$ then $f_{N+1}^{*}$ is not injective, and  $H_h(N+1)^{*}$ (the closure of $H_h(N+1)$ in $\sF(N+1)^{*}$) is not normal. In fact, $\sF(N)^{*}\longrightarrow H_h(N+1)^{*}$ is the normalization map (since it is birational, finite, with normal source) and the fiber over the unique Type III   boundary point (of class $a$) consists of the two Type III boundary components, one of class $a$, the other of class $b$.
\end{remark}
\subsection{Type II boundary components}
Let $J\in\cI_2(\Lambda_N)$, i.e.~$J$ is a rank-$2$ saturated isotropic subgroup of $\Lambda_N$. The quadratic form on $\Lambda_N$ induces a (negative-definite) quadratic form $q_J$ on $J^{\bot}/J$. If $J_1,J_2\in\cI_2(\Lambda_N)$ are $O(\Lambda_N)$-equivalent, then the corresponding  lattices $(J_r^{\bot}/J_r,q_{J_r})$ are isomorphic. Since e $\Gamma_{\xi_N}< O(\Lambda_N)$, we have a well-defined map 
\begin{equation}\label{mappalfa}
\begin{matrix}
\scriptstyle
\Gamma_{\xi_N}\backslash \cI_2(\Lambda_N) & \scriptstyle \overset{\alpha_N} {\lra} & 
 \scriptstyle\{\text{negative definite lattices of rank $(N-2)$}\}/\text{isom} \\
\scriptstyle \text{$\Gamma_{\xi_N}$-orbit of $J$} & \scriptstyle\mapsto &  \scriptstyle\text{isom.~class   of $J^{\bot}/J$}
\end{matrix}
\end{equation}
We will describe $\Gamma_{\xi_N}\backslash \cI_2(\Lambda_N)$ (i.e.~the set of boundary components of Type II) by describing the image of $\alpha_N$, and by analyzing the fibers of $\alpha_N$.
\begin{proposition}\label{prp:isoexe}
Let $3\le N$, and let  $J\subset\Lambda_N$ be a rank-$2$ saturated isotropic subgroup. Then 
one of the following holds:
\begin{enumerate}
\item 
The map $\Lambda_N\to \Hom(J,\ZZ)$ defined by the quadratic form is surjective, and there exist a 
  sublattice  $H\subset\Lambda_N$  isomorphic to $U\oplus U$, containing $J$, and a decomposition $\Lambda_N=H \oplus L$. Moreover   $J^{\bot}/J$ is  isomorphic to $L$, a lattice in the genus of $D_{N-2}$  (i.e.~of rank $(N-2)$,  even, negative definite, with $(A_L,q_L)\cong (A_{D_{N-2}},q_{D_{N-2}})$).
\item 
The map $\Lambda_N\to  \Hom(J,\ZZ)$ defined by the quadratic form has image a subgroup of index $2$, and there exist a 
  sublattice  $H\subset\Lambda_N$  isomorphic to $U\oplus U(2)$, containing $J$, and a decomposition $\Lambda_N=H \oplus L$. Moreover  $J^{\bot}/J$ is isomorphic to $L$,  a  rank $(N-2)$  even negative definite unimodular lattice.
\end{enumerate}
If $N\not\equiv 2\pmod{8}$ then Item~(2) does not occur.
\end{proposition}

\begin{corollary}\label{crl:isoexe}
Let $3\le N$. The image of  $\alpha_N$ is equal to the set of isomorphism classes of 
\begin{enumerate}
\item
 lattices in the genus of $D_{N-2}$  (i.e.~rank $(N-2)$  even negative definite lattices $L$ with $(A_L,q_L)\cong (A_{D_{N-2}},q_{D_{N-2}})$), if $N\not\equiv 2\pmod{8}$,  
\item
  lattices which are in the genus of $D_{N-2}$, or are  even, negative definite, unimodular, of rank $(N-2)$, if  $N\equiv 2\pmod{8}$.
\end{enumerate}
\end{corollary}

Now we restrict to the case $3\le N\le 20$. In order to describe the image of $\alpha_N$ we must classify the lattices appearing in~\Ref{crl}{isoexe} for $N$ in the chosen range. 
As is well-known, there is only one even, negative definite, unimodular lattice of rank $8$ up to isomorphism, namely  $E_8$, and there are two isomorphism classes of even, negative definite, unimodular lattices of rank $16$, namely $E_8\oplus E_8$, and the unique (up to isomorphism) unimodular overlattice of $D_{16}$, call it $D_{16}^{+}$.  

Before classifying lattices in the genus of $D_n$ for $1\le n\le 18$, we introduce a piece of notation. Given a negative definite lattice $L$, the \emph{root lattice} of $L$ is the  sublattice  $R(L)\subset L$ generated (over $\ZZ$) by \emph{roots} of $L$, i.e.~vectors $v\in L$ such that $v^2=-2$. 
\begin{theorem}\label{thm:classifydn}
Let $1\le n\le 18$. The isomorphism class of a lattice $M$ in the genus  of $D_n$ is determined
 by the isomorphism class of its root sublattice $R(M)$.  
The root sublattices that one gets are the following:
\begin{enumerate}
\item
$n=1$: $\es$.
\item
$2\le n\le 8$: $D_n$.
\item 
$n=9$: $D_9$, $E_8$.
\item 
$10\le n \le 12$: $D_n$, $D_{n-8}\oplus E_8$.
\item 
$n=13$: $D_{13}$, $D_{5}\oplus E_8$, $D_{12}$.
\item 
$n=14$: $D_{14}$, $D_{6}\oplus E_8$, $D_{12}\oplus D_{2}$.
\item 
$n=15$: $D_{15}$, $D_7\oplus E_8$, $D_{12}\oplus D_3$,  $A_{15}$, $(E_7)^2$.
\item $n=16$: $D_{16}$, $D_8\oplus E_8$, $D_{12}\oplus D_4$, $D_2\oplus (E_7)^2$, $(D_8)^2$, $A_{15}$.
\item 
$n=17$: $D_{17}$, $D_9\oplus E_8$,  $D_{12}\oplus D_5$, $D_3\oplus(E_7)^2$, $A_{15}\oplus D_2$,  , $A_{11}\oplus E_6$,  $(D_8)^2$, $D_{16}$, $(E_8)^2$.
\item
 $n=18$: $D_{18}$, $D_{10}\oplus E_8$,  $D_{12}\oplus D_6$, $D_4\oplus (E_7)^2$, $A_{15}\oplus D_3$, $(D_8)^2\oplus D_2$, $D_{16}\oplus D_2$, 
 $D_2\oplus (E_8)^2$, $A_{11}\oplus E_6$, $(D_6)^3$, $(A_9)^2$, $D_{10}\oplus E_7\oplus A_1$, $A_{17}\oplus A_1$.
\end{enumerate}
Moreover, if $R(M)$ has rank strictly smaller than $M$ (and hence $\rk R(M)=\rk M-1$, by the above list), then $M$ is an overlattice of $R(M)\oplus D_1$. 
\end{theorem}

\begin{remark}
The most involved cases, namely $n=17$ and $18$, were previously studied by Scattone \cite[\S6.3]{scattone} and Camere \cite[Prop. 6.21]{camere} (based on Nishiyama \cite[Thm. 3.1]{nishiyama}). The main method (occurring already in \cite{scattone}) for the proof of theorem is to study the embeddings of the $D_k$ into the unimodular rank $24$ lattices (the Niemeier lattices); the essential information is contained in Table \ref{tabledkembed} below.
\end{remark}

\begin{table}[htb!]
\renewcommand{\arraystretch}{1.60}
\begin{tabular}{|c|c|c|l|}
\hline
$R(\text{Niemeier})$& Max Sat. $D_k$& Non-Sat. $D_k$&Min $N$\\
\hline\hline
$D_{24}$ & $D_{23}$ & $D_{24}$& 3 \\
\hline
$\underline{D_{16}}\oplus E_8$ & $D_{15}$ & $D_{16}$ &10 \\
\hline
$(D_{12})^2$ & $D_{12}$ & -- &14 \\
$\underline{D_{10}}\oplus  (E_7)^2$ & $D_{10}$ & -- & 16 \\
$A_{15}\oplus  D_9$ & $D_{9}$ &  -- &17 \\
\hline
$(D_{8})^3$ & $D_{8}$ &  -- &18 \\
$D_{16}\oplus \underline{E_8}$ & $D_{7}$ &  $D_{8}$ & 19 \\
$(E_{8})^3$ & $D_{7}$ & $D_{8}$ & 19 \\
\hline
$A_{11}\oplus  D_7\oplus  E_6$ & $D_{7}$ & -- & 19 \\
$(D_{6})^4$ & $D_{6}$ &  -- &20 \\
$D_6\oplus (A_9)^2$ & $D_{6}$ &  -- & 20 \\
$D_{10}\oplus  \underline{(E_7)^2}$ & $D_{6}$ & -- & 20 \\
$A_{17}\oplus  E_7$ & $D_{6}$ &  -- & 20\\
\hline
\end{tabular}
\vspace{0.2cm}
\caption{Embeddings of $D_k$ ($k\ge 6$) into Niemeier lattices}\label{tabledkembed}
\end{table}

 \begin{remark}
 Many of the lattices in the genus of $D_n$  listed above are obtained as an index $2$ overlattice $M$ of $D_{a}\oplus D_{b}$ (with $a+b=n$). We recall that such an $M$ is determined by an order $2$ isotropic subgroup $H$ of $A_{D_a}\oplus A_{D_b}$ (such that $H^\perp/H\cong A_{D_n}$). We can take  $\alpha_{a}= (\frac{1}{2},\frac{1}{2},\dots, \frac{1}{2})$, $\xi_a=(1,0,\dots,0)$, and $\beta_{a}= (-\frac{1}{2},\frac{1}{2},\dots, \frac{1}{2})$ to be the non-zero elements of $A_{D_a}=(D_a)^*/D_a$ (and similarly for $D_b$).  The choice $H=\langle (\xi_a,\xi_b)\rangle\subset A_{D_a}\oplus A_{D_b}$ is always possible and leads to $D_{n(=a+b)}$. For special choices of $a,b$, there are additional choices for $H$, leading to different lattices in the genus of $D_n$. For instance, if $a=b=8$, we can take $H=\langle (\alpha_a,\alpha_b)\rangle\cong \bZ/2$ which corresponds to adjoining to $D_8\oplus D_8$ the norm $-4$ vector $\frac{1}{2}((1,\dots,1),(1,\dots,1))$.  Similarly, if $a=12$, $b\ge 1$, we can take $H=\langle (\alpha_{12},\xi_b)\rangle$ which corresponds to adjoining to $D_{12}\oplus D_b$ the norm $-4$ vector $\frac{1}{2}((1,\dots, 1), (2,0,\dots,0))$. 
 \end{remark}
In describing the fibers of   $\alpha_N$, we will appeal to the following observation.
\begin{remark}\label{rmk:fibercard}
Let $X$ be  a set, and  $G$   a group acting (on the left) on $X$. Let $H<G$ be a finite-index subgroup, and $\pi\colon H\backslash X\to G\backslash X$  the natural map. Then the  fiber of $\pi$ over the $G$-orbit $[x]_G$ has cardinality given by
\begin{equation*}
|\pi^{-1}([x]_G )|=\frac{[G:H]}{[\Stab_G(x):\Stab_H(x)]},
\end{equation*}
where $\Stab_G(x)$ and $\Stab_H(x)$ are the stabilizers of $x$ in $G$ and $H$ respectively. 
\end{remark}
We will also need the following elementary result.  
\begin{lemma}\label{lmm:minilemma}
\begin{enumerate}
\item
The group  $O(U\oplus U)$ acts transitively on $\cI_2(U\oplus U)$. 
Moreover, given $J\in \cI_2(U\oplus U)$,  there exists $\varphi\in(O(U\oplus U)\setminus O^{+}(U\oplus U))$ which maps $J$ to itself.
\item
The group  $O(U\oplus U(2))$ acts transitively on the set of $J\in\cI_2(U\oplus U(2))$ such that the map $(U\oplus U(2))\to\Hom(J,\ZZ)$ is not surjective. Moreover, given $J$ as above,  there exists $\varphi\in(O(U\oplus U)\setminus O^{+}(U\oplus U))$ which maps $J$ to itself. 
\end{enumerate}
\end{lemma}

\begin{proposition}
Let $3\le N$. If $N\not\equiv 6\pmod{8}$, the map $\alpha_N$ in~\eqref{mappalfa} is injective. The map  $\alpha_{6}$ is injective as well. The fiber of $\alpha_{14}$ over the isomorphism class of $D_{12}$ has cardinality $3$, while the fiber over the isomorphism class  of $D_{4}\oplus E_8$ is a  singleton. 
\end{proposition}

 We conclude:
 \begin{theorem}\label{thm:bbtype2comp}
 Let $3\le N\le 20$. The number of Type II  components of $\sF(N)$ is given in ~\Ref{table}{tabletype2}. More precisely,
 \begin{enumerate}
 \item 
 For $N\not \equiv 2 \pmod 8$,  Type II  components are uniquely labeled by the lattices in the genus of $D_{N-2}$ (as listed in~\Ref{thm}{classifydn}), with the exception of $N=14$, in this case  there are three Type II components labeled by $D_{12}$.  The Type II components are pairwise disjoint, but their closures  all meet in a single point, corresponding to the unique Type III boundary point (cf.~\Ref{prp}{classifyType3}). 
 \item
  If $N\equiv 2 \pmod 8$, there are two kinds of Type II components, those mapped by $\alpha_N$ to the isomorphism class of a lattice in the genus of $D_{N-2}$, and those which are mapped 
   by $\alpha_N$ to the isomorphism class of an even negative definite  unimodular lattice of rank $N-2$; in both cases the isomorphism class of the corresponding lattice uniquely determined the boundary component.  
The Type II components are pairwise disjoint, but    the closure of each component of the first kind contains the Type III point  of  class a  (see~\Ref{dfn}{treatreb}) and not the Type III component of class b, while 
  the closure of each Type II component of the second kind contains both Type III boundary points. 
 \end{enumerate}
 \end{theorem}

Diagram~\eqref{bbpicture}  illustrates the structure of the Baily--Borel boundary for  $N\in \{9,10,11\}$. Type II and Type III boundary components are represented by $\circ$ and $\bullet$ respectively. We recall that a Type III component in the Baily--Borel compactification is just a point, while  Type II components are modular curves $G\backslash\mathfrak h$, where $\mathfrak h$ is the Siegel upper half-space and $G< \SL(2,\QQ)$ is a 
 subgroup commensurable to $\SL(2,\bZ)$  (typically $G=\SL(2,\bZ)$, e.g. this is the case for all Type II components for quartic $K3$s ($N=19$) cf. Scattone  \cite[Fig. 5.5.7, on p. 70]{scattone}). The incidence between Type II and Type III (meaning that a Type III component is contained in the closure of a Type II component) is illustrated by a line. 
\begin{equation}\label{bbpicture}
\xymatrix{
N=9: & &&\bullet^{III_a}\ar@{-}[r]&\circ^{II(D_7)} \\
N=10: & \bullet^{III_b}\ar@{-}[r]&\circ^{II(E_8)}\ar@{-}[r]&\bullet^{III_a}\ar@{-}[r]&\circ^{II(D_8)}\\
 N=11: &&\circ^{II(E_8\oplus D_1)}\ar@{-}[r]&\bullet^{III_a}\ar@{-}[r]&\circ^{II(D_9)}
}
\end{equation}
The appendices $a$ and $b$ for Type III components refer to the notation introduced in~\Ref{dfn}{treatreb}. 

\begin{remark}
Let $f_{N+1}^{*}\colon\sF(N)^{*}\to\sF(N+1)^{*}$ be the map in~\eqref{effennestar}. Then $f_{N+1}^{*}$ maps a Type II boundary component of $\sF(N)^{*}$ to a Type II boundary component of $\sF(N+1)^{*}$, and the inverse image by $f_{N+1}^{*}$ of a Type II boundary component  of $\sF(N+1)^{*}$ is a union of Type II boundary components  of $\sF(N)^{*}$. The corresponding map between Type II boundary components is related to inclusion relations among  the lattices  listed in~\Ref{thm}{classifydn} for $n\in\{N-2,N-1\}$ (we must include also the lattice $E_8$ if $N\in\{9,10\}$, and the lattices $E_8\oplus E_8$, $D_{16}^{+}$  if $N\in\{18,19\}$).
\end{remark}

 \bibliography{refk3}
 \end{document}